\documentclass[11pt]{article}
\usepackage{latexsym,amssymb}
\usepackage{amsmath}
\usepackage[mathscr]{eucal}
\usepackage{xcolor}
\usepackage[abbrev]{amsrefs}
\usepackage[all]{xy}
\usepackage{mathtools}
\usepackage{soul}
\usepackage{enumitem}

\newtheorem{Theorem}{\bf Theorem}[section]
\newtheorem{Lemma}{\bf Lemma}[section]
\newtheorem{Proposition}{\bf Proposition}[section]
\newtheorem{Corollary}{\bf Corollary}[section]
\newtheorem{Remark}{\bf Remark}[section]
\newtheorem{Example}{\bf Example}[section]
\newtheorem{Definition}{\bf Definition}[section]

\newenvironment{theorem}{\begin{Theorem}$\!\!\!$}{\end{Theorem}}
\newenvironment{lemma}{\begin{Lemma}$\!\!\!$}{\end{Lemma}}
\newenvironment{proposition}{\begin{Proposition}$\!\!\!$}{\end{Proposition}}

\newenvironment{remark}{\begin{Remark}$\!\!\!$}{\end{Remark}}

\newenvironment{definition}{\begin{Definition}$\!\!\!$}{\end{Definition}}

\numberwithin{equation}{section}

\numberwithin{equation}{section}

\newcommand{\G}{\mathsf{G}}
\newcommand{\dee}{{\rm{d}}}
\newcommand{\R}{\mathbb{R}}

\newtagform{red}{\color{red}(}{)}
\newtagform{teal}{\color{teal}(}{)}
\newtagform{blue}{\color{blue}(}{)}
\newtagform{magenta}{\color{magenta}(}{)}

\def\XXint#1#2#3{{\setbox0=\hbox{$#1{#2#3}{\int}$}
\vcenter{\hbox{$#2#3$}}\kern-.5\wd0}}

\newcommand{\uHDD}[3]{u_{\rm HDD}^{#1,#2,#3}}
\newcommand{\uHD}[2]{u_{\rm HD}^{#1,#2}}
\newcommand{\uHDi}[2]{u_{{\rm H}_{{\rm D},#2}}^{#1}}
\newcommand{\uHhN}[1]{u_{\rm H_{N,0}}^{#1}}
\newcommand{\uHDP}[2]{u_{\rm H_{D,\Psi}}^{#1,#2}}
\newcommand{\uHDN}[2]{u_{\rm HDN}^{#1,#2}}
\newcommand{\uLDD}[2]{u_{\rm LDD}^{#1,#2}}
\newcommand{\uLD}[1]{u_{\rm LD}^{#1}}
\newcommand{\uLDi}[1]{u_{{\rm L}_{{\rm D},#1}}}
\newcommand{\uLDP}[1]{u_{{\rm L}_{{\rm D},\Psi}}^{#1}}
\newcommand{\GLDD}[2]{U^{#1,#2}}
\newcommand{\GHDN}[2]{V^{#1,#2}}
\newcommand{\HHDN}{\hat{H}}
\newcommand{\HHDiLH}{\tilde{H}}

\allowdisplaybreaks[4]

\begin{document}
\title{Fundamental solution and diffusion limits for the heat equation\\ 
in a half-space with a diffusive dynamical boundary condition}
\author{Kazuhiro Ishige, Sho Katayama, and Tatsuki Kawakami}
\date{}
\maketitle
\begin{abstract}
We derive an explicit representation of the fundamental solution to the heat equation in a half-space of ${\mathbb R}^N$ with a diffusive dynamical boundary condition, 
and establish sharp pointwise upper and lower bounds.
We also investigate qualitative properties of the associated solutions, including precise decay estimates.
Furthermore, we analyze the diffusion limits of solutions to the initial--boundary value problem, 
and reveal the role of the diffusive dynamical boundary condition in the behavior of solutions.
\end{abstract}
\vspace{40pt}
\noindent 
Addresses:

\smallskip
\noindent 
K. I.: Graduate School of Mathematical Sciences, The University of Tokyo,\\ 
3-8-1 Komaba, Meguro-ku, Tokyo 153-8914, Japan\\
\noindent 
E-mail: {\tt ishige@ms.u-tokyo.ac.jp}\\

\smallskip
\noindent 
S. K.: Department of Mathematics, Faculty of Science and Technology, Keio University,\\ 
3-14-1 Hiyoshi, Kohoku-ku, Yokohama, Kanagawa 223-8522, Japan\\
\noindent 
E-mail: {\tt sho-katayama@keio.jp}\\

\smallskip
\noindent 
{T. K.}: Applied Mathematics and Informatics Course,\\ 
Faculty of Advanced Science and Technology, Ryukoku University,\\
1-5 Yokotani, Seta Oe-cho, Otsu, Shiga 520-2194, Japan\\
\noindent 
E-mail: {\tt kawakami@math.ryukoku.ac.jp}\\

\newpage
\section{Introduction}
We consider the heat equation
in the half-space $\Omega={\mathbb R}^{N-1}\times(0,\infty)$
with a diffusive dynamical boundary condition:
\begin{equation}
\tag{HDD}
\label{eq:HDD}
	\left\{
	\begin{array}{ll}
	\epsilon\partial_t u-\Delta u=0\quad & \mbox{in}\quad \Omega\times(0,\infty),\vspace{3pt}\\
	\delta\partial_t u-k\Delta' u-\partial_{x_N}u=0\quad & \mbox{on}\quad\partial\Omega\times(0,\infty),\vspace{3pt}\\
	u=\phi\quad & \mbox{in}\quad\Omega\times\{0\},\vspace{3pt}\\
	u=\psi\quad & \mbox{on}\quad\partial\Omega\times\{0\},
	\end{array}
	\right.
\end{equation}
where $N\ge 2$, $\epsilon>0$, $\delta>0$, $k>0$, 
$\partial_t:=\partial/\partial t$, and $\partial_{x_N}:=\partial/\partial x_N$. 
Here $\Delta$ is the Laplace operator on ${\mathbb R}^N$, 
$\Delta'$ is the Laplace operator on $\partial\Omega$, 
and  
$(\phi,\psi)\in L^\infty(\Omega)\times L^\infty(\partial\Omega)$.

The diffusive dynamical boundary condition given in the second equation in~\eqref{eq:HDD}
arises when the boundary surface has a finite storage capacity,
supports tangential diffusion, and simultaneously exchanges flux with the bulk
(see, e.g., \cites{C75bk, CR90, E95, GM14, GLR23, LSWW} for further details on the physical and historical background).
For the case where $\Omega$ is a bounded smooth domain,
well-developed existence theories for problem~\eqref{eq:HDD} are available
(see, e.g., \cites{BDV01, E93, E95, GM14, Hintermann89} and the references therein).
However, these theories do not extend to unbounded domains and 
are not applicable to the half-space setting considered in this paper.

In this paper, we derive an explicit representation of the fundamental solution to problem~\eqref{eq:HDD}
and establish sharp pointwise upper and lower bounds.
To the best of our knowledge, no previous work has obtained such an explicit and detailed description
of the fundamental solution for problem~\eqref{eq:HDD} in the half-space setting.
These results yield precise decay estimates for solutions and further allow us to study
the diffusion limits with respect to the parameters $\epsilon$, $\delta$, and~$k$,
including optimal convergence rates.
\medskip

We introduce some notation. 
We identify $\partial\Omega$ with ${\mathbb R}^{N-1}$ whenever no confusion arises.
We denote by $\dee\sigma$ the surface measure on $\partial\Omega$.
For $x\in\overline{\Omega}$, we often write $x=(x',x_N)$
with $x'\in{\mathbb R}^{N-1}$ and $x_N\in[0,\infty)$. 
For $x\in\overline{\Omega}$ and $r>0$, we set 
\[
	B^+_r(x):=\{y\in\Omega\,:\, |x-y|<r\}, \quad
	B'_r(x):=\{y\in\partial\Omega\,:\, |x-y|<r\}.
\]
For $(f,g)\in L^r(\Omega)\times L^r(\partial\Omega)$ with $r\in[1,\infty]$, 
we define
\[
	\|(f,g)\|_{L^r(\Omega)\times L^r(\partial\Omega)}
	:={\begin{dcases}
	  \left(\|f\|_{L^r(\Omega)}^r+\|g\|_{L^r(\partial\Omega)}^r\right)^{1/r}
	  &\textrm{if}\quad r\in[1,\infty),\vspace{3pt}\\
      \max\left\{\|f\|_{L^\infty(\Omega)},\|g\|_{L^\infty(\partial\Omega)}\right\}
      &\textrm{if}\quad r=\infty.
	\end{dcases}}
\]
For any continuous function $f$ on $\overline{\Omega}$, 
we denote by $f|_{\partial\Omega}$ the restriction of $f$ to $\partial\Omega$.
For any $d=1,2,\dots$, we denote by $\Gamma_d=\Gamma_d(x,t)$ 
the function on ${\mathbb R}^d\times(0,\infty)$ defined by 
\begin{equation}
\label{eq:1.1}
	\Gamma_d(x,t) := (4\pi t)^{-\frac{d}{2}} \exp\left(-\frac{|x|^2}{4t}\right),
	\quad (x,t)\in {\mathbb R}^d\times(0,\infty),
\end{equation}
which is the fundamental solution to the heat equation in ${\mathbb R}^d$, that is, $\Gamma_d$ satisfies
\begin{equation}
\label{eq:1.2}
	\partial_t \Gamma_d = \Delta \Gamma_d \quad \text{in}\quad {\mathbb R}^d\times(0,\infty),
	\quad
	\Gamma_d(\cdot,0) = \delta_d \quad \text{in}\quad {\mathbb R}^d,
\end{equation}
where $\delta_d$ denotes the Dirac delta distribution in ${\mathbb R}^d$. 
We also denote by $G_0=G_0(x,y,t)$ the fundamental solution to the heat equation in $\Omega$ 
with the homogeneous Dirichlet boundary condition, defined by
\begin{equation}
\label{eq:1.3}
	\begin{split}
	G_0(x,y,t):= & \, \Gamma_N(x-y,t) - \Gamma_N(x-y^*,t)
	\\
	= & \, \Gamma_{N-1}(x'-y',t)\left(\Gamma_1(x_N-y_N,t) - \Gamma_1(x_N+y_N,t)\right)
	\end{split}
\end{equation}
for $(x,y,t)\in D:=\overline{\Omega}\times\overline{\Omega}\times(0,\infty)$,
where $y^*:=(y',-y_N)$ for $y=(y',y_N)\in\overline{\Omega}$. 
\medskip

Before considering problem~\eqref{eq:HDD}, 
we focus on the case $k=0$, that is,
the heat equation in the half-space $\Omega$ with a (non-diffusive) dynamical boundary condition:
\begin{equation}
\tag{HD}
\label{eq:HD}
	\left\{
	\begin{array}{ll}
	\epsilon\partial_t u-\Delta u=0\quad & \mbox{in}\quad {\Omega\times(0,\infty)},
	\vspace{3pt}\\
	\delta\partial_t u-\partial_{x_N}u=0\quad & \mbox{on}\quad{\partial\Omega\times(0,\infty)},
	\vspace{3pt}\\
	u=\phi\quad & \mbox{in}\quad\Omega\times\{0\},
	\vspace{3pt}\\
	u=\psi\quad & \mbox{on}\quad\partial\Omega\times\{0\},
	\end{array}
\right.
\end{equation}
where $(\phi,\psi)\in L^\infty(\Omega)\times L^\infty(\partial\Omega)$. 
In \cite{FIK21}, 
the first and third authors of the present paper, together with Fila, 
constructed a solution~$\uHD{\epsilon}{\delta}$ to problem~\eqref{eq:HD}
and proved that 
\begin{equation}
\label{eq:1.4}
	\lim_{\epsilon\to 0^+}\sup_{(x,t)\in\Omega_L\times I}
	|\uHD{\epsilon}{\delta}(x,t)-\uLD{\delta}(x,t)|=0
\end{equation}
for any $L>0$ and any compact set $I\subset(0,\infty)$, where 
\[
	\Omega_L:=\{x=(x',x_N)\in\overline{\Omega}\,:\,0\le x_N\le L\}.
\]
Here $\uLD{\delta}$ is a function on $\overline{\Omega}\times(0,\infty)$ 
defined by 
\begin{equation}
\label{eq:1.5}
	\uLD{\delta}(x,t)
	:=\int_{\partial\Omega}P\left(x'-y',x_N+\frac{t}{\delta}\right)\psi(y)\,\dee\sigma(y),
	\quad (x,t)\in \overline{\Omega}\times(0,\infty),
\end{equation}
where $P$ is the Poisson kernel on the half-space $\Omega$, given by
\begin{equation}
\label{eq:1.6}
	P(x)=P(x',x_N)=c_N x_N|x|^{-N},\quad x=(x',x_N)\in\Omega,
\end{equation}
with $c_N:=\pi^{-\frac{N}{2}}\Gamma(N/2)$.
The function $\uLD{\delta}$ corresponds to a solution to the Laplace equation
in the half-space $\Omega$ with a (non-diffusive) dynamical boundary condition:
\begin{equation}
\tag{LD}
\label{eq:LD}
	\left\{
	\begin{array}{ll}
	-\Delta u=0\quad & \mbox{in}\quad{\Omega\times(0,\infty)},
	\vspace{3pt}\\
	\delta\partial_t u-\partial_{x_N}u=0\quad & \mbox{on}\quad{\partial\Omega\times(0,\infty)},
	\vspace{3pt}\\
	u=\psi\quad & \mbox{on}\quad\partial\Omega\times\{0\}.
	\end{array}
	\right.
\end{equation}
Subsequently, relation~\eqref{eq:1.4} was improved in \cite{FIKL19} to
\begin{equation}
\label{eq:1.7}
	\sup_{(x,t)\in\Omega_L\times I}|\uHD{\epsilon}{\delta}(x,t)-\uLD{\delta}(x,t)|
	=O\left(\epsilon^{\frac{1}{2}}\right)\quad\mbox{as}\quad\epsilon\to 0^+
\end{equation}
for any $L>0$ and any compact set $I\subset(0,\infty)$. 
The convergence rate in \eqref{eq:1.7} is optimal.
(See \cites{FIK23, FIKL19, FIKL20} for related results.)
In \cites{FIK21, FIK23, FIKL19, FIKL20}, 
the solution $\uHD{\epsilon}{\delta}$ to problem \eqref{eq:HD} 
is expressed as the sum of a solution $v$ of an inhomogeneous heat equation 
with the homogeneous Dirichlet boundary condition
and a solution $w$ of the Laplace equation with an inhomogeneous dynamical boundary condition.
Consequently, even if the initial data $\phi$ and $\psi$ are nonnegative
and we may therefore expect the solution $\uHD{\epsilon}{\delta}$ to be nonnegative,
the signs of the solutions $v$ and $w$ may change, 
which makes it difficult to systematically study the limiting behavior of solutions to problem~\eqref{eq:HD}
under various diffusion regimes.
More recently, in \cite{IKK25}, 
the present authors derived an explicit representation of the fundamental solution to problem~\eqref{eq:HD} 
and established sharp pointwise lower and upper bounds, 
which in turn yield precise decay estimates for solutions to problem~\eqref{eq:HD}.
\medskip

In this paper, we extend the method of \cite{IKK25} to derive an explicit representation of the fundamental solution to problem~\eqref{eq:HDD} 
and obtain sharp pointwise bounds. We also establish precise decay estimates for solutions to problem~\eqref{eq:HDD}
and study their limiting behavior under various diffusion regimes, 
including optimal convergence rates. 
The study of diffusion limits of solutions to problem~\eqref{eq:HDD} reveal the role of 
the diffusive dynamical boundary condition in the behavior of solutions.
As a byproduct, we also derive explicit representations of fundamental solutions to
the following two problems:
\begin{itemize}
\item 
 	the Laplace equation in $\Omega$ with a diffusive dynamical boundary condition: 
	\begin{equation}
	\tag{LDD}
	\label{eq:LDD}
  		\left\{
  		\begin{array}{ll}
  		-\Delta u=0\quad & \mbox{in}\quad {\Omega\times(0,\infty)},
		\vspace{3pt}\\
  		\delta\partial_t u-k\Delta' u-\partial_{x_N}u=0\quad & \mbox{on}\quad{\partial\Omega\times(0,\infty)},
		\vspace{3pt}\\
  		u=\psi\quad & \mbox{on}\quad\partial\Omega\times\{0\};
  		\end{array}
  		\right.
	\end{equation}
\item 
	the heat equation in $\Omega$ with a diffusive Neumann boundary condition:
	\begin{equation}
	\tag{HDN}
	\label{eq:HDN}
  		\left\{
 		\begin{array}{ll}
  		\epsilon\partial_t u-\Delta u=0\quad & \mbox{in}\quad {\Omega\times(0,\infty)},
		\vspace{3pt}\\
  		-k\Delta' u-\partial_{x_N}u=0\quad & \mbox{on}\quad { \partial\Omega\times(0,\infty)},
		\vspace{3pt}\\
  		u=\phi\quad & \mbox{in}\quad\Omega\times\{0\}.
  		\end{array}
  		\right.
	\end{equation}
\end{itemize}
To the best of our knowledge, 
no results are available on explicit representations of fundamental solutions 
to problems~\eqref{eq:HDD}, \eqref{eq:LDD}, and \eqref{eq:HDN}.
\medskip

We now present the explicit representation of the fundamental solution to problem~\eqref{eq:HDD},
together with an integral representation of bounded classical solutions to problem~\eqref{eq:HDD}. 
In Theorem~\ref{Theorem:1.1}, we also treat problem~\eqref{eq:HD}, which corresponds to the case $k=0$.
A classical solution to problem~\eqref{eq:HDD} is a function $u$ 
defined on $\overline{\Omega}\times[0,\infty)$ such that
\[
	u\in C^{2;1}(\overline{\Omega}\times(0,\infty)) \cap C(\Omega\times[0,\infty)),
	\quad
	u|_{\partial\Omega}\in C(\partial\Omega\times[0,\infty)),
\]
and $u$ satisfies the equations in~\eqref{eq:HDD} pointwise.  
Note that we do not assume that $u\in C(\overline{\Omega}\times[0,\infty))$. 
Similarly, a classical solution to problem~\eqref{eq:HD} is a function $u$ 
defined on $\overline{\Omega}\times[0,\infty)$ such that
\[
	u\in {C^{2;1}(\Omega\times(0,\infty))\cap C^{1;1}(\overline{\Omega}\times(0,\infty))
   	\cap C(\Omega\times[0,\infty))},
	\quad
	u|_{\partial\Omega}\in C(\partial\Omega\times[0,\infty)),
\]
and $u$ satisfies the equations in~\eqref{eq:HD} pointwise.
\begin{theorem}
\label{Theorem:1.1}
Let $\epsilon>0$, $\delta>0$, and $k\ge0$.
Define
\begin{align}
\label{eq:1.8}
 	H(x,y,t)
 	& :=-2\int_0^t\Gamma_{N-1}\!\left(x'-y',\frac{t-\tau}{\epsilon}+\frac{k}{\delta}\tau\right)
 	\partial_{x_N}\Gamma_1\!\left(x_N+y_N+\frac{\tau}{\delta},
 	\frac{t-\tau}{\epsilon}\right)\,\dee\tau,
	\\
\label{eq:1.9}
 	G(x,y,t)
 	& :=G_0\!\left(x,y,\frac{t}{\epsilon}\right)+\frac{1}{\delta}H(x,y,t),
	\end{align}
for $(x,y,t)\in D=\overline{\Omega}\times\overline{\Omega}\times(0,\infty)$.
Then the following assertions hold.
\begin{enumerate}[label={\rm(\arabic*)}]
\item
$G(x,y,t)=G(y,x,t)>0$ for $(x,y,t)\in D$.

\item 
For any fixed $y\in\overline{\Omega}$, the function
$G=G(x,y,t)$ satisfies
\[
	\begin{array}{ll}
	\epsilon\partial_t G-\Delta_x G=0
	& \mbox{in}\quad \Omega\times(0,\infty),
	\vspace{5pt}
	\\
	\delta\partial_t G-k\Delta'_x G-\partial_{x_N}G=0
	& \mbox{on}\quad {\partial\Omega\times(0,\infty)},
	\end{array}
\]
with respect to the variables $(x,t)\in\overline{\Omega}\times(0,\infty)$.

\item
For any $(x,t)\in\overline{\Omega}\times(0,\infty)$,
\[
	\int_\Omega G(x,y,t)\,\dee y+\frac{\delta}{\epsilon}\int_{\partial\Omega}G(x,y,t)\,\dee\sigma(y)=1.
\]

\item 
For any $(x,y,t)\in D$ and $s\in(0,\infty)$,
\[
	G(x,y,t+s)
	=
	\int_\Omega G(x,z,t)G(z,y,s)\,\dee z
	+\frac{\delta}{\epsilon}\int_{\partial\Omega}G(x,z,t)G(z,y,s)\,\dee\sigma(z).
\]

\item
Let $(\phi,\psi)\in BC(\Omega)\times BC(\partial\Omega)$.
Define a function $u$ in $\overline{\Omega}\times(0,\infty)$ by
\begin{equation}
\label{eq:1.10}
	\begin{aligned}
	&
	u(x,t)
	\\
	& :=\int_\Omega G(x,y,t)\phi(y)\,\dee y
	+\frac{\delta}{\epsilon}\int_{\partial\Omega}G(x,y,t)\psi(y)\,\dee\sigma(y)
	\\
	& =\int_\Omega G_0\!\left(x,y,\frac{t}{\epsilon}\right)\phi(y)\,\dee y
	+\frac{1}{\delta}\int_\Omega H(x,y,t)\phi(y)\,\dee y
	+\frac{1}{\epsilon}\int_{\partial\Omega}H(x,y,t)\psi(y)\,\dee\sigma(y)
	\end{aligned}
\end{equation}
for $(x,t)\in \overline{\Omega}\times(0,\infty)$.
Then $u$ is the unique bounded classical solution to
problem~\eqref{eq:HDD}
{\rm({\it resp.}~problem~\eqref{eq:HD})}
when $k>0$
{\rm({\it resp.}~$k=0$)}.
In particular, formula~\eqref{eq:1.10} provides an explicit representation of
every bounded classical solution to problem~\eqref{eq:HDD}
{\rm({\it resp.}~problem~\eqref{eq:HD})}
when $k>0$
{\rm({\it resp.}~$k=0$)}.
\end{enumerate}
\end{theorem}
\begin{remark}
\label{Remark:1.1}
In the case of problem~\eqref{eq:HD}, i.e., when $k=0$, 
assertions~{\rm (1)}--{\rm (5)} have already been proved in \cite{IKK25}*{Theorem~1.1}.
\end{remark}
In the proof of Theorem~\ref{Theorem:1.1}~(4)~and~(5), 
we establish a T\"acklind-type uniqueness theorem for problems~\eqref{eq:HDD} and~\eqref{eq:HD} 
(see Proposition~\ref{Proposition:2.1}).
This theorem guarantees the uniqueness of classical solutions with exponential growth at spatial infinity, 
in particular including bounded classical solutions, 
and it plays a crucial role in the proofs of assertions~(4) and~(5).
\medskip

Next, we establish upper and lower bounds for the integral kernel $H$ for the case $k>0$.
\begin{theorem}
\label{Theorem:1.2}
Let $\epsilon>0$, $\delta>0$, and $k>0$ be fixed.
Set
\[
	\Lambda:=\max\left\{\delta,k\epsilon\right\},\quad
	\lambda:=\min\left\{\delta,k\epsilon\right\}.
\]
Then there exists $C>0$ such that 
\begin{equation}
\label{eq:1.11}
	C^{-1}\underline{h}(x,y,t)\Gamma_{N-1}\left(x'-y',\frac{\lambda}{\epsilon\delta}t\right) 
	\le H(x,y,t)\le C \overline{h}(x,y,t)\Gamma_{N-1}\left(x'-y',\frac{\Lambda}{\epsilon\delta}t \right)
\end{equation}
for $(x,y,t)\in D$,
where
\[
	\begin{aligned}
 	& \overline{h}(x,y,t):=
	\left\{
	\begin{array}{ll}
	1,
	&\quad (x,y,t)\in D_1,
	\vspace{5pt}\\
	\displaystyle{\Gamma_1\left(x_N+y_N,\frac{t}{\epsilon}\right)},
	&\quad (x,y,t)\in D_2\cup D_4,
	\vspace{5pt}\\
 	\displaystyle{\left(x_N+y_N+\frac{t}{\delta}\right)\Gamma_1\left(x_N+y_N,\frac{t}{\epsilon}\right)},\,\,
	&\quad (x,y,t)\in D_3,\vspace{2pt}
	\end{array}
	\right.\\
 	& \underline{h}(x,y,t):=
	\left\{
	\begin{array}{ll}
	1,
	&\quad (x,y,t)\in D_1,
	\vspace{5pt}\\
	\displaystyle{\Gamma_1\left(x_N+y_N,\frac{t}{2\epsilon}\right)},
	&\quad (x,y,t)\in D_2\cup D_4,
	\vspace{5pt}\\
	\displaystyle{\left(x_N+y_N+\frac{t}{\delta}\right)\Gamma_1\left(x_N+y_N,\frac{t}{2\epsilon}\right)},
	&\quad (x,y,t)\in D_3.
	\end{array}
	\right.
	\end{aligned}
\]
Here
\[
	\begin{aligned}
	&
	D_1:=\left\{(x,y,t)\in D\,:\,\epsilon(x_N+y_N)^2<6 t,\,\,\,t<\frac{12\delta^2}{\epsilon}\right\},
	\\
	&
	D_2:=\left\{(x,y,t)\in D\,:\,\epsilon(x_N+y_N)^2<6 t,\,\,\,t\ge \frac{12\delta^2}{\epsilon}\right\},
	\\
	&
	D_3:=\left\{(x,y,t)\in D\,:\,\epsilon(x_N+y_N)^2\ge 6t,\,\,\, x_N+y_N+\frac{t}{\delta}
	<\frac{\delta}{\epsilon}\right\},
	\\
	&
	D_4:=\left\{(x,y,t)\in D\,:\,\epsilon(x_N+y_N)^2\ge 6t,\,\,\, x_N+y_N+\frac{t}{\delta}
	\ge \frac{\delta}{\epsilon}\right\}.
	\end{aligned}
\]
\end{theorem}
\begin{remark}
\label{Remark:1.2}
Let $k=0$. 
Then $H(x,y,t)$ satisfies the same lower and upper bounds as in Theorem~{\rm\ref{Theorem:1.2}}
for $(x,y,t)\in D_2\cup D_3\cup D_4$, up to a time scaling.
In contrast, the behavior is different in the region $D_1$.
Indeed, there exists $C>0$ such that
\[
	C^{-1}P\left(x-y^*+\frac{t}{\delta}e_N\right)\le H(x,y,t)
	\le CP\left(x-y^*+\frac{t}{\delta}e_N\right),\quad (x,y,t)\in D_1
\]
{\rm ({\it see} \cite{IKK25}*{Theorem~1.2})}. 
In other words, 
the effect of the surface diffusion term $k\Delta'$ explicitly appears in the behavior of $G$ 
in the region $D_1$. 
\end{remark}

Next, motivated by Theorem~\ref{Theorem:1.1}~(5),
we define a solution to problems~\eqref{eq:HDD} and \eqref{eq:HD} 
with initial data $(\phi,\psi)\in L^p(\Omega)\times L^p(\partial\Omega)$, 
where $p\in[1,\infty]$.
\begin{definition}
\label{Definition:1.1}
Let $(\phi,\psi)\in L^p(\Omega)\times L^p(\partial\Omega)$, where $p\in[1,\infty]$.  
A function $u$ in $\overline{\Omega}\times(0,\infty)$
is called a solution to problem~\eqref{eq:HDD}
{\rm ({\it resp.~problem}~\eqref{eq:HD})} if $u$ satisfies
\begin{equation}
\label{eq:1.12}
	\begin{aligned}
	& u(x,t)=[\G(t)(\phi,\psi)](x)
	\\
	& :=\int_\Omega G(x,y,t)\phi(y)\,\dee y
	+\frac{\delta}{\epsilon}\int_{\partial\Omega} G(x,y,t)\psi(y)\,\dee \sigma(y)
	\\
 	& =\int_\Omega G_0\left(x,y,\frac{t}{\epsilon}\right)\phi(y)\,\dee y
   	+\frac{1}{\delta}\int_\Omega H(x,y,t)\phi(y)\,\dee y
   	+\frac{1}{\epsilon}\int_{\partial\Omega} H(x,y,t)\psi(y)\,\dee\sigma(y)
	\end{aligned}
\end{equation}
for all $(x,t)\in \overline{\Omega}\times(0,\infty)$,
where $k>0$ {\rm ({\it resp.} $k=0$)}.
\end{definition}
When needed, we write $u=\uHDD{\epsilon}{\delta}{k}$ when $k>0$ (resp.~$u=\uHD{\epsilon}{\delta}$ when $k=0$)
to indicate that $u$ is a solution to problem~\eqref{eq:HDD} (resp.~problem~\eqref{eq:HD}).
\medskip

As an application of Theorem~\ref{Theorem:1.2},
we obtain several qualitative properties of $\G(t)(\phi,\psi)$. 
In particular, in Theorem~\ref{Theorem:1.3}~(1) and~(2), we discuss the behavior of $\G(t)(\phi,\psi)$ as $t\to 0^+$,
while in Theorem~\ref{Theorem:1.3}~(3) we establish sharp decay estimates for $\G(t)$ in Lebesgue spaces.
Furthermore, in Theorem~\ref{Theorem:1.3}~(5), 
we obtain the diffusion limit of solutions to problem~\eqref{eq:HDD} as $k\to 0^+$.
\begin{theorem}
\label{Theorem:1.3}
Let $\epsilon>0$, $\delta>0$, $k\ge 0$, and $p\in[1,\infty]$. 
\begin{enumerate}[label={\rm(\arabic*)}]
\item 
	Let $u(t):=\G(t)(\phi,\psi)$, where $(\phi,\psi)\in L^p(\Omega)\times L^p(\partial\Omega)$. 
	Then $u$ is bounded and smooth on $\overline{\Omega}\times[T,\infty)$ for any $T>0$,
	and $u$ satisfies
	\begin{equation}
	\label{eq:1.13}
		\left\{
		\begin{array}{ll}
		\epsilon\partial_t u-\Delta u=0 
		& \mbox{in}\quad \Omega\times(0,\infty), \vspace{3pt}\\
		\delta\partial_t u-k\Delta' u-\partial_{x_N}u=0 
		& \mbox{on}\quad \partial\Omega\times(0,\infty),
		\end{array}
		\right.
	\end{equation}
	pointwise.
	Furthermore, 
	\begin{equation}
	\label{eq:1.14}
	\begin{aligned}
	 & \lim_{t\to 0^+}
	\|(u(t),u(t)|_{\partial\Omega})-(\phi,\psi)\|_{L^p(\Omega)\times L^p(\partial\Omega)}=0\quad\mbox{if }1\le p<\infty,\\
	 & \lim_{t\to 0^+}u(x,t)=\phi(x)\quad\mbox{for } x\in\Omega\,\,\,\mbox{if}\,\,\,\phi\in BC(\Omega).
	\end{aligned}
	\end{equation}
\item  
	Let $(\phi,\psi)\in L^p(\Omega)\times BC(\partial\Omega)$. 
	Then 	
	\begin{equation}
	\label{eq:1.15}
  	\lim_{t\to 0^+}u(x,t)=\psi(x)
	\quad\textrm{for }x\in\partial\Omega
	\end{equation}
	if
	\begin{equation}
	\label{eq:1.16}
	\int_0^1y_N^{-\frac{N-1}{p}}
	\|\phi(\cdot,y_N)\|_{L^p(\mathbb R^{N-1})}\,\dee y_N<\infty.
	\end{equation}
\item
  	Let $1\le p\le q\le\infty$. Define
	\[
	\begin{aligned}
		\|\G(t)\|_{p\to q} 
		& :=\sup\biggr\{\frac{\|(\G(t)(\phi,\psi),\G(t)(\phi,\psi)|_{\partial\Omega})\|_{L^q(\Omega)\times
	 	L^q(\partial\Omega)}}{\|(\phi,\psi)\|_{L^p(\Omega)\times L^p(\partial\Omega)}}
		\\
		 & \qquad\qquad\quad
		 \,:\,
		(\phi,\psi)\in L^p(\Omega)\times L^p(\partial\Omega),\,\,(\phi,\psi)\not=(0,0)\biggr\}.
	\end{aligned}
	\]
	Then there exists $C>0$ such that,
	for the case $k>0$,
	\[
		C^{-1}t^{-\frac{N}{2}\left(\frac{1}{p}-\frac{1}{q}\right)}
 		\le \|\G(t)\|_{p\to q}\le Ct^{-\frac{N}{2}\left(\frac{1}{p}-\frac{1}{q}\right)},\quad t>0, 
	\]
	and, for the case $k=0$,
	\begin{align*}
		C^{-1}\left(t^{-(N-1)\left(\frac{1}{p}-\frac{1}{q}\right)}
		+t^{-\frac{N}{2}\left(\frac{1}{p}-\frac{1}{q}\right)}\right)
 		 & \le \|\G(t)\|_{p\to q}\\
		 & \le C\left(t^{-(N-1)\left(\frac{1}{p}-\frac{1}{q}\right)}
		+t^{-\frac{N}{2}\left(\frac{1}{p}-\frac{1}{q}\right)}\right),\quad t>0.
	\end{align*}
  	In particular, $\|\G(t)\|_{p\to p}=1$ for $t>0$. 
\item
  	Let $(\phi,\psi)\in L^p(\Omega)\times L^p(\partial\Omega)$. 
  	Then 
 	\[
  		\G(t)\left(\G(s)(\phi,\psi),\G(s)(\phi,\psi)|_{\partial\Omega}\right)=\G(t+s)(\phi,\psi)
  	\] 
	for $t,s\in(0,\infty)$. 
	More precisely, 
  	\[
  		[\G(t+s)(\phi,\psi)](x)=\int_\Omega G(x,y,t)[\G(s)(\phi,\psi)](y)\,\dee y
  		+\frac{\delta}{\epsilon}\int_{\partial\Omega}G(x,y,t)[\G(s)(\phi,\psi)](y)\,\dee\sigma(y)
  	\]
  	for $x\in\overline{\Omega}$ and $t, s\in(0,\infty)$.
\item
  	For any $T>0$ and $R>0$,   
  	$$
  		\left|\uHDD{\epsilon}{\delta}{k}(x,t)-\uHD{\epsilon}{\delta}(x,t)\right|
		=O(k)\quad\mbox{as}\quad k\to 0^+
  	$$
  	uniformly for $(x,t)\in \overline{\Omega}\times(0,T]$ with $x_N+t/\delta\ge R$.
\end{enumerate}
\end{theorem}
The convergence rate in assertion~(5) is optimal (see Remark~\ref{Remark:3.1}).
\begin{Remark}
\label{Remark:1.3}
The sufficient condition~\eqref{eq:1.16} for~\eqref{eq:1.15} is sharp.
To show this, let $\epsilon>0$, $\delta>0$, and $k\ge 0$ be fixed. 
Let $u(t)={\G}(t)(\phi,\psi)$, where $\phi(x)=|x|^{-\alpha}\chi_{B^+_1(0)}$ for $x\in\Omega$ with $\alpha<N$ 
and $\psi=0$ on~$\partial\Omega$.
Then it follows from~\eqref{eq:1.11} that
\[
	u(0,t) 
	\ge \int_{B^+_{(6t/\epsilon)^{1/2}}(0)}\Gamma_{N-1}\left(y',Ct\right)|y|^{-\alpha}\,\dee y
	\ge C\int_{B^+_{(6t/\epsilon)^{1/2}}(0)}t^{-\frac{N-1}{2}}|y|^{-\alpha}\,\dee y
	\ge Ct^{-\frac{\alpha-1}{2}}
\]
for all sufficiently small $t>0$.
Hence, \eqref{eq:1.15} does not hold at $x=0$ if $\alpha\ge 1$.
On the other hand, for any sufficiently small $L>0$, we have
\[
	\begin{aligned}
	\int_0^L y_N^{-\frac{N-1}{p}}\|\phi(\cdot,y_N)\|_{L^p({\mathbb R}^{N-1})}\,\dee y_N
	& \ge\int_0^L y_N^{-\frac{N-1}{p}}
	\left(\int_{B'_{y_N}(0)}(y_N^2+|y'|^2)^{-\frac{\alpha p}{2}}\,\dee y'\right)^{\frac{1}{p}}\,\dee y_N
	\\
	& \ge C\int_0^L y_N^{-\frac{N-1}{p}}y_N^{-\alpha+\frac{N-1}{p}}\,\dee y_N
	=C\int_0^Ly_N^{-\alpha}\,\dee y_N.
	\end{aligned}
\]
This implies that if $\phi$ satisfies \eqref{eq:1.16}, then $\alpha<1$.
Consequently, condition~\eqref{eq:1.16} is sharp for~\eqref{eq:1.15}.
\end{Remark}

As an application of Theorem~\ref{Theorem:1.1}, 
we further investigate the diffusion limits of solutions to problem~\eqref{eq:HDD} in various regimes. 
In particular, we show that several initial--boundary value problems such as 
\eqref{eq:HD}, \eqref{eq:HDN}, \eqref{eq:LDD}, and \eqref{eq:LD}, 
as well as the heat equation with the Dirichlet or Neumann boundary conditions, 
arise as diffusion limits of problem~\eqref{eq:HDD}.
The corresponding results are presented in Sections~4 and~5.
\medskip

The remainder of this paper is organized as follows.
In Section~\ref{section:2} we present several preliminary results on the integral kernels $\Gamma_d$, $G_0$, and $P$. 
Furthermore, we establish a T\"acklind-type uniqueness theorem for problems~\eqref{eq:HDD} and \eqref{eq:HD}.  
In Section~\ref{section:3} we prove Theorems~\ref{Theorem:1.1}, \ref{Theorem:1.2}, and~\ref{Theorem:1.3}.
In Section~\ref{section:4} we investigate the diffusion limits of solutions in the cases $\epsilon\to 0^+$ and $\epsilon\to\infty$, 
together with some related limiting problems.
In Section~\ref{section:5} we study the diffusion limits of solutions in the cases $\delta\to 0^+$, $\delta\to\infty$, and $k\to\infty$,
along with further related limiting problems. 
%
\section{Preliminaries}\label{section:2}
In this section, we recall some properties of the integral kernels
$\Gamma_d$, $G_0$, and $P$.
We also establish a T\"acklind-type uniqueness theorem for
problems~\eqref{eq:HDD} and \eqref{eq:HD}.
Throughout the paper, we use $C$ to denote generic positive constants,
whose values may change from line to line.
As in Section~1, we set
\[
	D:=\overline{\Omega}\times\overline{\Omega}\times(0,\infty),
	\qquad
	\Omega_L:=\{x=(x',x_N)\in\overline{\Omega}:\,0\le x_N\le L\}
	\quad\mbox{for $L>0$}.
\]

Let $\Gamma_d$ be defined as in \eqref{eq:1.1}. 
Then $\Gamma_d$ satisfies the following properties. 
\begin{enumerate}[label={\rm($\Gamma$\arabic*)}]
\item
	For $x\in{\mathbb R}^d$, $t$, $s\in(0,\infty)$, 
  	\[
  		\int_{{\mathbb R}^d}\Gamma_d(x,t)\,\dee x=1,
  		\qquad
  		\Gamma_d(x,t+s)=\int_{{\mathbb R}^d}\Gamma_d(x-z,t)\Gamma_d(z,s)\,\dee z.
  	\] 
\item
  	For any $k$, $\ell=0,1,2,\dots$, there exists $C>0$ such that 
  	\begin{equation}
  	\label{eq:2.1}
  		\left|\partial_t^\ell \nabla_x^k\Gamma_d(x,t)\right|
		\le Ct^{-\frac{d+k+2\ell}{2}}\exp\left(-\frac{|x|^2}{8t}\right),
  		\quad (x,t)\in{\mathbb R}^d\times(0,\infty).
  	\end{equation}
  	This yields 
  	\begin{equation}
  	\label{eq:2.2}
		\begin{aligned}
         	&  
		\|\partial_t^\ell \nabla_x^k\Gamma_d(t)\|_{L^p({\mathbb R}^d)}
		\le Ct^{-\frac{d}{2}\left(1-\frac{1}{p}\right)-\frac{k+2\ell}{2}},
		\\
	 	& 
		\|\partial_t^\ell \partial_x^k\Gamma_1(\cdot+\tau,t)\|_{{L^p(\mathbb R_+)}}
	 	\le Ct^{-\frac{1}{2}\left(1-\frac{1}{p}\right)-\frac{k+2\ell}{2}}\exp\left(-\frac{\tau^2}{16t}\right),
		\end{aligned}
  	\end{equation}
  	for $t>0$ and $\tau\ge 0$, where $p\in[1,\infty]$. 
\item
  	If $\phi\in L^p({\mathbb R}^d)$ with $1\le p<\infty$,   
  	\[
  		\lim_{t\to 0^+}\left\|\,\int_{{\mathbb R}^d}\Gamma_d(\cdot-y)\phi(y)\,\dee y-\phi\,\right\|_{L^p({\mathbb R}^d)}=0.
  	\]
  	If $\phi\in BC({\mathbb R}^d)$, then 
  	\[
  		\lim_{t\to 0^+}\sup_{x\in K}\left|\,\int_{\mathbb R^d}\Gamma_d(x-y)\phi(y)\,\dee y-\phi(x)\,\right|=0
  	\]
	for any compact set $K\subset{\mathbb R}^d$.
\end{enumerate} 

Next, we consider the heat equation in $\Omega$ with an inhomogeneous Dirichlet boundary condition:
\begin{equation}
\label{eq:HiD}
\tag{$\mbox{H}_{{\rm D}, F}$}
	\left\{
	\begin{array}{ll}
	\epsilon\partial_t u-\Delta u=0\quad & \mbox{in}\quad\Omega\times(0,\infty),
	\vspace{3pt}\\
	u=F\quad & \mbox{on}\quad \partial\Omega\times(0,\infty),
	\vspace{3pt}\\
	u=\phi\quad & \mbox{in}\quad\Omega\times\{0\},
	\end{array}
	\right.
\end{equation}
where $\epsilon>0$, $F\in L^\infty(\partial\Omega\times(0,\infty))$, and $\phi\in L^p(\Omega)$ with $p\in[1,\infty]$.
A function $u$ on $\overline{\Omega}\times(0,\infty)$ 
is called a solution to problem~\eqref{eq:HiD} 
if $u$ satisfies 
\begin{equation}
\label{eq:2.3}
	\begin{aligned}
	&
	u(x,t)
	=\uHDi{\epsilon}{F}(x,t) 
	\\
	& :=\int_\Omega G_0\left(x,y,\frac{t}{\epsilon}\right)\phi(y)\,\dee y
	+\frac{1}{\epsilon}\int_0^t\int_{\partial\Omega} 
	\partial_{y_N}G_0\left(x,y,\frac{t-\tau}{\epsilon}\right)F(y,\tau)\,\dee\sigma(y)\,\dee\tau
	\\
 	& \,\,=\int_\Omega G_0\left(x,y,\frac{t}{\epsilon}\right)\phi(y)\,\dee y
	\\
 	& \qquad
 	-\frac{2}{\epsilon}\int_0^t\int_{\partial\Omega}\Gamma_{N-1}\left(x'-y',\frac{t-\tau}{\epsilon}\right)
	\partial_{x_N}\Gamma_1\left(x_N,\frac{t-\tau}{\epsilon}\right)F(y,\tau)\,\dee\sigma(y)\,\dee\tau
	\end{aligned}
\end{equation}
for all $(x,t)\in\overline{\Omega}\times(0,\infty)$. 
If $F\in BC(\partial\Omega\times(0,\infty))$ and $\phi\in BC(\Omega)$, 
then $\uHDi{\epsilon}{F}$ is a bounded classical solution to problem~\eqref{eq:HiD}. 
In particular, if $F\equiv 0$ on $\partial\Omega\times(0,\infty)$), 
then we refer to problem~\eqref{eq:HiD} as $({\rm H_{D,0}})$ and write $\uHDi{\epsilon}{F}=\uHDi{\epsilon}{0}$. 
It turns out that
\begin{equation}
\label{eq:2.4}
	\uHDi{\epsilon}{0}(x,t)
	=\int_\Omega G_0\left(x,y,\frac{t}{\epsilon}\right)\phi(y)\,\dee y,\quad (x,t)\in \overline{\Omega}\times(0,\infty),
\end{equation}
and it follows from properties ($\Gamma$2) and ($\Gamma$3) that 
\begin{equation}
\label{eq:2.5}
	\begin{aligned}
 	& 
	\lim_{t/\epsilon\to 0^+}\|\uHDi{\epsilon}{0}(t)-\phi\|_{L^p(\Omega)}=0
	\quad\mbox{if $\phi\in L^p(\Omega)$, where $1\le p<\infty$};
	\\
 	& 
	\lim_{t/\epsilon\to 0^+}\sup_{x\in K}|\uHDi{\epsilon}{0}(x,t)-\phi(x)|=0
	\quad\mbox{for any compact set $K\subset\Omega$ if $\phi\in BC(\Omega)$};
	\\
 	& 
	\sup_{t>0}\,\,t^{\frac{N}{2p}}\|\uHDi{\epsilon}{0}(t)\|_{L^\infty(\Omega)}
	\le C\epsilon^{\frac{N}{2p}}\|\phi\|_{L^p(\Omega)}
	\quad \mbox{if $\phi\in L^p(\Omega)$, where $1\le p\le\infty$}.
	\end{aligned}
\end{equation}
It also follows from \cite{FIK21}*{Lemma~2.1} that, for any $L>0$, 
\begin{equation}
\label{eq:2.6}
	\sup_{t>0}\,\,t^{\frac{1}{2}}\|\uHDi{\epsilon}{0}(t)\|_{L^\infty(\Omega_L)}
	\le C\epsilon^{\frac{1}{2}}\|\phi\|_{L^\infty(\Omega)}.
\end{equation}
In this paper, we consider the case where the boundary data $F$ is given by
\begin{equation}
\label{eq:2.7}
	\Psi(x,t):=\int_{\partial\Omega}\Gamma_{N-1}\left(x'-y',\frac{t}{\theta}\right)\psi(y)\,\dee\sigma(y),
	\quad (x,t)\in \partial\Omega\times(0,\infty),
\end{equation}
where $\theta>0$ and $\psi\in L^p(\partial\Omega)$ with $p\in[1,\infty]$.
Here $\Psi$ corresponds to a solution to the problem 
\[
	\theta\partial_t\Psi-\Delta'\Psi=0\quad\mbox{in}\quad{\mathbb R}^{N-1}\times(0,\infty),
	\quad 
	\Psi=\psi\quad\mbox{on}\quad{\mathbb R}^{N-1}\times\{0\}. 
\]
We denote by $\uHDP{\epsilon}{\theta}$ the solution of problem~\eqref{eq:HiD} with $F=\Psi$.
Then, by property~($\Gamma$1), \eqref{eq:2.3}, and \eqref{eq:2.7}, we obtain 
\begin{equation}
\label{eq:2.8}
 	\uHDP{\epsilon}{\theta}(x,t)
 	=\int_\Omega G_0\left(x,y,\frac{t}{\epsilon}\right)\phi(y)\,\dee y
 	+\frac{1}{\epsilon}\int_{\partial\Omega}\tilde H(x,y,t)\psi(y)\,\dee\sigma(y)
\end{equation}
for $(x,t)\in \overline{\Omega}\times(0,\infty)$, 
where
\begin{equation}
\label{eq:2.9}
	\tilde H(x,y,t):=-2\int_0^t\Gamma_{N-1}\left(x'-y',\frac{t-\tau}{\epsilon}+\frac{\tau}{\theta}\right)
	\partial_{x_N}\Gamma_1\left(x_N,\frac{t-\tau}{\epsilon}\right)\,\dee\tau\ge 0.
\end{equation}
Similarly, we consider the heat equation in $\Omega$ with the homogeneous Neumann boundary condition:
\begin{equation}
\label{eq:HhN}
\tag{$\mbox{H}_{{\rm N, 0}}$}
	\left\{
	\begin{array}{ll}
	\epsilon\partial_t u-\Delta u=0\quad & \mbox{in}\quad \Omega\times(0,\infty),
	\vspace{3pt}\\
	-\partial_{x_N}u=0\quad & \mbox{on}\quad \partial\Omega\times(0,\infty),
	\vspace{3pt}\\
	u=\phi\quad & \mbox{in}\quad\Omega\times\{0\},
	\end{array}
	\right.
\end{equation}
where $\epsilon>0$ and $\phi\in L^p(\Omega)$ with $p\in[1,\infty]$.
A function $u$ on $\overline{\Omega}\times(0,\infty)$ is called a solution to problem~\eqref{eq:HhN} 
if $u$ satisfies 
\[
	u(x,t)={\uHhN{\epsilon}(x,t)}
	:=\int_\Omega G_N\left(x,y,\frac{t}{\epsilon}\right)\phi(y)\,\dee y
\]
for all $(x,t)\in\overline{\Omega}\times(0,\infty)$,
where
\begin{equation}
\label{eq:2.10}
	G_N(x,y,t)=\Gamma_N(x-y,t)+\Gamma_N(x-y^*,t),\quad (x,y,t)\in D.
\end{equation}
Then 
$$
	\uHhN{\epsilon}\in C^{2;1}(\Omega\times(0,\infty))\cap C^{1;0}(\overline{\Omega}\times(0,\infty)),
$$ 
and it satisfies the first and second equations of \eqref{eq:HhN} pointwise, 
as well as the properties stated in~\eqref{eq:2.5}.
\medskip

Let $P$ be defined as in \eqref{eq:1.6}. 
For any $\psi\in L^p(\partial\Omega)$ with $p\in[1,\infty]$, 
a function $u$ on $\Omega$ is called a solution to the Laplace equation in $\Omega$ 
with an inhomogeneous Dirichlet boundary condition:
\begin{equation}
\label{eq:LiD}
\tag{$\mbox{L}_{{\rm D,\psi}}$}
	-\Delta u=0\quad\mbox{in}\quad\Omega,\qquad u=\psi\quad\mbox{on}\quad\partial\Omega,
\end{equation}
if $u$ satisfies 
\begin{equation}\label{eq:2.11}
	u(x)=\uLDi{\psi}(x):=\int_{\partial\Omega}P(x'-y',x_N)\psi(y)\,\dee\sigma(y)
\end{equation}
for all $x=(x',x_N)\in\Omega$. 
Then $\uLDi{\psi}\in C^2(\Omega)$ and is harmonic in $\Omega$. 
Furthermore, 
$\uLDi{\psi}$ satisfies
\begin{equation}
\label{eq:2.12}
	\begin{aligned}
  	& \lim_{x_N\to 0^+}\|\uLDi{\psi}(x_N)-\psi\|_{L^p(\mathbb R^{N-1})}=0
	\quad\mbox{if}\quad p\in[1,\infty),
	\\
  	& \lim_{x_N\to 0^+}\sup_{x'\in K'}\left|\uLDi{\psi}(x',x_N)-\psi(x')\right|=0
	\quad\mbox{for any compact set $K'\subset\partial\Omega$ if $\psi\in BC(\partial\Omega)$},
	\\
  	& \sup_{x_N>0}\,\,x_N^{\frac{N-1}{p}}\|\uLDi{\psi}(x_N)\|_{L^\infty({\mathbb R}^{N-1})}
	\le C\|\psi\|_{L^p(\partial\Omega)}
	\quad \mbox{if $\psi\in L^p(\partial\Omega)$, where $1\le p\le\infty$}.
	\end{aligned}
\end{equation}
(See e.g., \cite{FIK21}*{Section~2}.)
In a similar way, we consider problem~\eqref{eq:LD} 
(the Laplace equation in $\Omega$ with a (non-diffusive) dynamical boundary condition), 
which already appeared in Section~1. 
For any $\psi\in L^p(\partial\Omega)$ with $1\le p\le\infty$, 
a function $u$ on $\Omega\times(0,\infty)$ is called a solution to problem~\eqref{eq:LD} 
if $u$ satisfies 
\[
	u(x,t)={\uLD{\delta}(x,t)}
	:=\int_{\partial\Omega}P\left(x'-y',x_N+\frac{t}{\delta}\right)\psi(y)\,\dee\sigma(y)
\]
for all $(x,t)=(x',x_N,t)\in\overline{\Omega}\times(0,\infty)$. 
Then $\uLD{\delta}\in {C^{2;0}(\Omega\times(0,\infty))\cap C^{1;1}(\overline{\Omega}\times(0,\infty))}$
and it satisfies the first and second equations of~\eqref{eq:LD} pointwise. 
Furthermore, $\uLD{\delta}$ satisfies 
\[
	\begin{aligned}
  	& \lim_{t\to 0^+}\|\uLD{\delta}(t)-\psi\|_{L^p(\partial\Omega)}=0
	\quad\mbox{if}\quad p\in[1,\infty),
	\\
  	& \lim_{t\to 0^+}\sup_{x\in K'}\left|\uLD{\delta}(x,t)-\psi(x)\right|=0
	\quad\mbox{for any compact set $K'\subset\partial\Omega$ if $\psi\in BC(\partial\Omega)$}.
	\end{aligned}
\]
In addition, it follows from \eqref{eq:2.12} that
\begin{equation}
\label{eq:2.13}
	\sup_{(x_N,t)\in[0,\infty)\times(0,\infty)}\,
	\left(x_N+\frac{t}{\delta}\right)^{\frac{N-1}{p}}\|\uLD{\delta}(x_N,t)\|_{L^\infty(\mathbb R^{N-1})}
	\le C\|\psi\|_{L^p(\partial\Omega)},\qquad p\in[1,\infty].
\end{equation}

In the rest of this section, 
we prove a T\"acklind-type uniqueness theorem for problems~\eqref{eq:HDD} and \eqref{eq:HD}.
\begin{proposition}
\label{Proposition:2.1}
	Let $u$ be a bounded classical solution to problem~\eqref{eq:HDD} or problem~\eqref{eq:HD} 
	with zero initial data, i.e.,
	$\phi\equiv 0$ in $\Omega$ and $\psi\equiv 0$ on $\partial\Omega$. 
	Assume that, for any $T\in(0,\infty)$, there exists $\gamma>0$ such that
	\begin{equation}
	\label{eq:2.14}
		\int_0^T\int_{B^+_R(0)\setminus B^+_{R/2}(0)} u^2\,\dee x\,\dee t
		+\int_0^T\int_{B'_R(0)\setminus B'_{R/2}(0)} u^2\,\dee\sigma(x)\,\dee t
		\le \exp(\gamma R^2)
	\end{equation}
	for all $R>1$. 
	Then $u$ is identically zero in $\overline{\Omega}\times(0,\infty)$. 
\end{proposition}
{\bf Proof.}
The proof is a modification of the proof of \cite{IM98}*{Theorem~B} (see also \cite{IM01}*{Theorem~2.1}).
Let $R>1$, and set 
\begin{equation}
\label{eq:2.15}
	d_R(x):=\max\left\{0,|x|-\frac{R}{2}\right\},
	\quad
	\zeta_R(x):=\frac{1}{R}\max\{0,\min\{2R-|x|,R\}\},\quad x\in\overline{\Omega}.
\end{equation}
Then 
\begin{equation}
\label{eq:2.16}
	\zeta_R=1\,\,\,\mbox{in}\,\,\,B_R^+(0),\quad
	|\nabla\zeta_R(x)|^2\le \frac{1}{R^2}\,\,\,\mbox{in}\,\,\,B^+_{2R}(0)\setminus B_R^+(0),
	\quad
	\zeta_R=0\,\,\,\mbox{in}\,\,\,\Omega\setminus B^+_{2R}(0). 
\end{equation}
Let $T>0$ be fixed. 
Let $a$ be a sufficiently small positive constant to be chosen later such that $a\in(0,T/4]$. 
For any $\tilde{T}\in[a,T]$, we define $t_1:=\tilde{T}-a\ge 0$ and $t_2:=\tilde{T}-a/2$. 
Set 
\begin{equation}
\label{eq:2.17}
	w(x,t):=-\frac{\lambda d_R(x)^2}{\tilde{T}-t},\quad (x,t)\in\overline{\Omega}\times(t_1,t_2),
\end{equation}
where $\lambda>0$. 
If $t_1>0$, we obtain 
\[
	\begin{aligned}
 	& \epsilon\int_{t_1}^{t_2}\int_\Omega\partial_t u e^w u\zeta_R^2\,\dee x\,\dee t
	=\epsilon\int_{t_1}^{t_2}\int_\Omega
	\left\{\frac{1}{2}\partial_t (e^w u^2\zeta_R^2)-\frac{1}{2}e^w u^2\zeta_R^2\partial_t w\right\}\,\dee x\,\dee t
	\\
 	& =\frac{\epsilon}{2}\int_\Omega e^{w(t)}u(t)^2\zeta_R^2\,\dee x\,\biggr|_{t=t_1}^{t=t_2}
	-\frac{\epsilon}{2}\int_{t_1}^{t_2}\int_\Omega e^w u^2\zeta_R^2\partial_tw\,\dee x\,\dee t
	\end{aligned}
\]
and
\[
	\begin{aligned}
  	& -\int_{t_1}^{t_2}\int_\Omega \Delta u e^wu \zeta_R^2\,\dee x\,\dee t
  	=\int_{t_1}^{t_2}\int_{\partial\Omega}\partial_{x_N}u\cdot e^w u\zeta_R^2\,\dee\sigma(x)\,\dee t
	+\int_{t_1}^{t_2}\int_\Omega \nabla u\cdot\nabla (e^w u\zeta_R^2)\,\dee x\,\dee t
	\\
 	& \ge \int_{t_1}^{t_2}\int_{\partial\Omega}(\delta\partial_t-k\Delta')u\cdot e^wu\zeta_R^2\,\dee\sigma(x)\,\dee t
	\\
 	& \qquad\quad
 	+\int_{t_1}^{t_2}\int_\Omega \left\{e^w|\nabla u|^2\zeta_R^2-e^w u|\nabla u|
	[|\nabla w|\zeta_R^2+2\zeta_R|\nabla\zeta_R|]\right\}\,\dee x\,\dee t
	\\
  	& \ge \int_{t_1}^{t_2}\int_{\partial\Omega}(\delta\partial_t-k\Delta')u\cdot e^wu\zeta_R^2\,\dee\sigma(x)\,\dee t
	 -2\int_{t_1}^{t_2}\int_\Omega\left\{e^w u^2[|\nabla w|^2\zeta_R^2+|\nabla\zeta_R|^2]\right\}\,\dee x\,\dee t.
	\end{aligned}
\]
Similarly, we have
\[
	\begin{aligned}
	& \int_{t_1}^{t_2}\int_{\partial\Omega}(\delta\partial_t-k\Delta')u\cdot e^wu\zeta_R^2\,\dee\sigma(x)\,\dee t
	\\
 	& \ge \frac{\delta}{2}\int_{\partial\Omega} e^{w(t)}u(t)^2\zeta_R^2\,\dee\sigma(x)\,\biggr|_{t=t_1}^{t=t_2}
 	-\frac{\delta}{2}\int_{t_1}^{t_2}\int_{\partial\Omega} e^wu^2\zeta_R^2\partial_t w\,\dee\sigma(x)\,\dee t
	\\
 	& \qquad\quad
 	-2k\int_{t_1}^{t_2}\int_{\partial\Omega} \left\{e^w u^2[|\nabla' w|^2\zeta_R^2+|\nabla'\zeta_R|^2]\right\}\,\dee\sigma(x)\,\dee t.
	\end{aligned}
\]
Hence, by the equations in \eqref{eq:HDD} and \eqref{eq:2.16}, we obtain 
\begin{equation}
\label{eq:2.18}
	\begin{aligned}
 	& \left[\epsilon\int_{B^+_R(0)}e^{w(t)}u(t)^2\,\dee x+\delta\int_{B'_R(0)}e^{w(t)}u(t)^2\,\dee\sigma(x)\right]\,\biggr|_{t=t_2}
	\\
 	& \le\left[\epsilon\int_{B^+_{2R}(0)}e^{w(t)}u(t)^2\,\dee x
	+\delta\int_{B'_{2R}(0)}e^{w(t)}u(t)^2\,\dee\sigma(x)\right]\,\biggr|_{t=t_1}
	\\
 	& \qquad
  	+\int_{t_1}^{t_2}\int_\Omega e^w u^2\zeta_R^2\left\{\epsilon\partial_t w+4|\nabla w|^2\right\}\,\dee x\,\dee t
	\\
 	& \qquad\qquad
  	+\int_{t_1}^{t_2}\int_{\partial\Omega} e^w u^2\zeta_R^2\left\{\delta\partial_t w
	+4k|\nabla' w|^2\right\}\,\dee\sigma(x)\,\dee t
	\\
 	& \qquad\qquad\qquad
 	+\frac{4}{R^2}\int_{t_1}^{t_2}\left(\int_{B^+_{2R}(0)\setminus B^+_R(0)}e^w u^2\,\dee x
	+k\int_{B'_{2R}(0)\setminus B'_R(0)}e^w u^2\,\dee\sigma(x)\right)\,\dee t.
	\end{aligned}
\end{equation}
If $t_1=0$, we first obtain the same estimate as in \eqref{eq:2.18}
with $t_1$ replaced by $\tau\in(0,t_2)$,
and let $\tau\to 0^+$.
The passage to the limit is justified by the Lebesgue dominated convergence theorem.

Setting $\lambda=\min\{\epsilon/16,\delta/(16k)\}$ when $k>0$ and $\lambda=\epsilon/16$ when $k=0$, 
we have
\[
	\begin{aligned}
 	& \epsilon\partial_t w+4|\nabla w|^2
 	\le -\frac{\epsilon\lambda d_R(x)^2}{(\tilde{T}-t)^2}
	+16\frac{\lambda^2 d_R(x)^2}{(\tilde{T}-t)^2}\le 0\quad\mbox{in}\quad\Omega\times(t_1,t_2),
	\\
 	& 
	\delta\partial_t w+4k|\nabla' w|^2
	\le -\frac{\delta\lambda d_R(x)^2}{(\tilde{T}-t)^2}
	+16k\frac{\lambda^2 d_R(x)^2}{(\tilde{T}-t)^2}\le 0\quad\mbox{on}\quad\partial\Omega\times(t_1,t_2).
	\end{aligned}
\]
These, together with \eqref{eq:2.18}, the non-positivity of $w$, 
and the fact that $d_R(x)=0$ for $x\in B_{R/2}$, imply that 
\begin{equation}
\label{eq:2.19}
	\begin{aligned}
 	& \left[\epsilon\int_{B^+_{R/2}(0)}u(t)^2\,\dee x+\delta\int_{B'_{R/2}(0)}u(t)^2\,\dee\sigma(x)\right]\,\biggr|_{t=t_2}
	\\
 	& \le\left[\epsilon\int_{B^+_{2R}(0)}u(t)^2\,\dee x+\delta\int_{B'_{2R}(0)}u(t)^2\,\dee\sigma(x)\right]\,\biggr|_{t=t_1}
	\\
 	& \qquad
 	+\frac{4}{R^2}\int_{t_1}^{t_2}\left(\int_{B^+_{2R}(0)\setminus B^+_R(0)}e^w u^2\,\dee x
	+\int_{B'_{2R}(0)\setminus B'_R(0)}e^w u^2\,\dee\sigma(x)\right)\,\dee t.
	\end{aligned}
\end{equation}
On the other hand, it follows from \eqref{eq:2.14}, \eqref{eq:2.15}, and \eqref{eq:2.17} that
\[
	\begin{aligned}
 	& \int_{t_1}^{t_2}\left(\int_{B^+_{2R}(0)\setminus B^+_R(0)}e^w u^2\,\dee x
	+\int_{B'_{2R}(0)\setminus B'_R(0)}e^w u^2\,\dee\sigma(x)\right)\,\dee t
	\\
 	& \le \exp\left(-\frac{\lambda R^2}{4(\tilde{T}-t_1)}\right)
 	\int_{t_1}^{t_2}\left(\int_{B^+_{2R}(0)\setminus B^+_R(0)}u^2\,\dee x
	+\int_{B'_{2R}\setminus B'_R(0)}u^2\,\dee\sigma(x)\right)\,\dee t
	\\
 	& \le \exp\left(-\frac{\lambda R^2}{4(\tilde{T}-t_1)}+4\gamma R^2\right)
 	=\exp\left(-\frac{\lambda R^2}{4a}+4\gamma R^2\right).
	\end{aligned}
\]
Taking sufficiently small $a\in(0,T/4]$ so that $a\le \lambda/(16\gamma)$ if necessary, 
we obtain 
\[
	\int_{t_1}^{t_2}\left(\int_{B^+_{2R}(0)\setminus B^+_R(0)}e^w u^2\,\dee x
	+\int_{B'_{2R}(0)\setminus B'_R(0)}e^w u^2\,\dee\sigma(x)\right)\,\dee t \le 1. 
\]
This, together with \eqref{eq:2.19}, implies that
\[ 
	\begin{aligned}
 	& \left[\epsilon\int_{B^+_{R/2}(0)}u(t)^2\,\dee x
	+\delta\int_{B'_{R/2}(0)}u(t)^2\,\dee\sigma(x)\right]\,\biggr|_{t=\tilde{T}-a/2}
	\\
 	& \le\left[\epsilon\int_{B^+_{2R}(0)}u(t)^2\,\dee x
	+\delta\int_{B'_{2R}(0)}u(t)^2\,\dee\sigma(x)\right]\,\biggr|_{t=\tilde{T}-a}+\frac{4}{R^2}.
	\end{aligned}
\]

Recalling that $a\in(0,T/4]$, we find an integer $\ell=2,3,\ldots$ such that 
$\ell a/2 \le T < (\ell+1)a/2$, and set $\tilde{T}=\ell a/2$.
Define $R_n=4^nR/2$ for $n=0,1,\dots$. 
Then we have
\[
	\begin{aligned}
	& \left[\epsilon\int_{B^+_{R_0}(0)}u(t)^2\,\dee x
	+\delta\int_{B'_{R_0}(0)}u(t)^2\,\dee\sigma(x)\right]\,\biggr|_{t=\tilde{T}-a/2}\\
 	& \le \left[\epsilon\int_{B^+_{R_{\ell-1}}(0)}u(t)^2\,\dee x
	+\delta\int_{B'_{R_{\ell-1}}(0)}u(t)^2\,\dee\sigma(x)\right]\,\biggr|_{t=0}
	+4\sum_{n=1}^{\ell-1}\left(\frac{R_n}{2}\right)^{-2}
	\le CR^{-2}
	\end{aligned}
\]
for $R>1$. 
Letting $R\to\infty$, we obtain 
\[
	\left[\epsilon\int_\Omega u(x,t)^2\,\dee x
	+\delta\int_{\partial\Omega} u(x,t)^2\,\dee\sigma(x)\right]\biggr|_{t=(\ell-1)a/2}=0,
\]
which implies that $u(x,t)\equiv 0$ in $\overline{\Omega}$ for $t=(\ell-1)a/2$. 
Since $\ell a/2 \le T < (\ell+1)a/2$, we have $(\ell-1)a/2\in(T-a,T-a/2]$.
Letting $a\to 0^+$, we deduce that $u(x,T)\equiv 0$ in $\overline{\Omega}$. 
As $T\in(0,\infty)$ is arbitrary, 
we conclude that $u\equiv 0$ in $\overline{\Omega}\times(0,\infty)$.
Thus, Proposition~\ref{Proposition:2.1} follows.
$\Box$
\begin{remark}
\label{Remark:2.1}
{\rm (1)}
In the proofs of \cite{IM98}*{Theorem~B} and \cite{IM01}*{Theorem~2.1},
the T\"acklind-type uniqueness theorems were proved under the assumption 
that solutions are extended by zero for $t\in(-\infty,0)$.
However, in the definition of our bounded classical solution,
such an extension cannot be applied.
\newline
{\rm (2)} 
In~\cite{IKK25}, the uniqueness of bounded classical solutions to problem~\eqref{eq:HD} 
was established as a consequence of \cite{GH18}*{Theorem~3.1}. 
Proposition~{\rm\ref{Proposition:2.1}} yields an alternative proof of uniqueness for bounded classical solutions to problem~\eqref{eq:HD}, 
which moreover extends to solutions with exponential growth.
\end{remark}
%
\section{Proof of Theorems~\ref{Theorem:1.1}--\ref{Theorem:1.3}}\label{section:3}
In this section, we prove Theorems~\ref{Theorem:1.1}--\ref{Theorem:1.3}.
Throughout this section, we set
\[
	E:=\{(x,y,t,\tau)\,:\,(x,y,t)\in D,\,\,\tau\in(0,t)\},
\]
where $D=\overline{\Omega}\times\overline{\Omega}\times(0,\infty)$.
\vspace{5pt}
\newline
{\bf Proof of Theorem~\ref{Theorem:1.1}.}
Let $\epsilon>0$, $\delta>0$, and $k\ge 0$ be fixed.
Define
\begin{equation}
\label{eq:3.1}
	\begin{aligned}
	f(x,y,t,\tau):= & \,
	\Gamma_{N-1}\left(x'-y',\frac{t-\tau}{\epsilon}+\frac{k}{\delta}\tau\right)g(x,y,t,\tau),
	\\
	g(x,y,t,\tau) := & \, 
	-2\partial_{x_N}\Gamma_1\left(x_N+y_N+\frac{\tau}{\delta},\frac{t-\tau}{\epsilon}\right)
	\\
	= & \,
	\frac{\epsilon(x_N+y_N+\tau/\delta)}{t-\tau}\Gamma_1\left(x_N+y_N+\frac{\tau}{\delta},\frac{t-\tau}{\epsilon}\right),
	\end{aligned}
\end{equation}
for $(x,y,t,\tau)\in E$. 
Then 
\begin{equation}
\label{eq:3.2}
	\begin{array}{ll}
	\displaystyle{H(x,y,t)=\int_0^t f(x,y,t,\tau)\,\dee\tau}, \quad 
	& (x,y,t)\in D,\vspace{7pt}\\
	\displaystyle{G(x,y,t)=G_0\left(x,y,\frac{t}{\epsilon}\right)+\frac{1}{\delta}\int_0^t f(x,y,t,\tau)\,\dee\tau}>0,\quad 
	& (x,y,t)\in D,\vspace{7pt}\\
	G(x,y,t)=G(y,x,t),\quad 
	& (x,y,t)\in D,\vspace{9pt}\\
	\epsilon\partial_t f(x,y,t,\tau)-\Delta_x f(x,y,t,\tau)=0,\quad  
	& (x,y,t,\tau)\in E,\vspace{7pt}\\
	\displaystyle{\lim_{\tau\to t^-}}f(x,y,t,\tau)=0,\quad 
	& (x,y,t)\in D. 
	\end{array}
\end{equation}
The second and third relations in \eqref{eq:3.2} implies assertion~(1). 
Since
$$
	(\epsilon\partial_t-\Delta_x)G_0\left(x,y,\frac{t}{\epsilon}\right)=0,\quad (x,y,t)\in D,
$$
it follows from the second, fourth, and fifth relations in \eqref{eq:3.2} that
\begin{equation}
\label{eq:3.3}
	\epsilon\partial_t G(x,y,t)-\Delta_x G(x,y,t)
 	=\frac{1}{\delta}\int_0^t (\epsilon\partial_t-\Delta_x)f(x,y,t,\tau)\,\dee\tau=0
 \end{equation}
for $(x,y,t)\in D$. 
It also follows from \eqref{eq:1.2} that
	\begin{align*}
 	& \partial_t f(x,y,t,\tau)-\frac{k}{\delta}\Delta_x' f(x,y,t,\tau)-\frac{1}{\delta}\partial_{x_N}f(x,y,t,\tau)
	\\
 	& =\left(-2+2\frac{k\epsilon}{\delta}\right)
	\partial_t\Gamma_{N-1}\left(x'-y',\frac{t-\tau}{\epsilon}+\frac{k}{\delta}\tau\right)
	\partial_{x_N}\Gamma_1\left(x_N+y_N+\frac{\tau}{\delta},\frac{t-\tau}{\epsilon}\right)
	\\
	& \quad
	-2\Gamma_{N-1}\left(x'-y',\frac{t-\tau}{\epsilon}+\frac{k}{\delta}\tau\right)
	\partial_t\partial_{x_N}\Gamma_1\left(x_N+y_N+\frac{\tau}{\delta},\frac{t-\tau}{\epsilon}\right)
	\\
	& \quad
	+\frac{2}{\delta}\Gamma_{N-1}\left(x'-y',\frac{t-\tau}{\epsilon}+\frac{k}{\delta}\tau\right)
	\partial_{x_N}^2\Gamma_1\left(x_N+y_N+\frac{\tau}{\delta},\frac{t-\tau}{\epsilon}\right)
	\\
 	& =-\partial_\tau f(x,y,t,\tau)
	\end{align*}
for $(x,y,t,\tau)\in E$. 
This, together with \eqref{eq:1.3} and the fifth relation
in \eqref{eq:3.2}, implies that 
\begin{equation}
\label{eq:3.4}
	\begin{aligned}
 	& \left(\partial_t-\frac{k}{\delta}\Delta_x'-\frac{1}{\delta}\partial_{x_N}\right)\int_0^t f(x,y,t,\tau)\,\dee\tau
	\\
	& =-\int_0^t \partial_\tau f(x,y,t,\tau)\,\dee\tau
	\\
 	& =f(x,y,t,0)-\lim_{\tau\to t^-}f(x,y,t,\tau)=f(x,y,t,0)
	\\
 	& =-2\Gamma_{N-1}\left(x'-y',\frac{t}{\epsilon}\right)
	\partial_{x_N}\Gamma_1\left(x_N+y_N,\frac{t}{\epsilon}\right)
	=\partial_{x_N}G_0\left(x,y,\frac{t}{\epsilon}\right)
	\end{aligned}
\end{equation}
for $(x,y,t)\in\partial\Omega\times \overline{\Omega}\times(0,\infty)$. 
Since $G_0(x,y,t)=0$ for $(x,y,t)\in\partial\Omega\times \overline{\Omega}\times(0,\infty)$, 
by the second relation in \eqref{eq:3.2} and \eqref{eq:3.4}, we obtain 
\begin{equation}
\label{eq:3.5}
	\begin{aligned}
 	& \delta\partial_t G(x,y,t)-k\Delta_x'G(x,y,t)-\partial_{x_N}G(x,y,t)
	\\
 	& =-\partial_{x_N}G_0\left(x,y,\frac{t}{\epsilon}\right)
	+\left(\partial_t-\frac{k}{\delta}\Delta_x'-\frac{1}{\delta}\partial_{x_N}\right)\int_0^t f(x,y,t,\tau)\,\dee\tau
	=0
 	\end{aligned}
\end{equation}
for $(x,y,t)\in \partial\Omega\times \overline{\Omega}\times(0,\infty)$. 
Therefore, by \eqref{eq:3.3} and \eqref{eq:3.5}, we obtain assertion~(2). 

Next, we prove assertion~(3). 
Let $(x,t)\in \overline{\Omega}\times(0,\infty)$.
By property~($\Gamma$1), 
\eqref{eq:1.3}{,} and the fact that $\Gamma_1(\xi,t)=\Gamma_1(-\xi,t)$ for $(\xi,t)\in{\mathbb R}\times(0,\infty)$, 
we have
\begin{equation}
\label{eq:3.6}
	\begin{aligned}
 	& \int_\Omega G_0\left(x,y,\frac{t}{\epsilon}\right)\,\dee y
	+\int_{\partial\Omega} G_0\left(x,y,\frac{t}{\epsilon}\right)\,\dee \sigma(y)
	\\
 	& =\int_0^\infty\left(\Gamma_1\left(x_N-y_N,\frac{t}{\epsilon}\right)
	-\Gamma_1\left(x_N+y_N,\frac{t}{\epsilon}\right)\right)\,\dee y_N
 	=\int_{-x_N}^{x_N}\Gamma_1\left(\xi,\frac{t}{\epsilon}\right)\,\dee\xi.
 	\end{aligned}
\end{equation}
It follows from property ($\Gamma$1) and \eqref{eq:1.8} that
\begin{equation}
\label{eq:3.7}
	\begin{aligned}
 	\frac{1}{\delta}\int_{\Omega}H(x,y,t)\,\dee y
  	& =-\frac{2}{\delta}\int_0^t\int_0^\infty 
	\partial_{y_N}\Gamma_1\left(x_N+y_N+\frac{\tau}{\delta},\frac{t-\tau}{\epsilon}\right)\,\dee y_N\,\dee\tau
	\\
  	& =\frac{2}{\delta}\int_0^t \Gamma_1\left(x_N+\frac{\tau}{\delta},\frac{t-\tau}{\epsilon}\right)\,\dee\tau
	\\
  	& =\frac{2}{\delta}\int_0^t \left(\frac{4\pi(t-\tau)}{\epsilon}\right)^{-\frac{1}{2}}
  	\exp\left(-\frac{\epsilon(x_N+\tau/\delta)^2}{4(t-\tau)}\right)\,\dee\tau
  	\le \frac{C\epsilon^{\frac{1}{2}}}{\delta}t^{\frac{1}{2}}
	\end{aligned}
\end{equation}
and
\begin{equation}
\label{eq:3.8}
	\begin{aligned}
	& \frac{1}{\epsilon}\int_{\partial\Omega}H(x,y,t)\,\dee\sigma(y)\\
 	& =-\frac{2}{\epsilon}\int_0^t
	\partial_{x_N}\Gamma_1\left(x_N+\frac{\tau}{\delta},\frac{t-\tau}{\epsilon}\right)\,\dee\tau
	\\
 	& =\int_0^t\left(\frac{4\pi(t-\tau)}{\epsilon}\right)^{-\frac{1}{2}}\frac{x_N+\tau/\delta}{t-\tau}
	\exp\left(-\frac{\epsilon(x_N+\tau/\delta)^2}{4(t-\tau)}\right)\,\dee\tau.
	\end{aligned}
\end{equation}
Since 
$$
	t^{-\frac{1}{2}}\dee\xi
	=\left(\frac{1}{\delta}(t-\tau)^{-\frac{1}{2}}+\frac{1}{2}(t-\tau)^{-\frac{3}{2}}
	\left(x_N+\frac{\tau}{\delta}\right)\right)\,\dee\tau
	\quad
	\mbox{for}
	\quad
	\xi:=\left(\frac{t}{t-\tau}\right)^{\frac{1}{2}}\left(x_N+\frac{\tau}{\delta}\right),
$$
by \eqref{eq:3.7} and \eqref{eq:3.8}, we have
\[
	\begin{aligned}
 	& \frac{1}{\delta}\int_\Omega H(x,y,t)\,\dee y+\frac{1}{\epsilon}\int_{\partial\Omega}H(x,y,t)\,\dee\sigma(y)
	\\
 	& =2\left(\frac{4\pi}{\epsilon}\right)^{-\frac{1}{2}}\int_0^t 
 	\left(\frac{1}{\delta}(t-\tau)^{-\frac{1}{2}}+\frac{1}{2}(t-\tau)^{-\frac{3}{2}}\left(x_N+\frac{\tau}{\delta}\right)\right)
 	\exp\left(-\frac{\epsilon(x_N+\tau/\delta)^2}{4(t-\tau)}\right)\,\dee\tau
	\\
 	& =2\left(\frac{4\pi t}{\epsilon}\right)^{-\frac{1}{2}}
 	\int_{x_N}^\infty\exp\left(-\frac{\xi^2}{4\epsilon t}\right)\,\dee\xi
 	=\left(\int_{-\infty}^{-x_N}+\int_{x_N}^\infty \right)\Gamma_1\left(\xi,\frac{t}{\epsilon}\right)\,\dee\xi.
	\end{aligned}
\]
This, together with \eqref{eq:1.8} and \eqref{eq:3.6}, implies that 
\[
	\begin{aligned}
	& \int_\Omega G(x,y,t)\,\dee y+\frac{\delta}{\epsilon}\int_{\partial\Omega} G(x,y,t)\,\dee \sigma(y)
	\\
	& =\int_\Omega G_0(x,y,t)\,\dee y+\frac{1}{\delta}\int_\Omega H(x,y,t)\,\dee y
	+\frac{1}{\epsilon}\int_{\partial\Omega}H(x,y,t)\,\dee\sigma(y)
	\\
	& =\int_{-\infty}^\infty \Gamma_1\left(\xi,\frac{t}{\epsilon}\right)\,\dee\xi=1,
	\quad (x,t)\in \overline{\Omega}\times(0,\infty).
	\end{aligned}
\]
Hence, assertion~(3) holds. 

Next, we prove assertion~(5).
Let $u$ be defined by \eqref{eq:1.10}, where $(\phi,\psi)\in BC(\Omega)\times BC(\partial\Omega)$. 
Assertion~(2) ensures that $u \in C^{2;1}(\overline{\Omega}\times(0,\infty))$ 
and that $u$ satisfies pointwise the first and second equations in problem~\eqref{eq:HDD} 
(resp.~problem~\eqref{eq:HD}) when $k>0$ (resp.~$k=0$).
Moreover, by assertion~(3), $u$ is bounded in $\overline{\Omega}\times[0,\infty)$. 
In addition, for any $x\in\Omega$ and $R>0$, 
since
\[
	\lim_{t\to 0^+}\int_{B^+_R(x)} G_0(x,y,t)\,\dee y=1,
\]
it follows from assertion~(3) that
\begin{equation}
\label{eq:3.9}
	\lim_{t\to 0^+}\left[\int_{\Omega\setminus B^+_R(x)}G_0\left(x,y,\frac{t}{\epsilon}\right)\,dy
	+\frac{1}{\delta}\int_\Omega H(x,y,t)\,\dee y+\frac{1}{\epsilon}\int_{\partial\Omega}H(x,y,t)\,\dee\sigma(y)\right]=0.
\end{equation}
Then, by assertion~(3) and \eqref{eq:3.9}, for any $x\in\Omega$ and $R>0$, 
we obtain
\[
	\begin{aligned}
	& \left|\int_\Omega G(x,y,t)\phi(y)\,\dee y-\phi(x)\right|
	\\
 	& =\left|\int_\Omega G(x,y,t)(\phi(y)-\phi(x))\,\dee y\right|
	+\frac{\delta}{\epsilon}\left|\phi(x)\int_{\partial\Omega} G(x,y,t)\,\dee \sigma(y)\right|
	\\
 	& \le \sup_{y\in B^+_R(x)}|\phi(y)-\phi(x)|\int_{B^+_R(x)} G(x,y,t)\,\dee y
	\\
 	& \qquad\quad
 	+2\|\phi\|_{L^\infty(\Omega)}\int_{\Omega\setminus B^+_R(x)} G(x,y,t)\,\dee y
	+\frac{1}{\epsilon}|\phi(x)|\int_{\partial\Omega} H(x,y,t)\,\dee \sigma(y)
	\\
	& \le \sup_{y\in B^+_R(x)}|\phi(y)-\phi(x)|+o(1) \quad\mbox{as}\quad t\to 0^+.
	\end{aligned}
\]
Furthermore, it follows from \eqref{eq:3.9} that
\[
	\left|\frac{1}{\epsilon}\int_{\partial\Omega}H(x,y,t)\psi(y)\,\dee\sigma(y)\right|
	\le\|\psi\|_{L^\infty(\partial\Omega)}
	\left|\frac{1}{\epsilon}\int_{\partial\Omega}H(x,y,t)\,\dee\sigma(y)\right|\to 0
	\quad\textrm{as}\quad t\to 0^+.
\]
Since $R$ is arbitrary, 
by the continuity of $\phi$, we deduce that 
\[
	\lim_{t\to 0^+}u(x,t)=\lim_{t\to 0^+}\int_\Omega G(x,y,t)\phi(y)\,\dee y=\phi(x),\quad x\in\Omega.
\]

In contrast, for any $x=(x',0)\in\partial\Omega$ and $R>0$, 
by \eqref{eq:3.6}, \eqref{eq:3.7}, \eqref{eq:3.8}, and assertion~(3), we obtain
\begin{equation}
\label{eq:3.10}
	\lim_{t\to 0^+}\int_\Omega G(x,y,t)\,\dee y=\frac{1}{\delta}\lim_{t\to 0^+}\int_\Omega H(x,y,t)\,\dee y=0
\end{equation}
and 
\begin{equation}
\label{eq:3.11}
	\begin{aligned}
 	& \frac{1}{\epsilon}\int_{\partial\Omega\setminus B'_R(x)} H(x,y,t)\,\dee\sigma(y)
	\\
 	& =-\frac{2}{\epsilon}
 	\int_0^t \left(\int_{{\mathbb R}^{N-1}\setminus B'_R(0)}
	\Gamma_{N-1}\left(y',\frac{t-\tau}{\epsilon}+\frac{k}{\delta}\tau\right)\,\dee y'\right)
	\partial_{x_N}\Gamma_1\left(x_N+\frac{\tau}{\delta},\frac{t-\tau}{\epsilon}\right)\,\dee\tau\,\biggr|_{x_N=0}
	\\
 	& =\frac{o(1)}{\epsilon}
 	\int_0^t \partial_{x_N}\Gamma_1\left(x_N+\frac{\tau}{\delta},\frac{t-\tau}{\epsilon}\right)\,\dee\tau\,\biggr|_{x_N=0}
 	=\frac{o(1)}{\epsilon}\int_{\partial\Omega}H(x,y,t)\,\dee\sigma(y)
 	=o(1) 
	\end{aligned}
\end{equation}
as $t\to 0^+$.
Then, by assertion~(3), \eqref{eq:3.10}, and \eqref{eq:3.11},
for any $x=(x',0)\in\partial\Omega$ and $R>0$, 
we have
\[
	\begin{aligned}
 	& \left|\frac{\delta}{\epsilon}\int_{\partial\Omega} G(x,y,t)\psi(y)\,\dee \sigma(y)-\psi(x)\right|
	\\
 	& \le|\psi(x)|\int_\Omega G(x,y,t)\,\dee y
	+\frac{1}{\epsilon}\left|\int_{\partial\Omega} H(x,y,t)(\psi(y)-\psi(x))\,\dee\sigma(y)\right|
	\\
 	& \le|\psi(x)|\int_\Omega G(x,y,t)\,\dee y
 	+\frac{2}{\epsilon}\|\psi\|_{L^\infty(\Omega)}\int_{\partial\Omega\setminus B'_R(x)} H(x,y,t)\,\dee\sigma(y)
	\\
 	& \qquad\quad
 	+\frac{1}{\epsilon}\sup_{y\in B'_R(x)}|\psi(x)-\psi(y)|\int_{B'_R(x)} H(x,y,t)\,\dee\sigma(y)
	\\
 	& \le \frac{1}{\epsilon}\sup_{y\in B'_R(x)}|\psi(x)-\psi(y)|+o(1)\quad\mbox{as}\quad t\to 0^+.
	\end{aligned}
\]
Furthermore, it follows from \eqref{eq:3.10} that
\[
	\left|\int_{\Omega}G(x,y,t)\phi(y)\,\dee y\right|\le\frac{1}{\delta}\|\phi\|_{L^\infty(\Omega)}\int_{\Omega}H(x,y,t)\,\dee y
	\to 0\quad\textrm{as}\quad t\to 0^+
\]
for $x\in\partial\Omega$.
Since $R$ is arbitrary, 
by the continuity of $\psi$, we deduce that 
\[
	\lim_{t\to 0^+}u(x,t)
	=\frac{\delta}{\epsilon}\lim_{t\to 0^+}\int_{\partial\Omega} G(x,y,t)\psi(y)\,\dee y
	=\psi(x),\quad x\in\partial\Omega.
\]
Then, combining the above arguments,
we conclude that $u$ is a bounded classical solution to problem~\eqref{eq:HDD} (resp.~problem~\eqref{eq:HD})
if $k>0$ (resp.~$k=0$).
Since the uniqueness of bounded classical solutions to problems~\eqref{eq:HDD} and \eqref{eq:HD} 
has been established in Proposition~\ref{Proposition:2.1}, assertion~(5) follows.
Moreover, assertion~(4) is an immediate consequence of assertion~(2) and~(5).
In fact, since $G(\cdot,y,s)\in BC(\overline\Omega)$ for any fixed $y\in\overline\Omega$ and $s>0$,
by assertion~(5) we see that
\[
	v(x,t):=\int_\Omega G(x,z,t)G(z,y,s)\,\dee z
	+\frac{\delta}{\epsilon}\int_{\partial\Omega}G(x,z,t)G(z,y,s)\,\dee\sigma(z)
\]
is the unique bounded classical solution to problem~\eqref{eq:HDD} 
$($resp.~problem~\eqref{eq:HD}$)$ when $k>0$ $($resp.~$k=0$$)$
with $(\phi(x),\psi(x))=(G(x,y,s), G(x,y,s)|_{\partial\Omega})$.
On the other hand, it follows from assertion~(2) that 
$G(x,y,t+s)$ is also a bounded classical solution to problem~\eqref{eq:HDD} 
$($resp.~problem~\eqref{eq:HD}$)$ when $k>0$ $($resp.~$k=0$$)$
with $G(x,y,t+s)|_{t=0}=G(x,y,s)$ for $x\in\overline\Omega$.
These imply that $v(x,t)=G(x,y,t+s)$ for $(x,t)\in \overline\Omega\times(0,\infty)$,
and assertion~(4) holds.
Thus, the proof of Theorem~\ref{Theorem:1.1} is complete. 
$\Box$
\vspace{5pt}

Next, we prove Theorem~\ref{Theorem:1.2}.
Note that the proof is not applicable to the case $k=0$.
\vspace{5pt}
\newline
{\bf Proof of Theorem~\ref{Theorem:1.2}.}
Let $\epsilon>0$, $\delta>0$, and $k>0$ be fixed.
Let $\Lambda$, $\lambda$ be defined as in Theorem~\ref{Theorem:1.2}. 
Since
\[
	\frac{\lambda}{\epsilon\delta} t\le \frac{t-\tau}{\epsilon}+\frac{k}{\delta}\tau
	\le \frac{\Lambda}{\epsilon\delta} t,
	\quad 
	\tau\in(0,t),
\]
it follows that 
\[
	C^{-1}\Gamma_{N-1}\left(x'-y',\frac{\lambda}{\epsilon\delta} t\right)
	\le \Gamma_{N-1}\left(x'-y',\frac{t-\tau}{\epsilon}+\frac{k}{\delta}\tau\right)
	\le C\Gamma_{N-1}\left(x'-y',\frac{\Lambda}{\epsilon\delta} t\right)
\]
for $(x,y,t,\tau)\in E$. 
This, together with \eqref{eq:1.8} and \eqref{eq:3.1}, implies that 
\begin{equation}
\label{eq:3.12}
	C^{-1}\Gamma_{N-1}\left(x'-y',\frac{\lambda}{\epsilon\delta} t\right) J(x,y,t)
	\le H(x,y,t)
	\le  C\Gamma_{N-1}\left(x'-y',\frac{\Lambda}{\epsilon\delta} t\right) I(x,y,t)
\end{equation}
for $(x,y,t)\in D$, where 
\begin{equation}
\label{eq:3.13}
	I(x,y,t):=\int_0^t g(x,y,t,\tau)\,\dee\tau,\qquad 
	J(x,y,t):=\int_{t/2}^t g(x,y,t,\tau)\,\dee\tau.
\end{equation}
\underline{Step 1.} 
We prove \eqref{eq:1.11} for 
\[
	(x,y,t)\in D_1=\left\{(x,y,t)\in D\,:\,\epsilon(x_N+y_N)^2<6 t,\,\,\,t<\frac{12\delta^2}{\epsilon}\right\}.
\]
Setting $z_N:=x_N+y_N+t/\delta$,
we have  
\[
	\left(x_N+y_N+\frac{\tau}{\delta}\right)^2
	=\left(x_N+y_N+\frac{t}{\delta}-\frac{t-\tau}{\delta}\right)^2
	=z_N^2-\frac{2(t-\tau)}{\delta}z_N+\frac{(t-\tau)^2}{\delta^2},
\]
which, together with \eqref{eq:3.1} and \eqref{eq:3.13}, implies that
\begin{equation}
\label{eq:3.14}
\begin{aligned}
 	& I(x,y,t)
	\\
 	& =\epsilon\exp\left(\frac{\epsilon z_N}{2\delta}\right)
	\int_0^t\left(\frac{4\pi(t-\tau)}{\epsilon}\right)^{-\frac{1}{2}}\bigg(\frac{z_N}{t-\tau}-\frac{1}{\delta}\bigg)
 	\exp\left(-\frac{\epsilon z_N^2}{4(t-\tau)}-\frac{\epsilon(t-\tau)}{4\delta^2}\right)\,\dee\tau
	\\
 	& =\epsilon(4\pi)^{-\frac{1}{2}}\exp\left(\frac{\epsilon z_N}{2\delta}\right)\int_0^{t/(\epsilon z_N^2)}
 	\xi^{-\frac{3}{2}}\left(1-\frac{\epsilon z_N}{\delta}\xi\right)
 	\exp\left(-\frac{1}{4\xi}-\frac{\epsilon^2 z_N^2}{4\delta^2}\xi\right)\,\dee\xi.
\end{aligned}
\end{equation}
Similarly, we obtain 
\begin{equation}
\label{eq:3.15}
	J(x,y,t) 
	=\epsilon(4\pi)^{-\frac{1}{2}}\exp\left(\frac{\epsilon z_N}{2\delta}\right)\int_0^{t/(2\epsilon z_N^2)}
 	\xi^{-\frac{3}{2}}\left(1-\frac{\epsilon z_N}{\delta}\xi\right)
 	\exp\left(-\frac{1}{4\xi}-\frac{\epsilon^2 z_N^2}{4\delta^2}\xi\right)\,\dee\xi.
\end{equation}
It follows from the definition of $D_1$ that 
\begin{equation}
\label{eq:3.16}
	\frac{t}{\delta}\le z_N=x_N+y_N+\frac{t}{\delta}\le\sqrt{\frac{6t}{\epsilon}}+\frac{t}{\delta}\le C,\qquad
	\frac{\epsilon z_N^2}{t}\le C,
\end{equation}
for $(x,y,t)\in D_1$. 
By \eqref{eq:3.14} and \eqref{eq:3.16}, we obtain
\begin{equation}
\label{eq:3.17}
 	I(x,y,t)\le C\int_0^\infty\xi^{-\frac{3}{2}}\exp\left(-\frac{1}{4\xi}\right)\,\dee\xi\le C, \quad (x,y,t)\in D_1.
\end{equation}
Furthermore, it follows from \eqref{eq:3.16} that 
\[
	1-\frac{\epsilon z_N}{\delta}\xi
	\ge 1-\frac{t}{2\delta z_N}\ge\frac{1}{2}
	\quad\mbox{for}\quad\xi\in\bigg(0,\frac{t}{2\epsilon z_N^2}\bigg),
\]
which, together with \eqref{eq:3.15} and \eqref{eq:3.16}, implies that
\begin{equation}
\label{eq:3.18}
	\begin{aligned}
	J(x,y,t) 
	& \ge C\int_0^{t/(2\epsilon z_N^2)}\xi^{-\frac{3}{2}}
	\exp\left(-\frac{1}{4\xi}-\frac{\epsilon^2 z_N^2}{4\delta^2}\xi\right)\,\dee\xi
	\\
 	& \ge C\int_0^{t/(2\epsilon z_N^2)}\xi^{-\frac{3}{2}}\exp\left(-\frac{1}{4\xi}-\frac{\epsilon t}{8\delta^2}\right)\,\dee\xi
	\ge C, \quad (x,y,t)\in D_1.
	\end{aligned}
\end{equation}
Combining \eqref{eq:3.12}, \eqref{eq:3.13}, \eqref{eq:3.17}, and \eqref{eq:3.18}, 
we obtain \eqref{eq:1.11} for $(x,y,t)\in D_1$. 
\vspace{5pt}
\newline
\underline{Step 2.} 
We prove \eqref{eq:1.11} for 
\[
	(x,y,t)\in D_2=\left\{(x,y,t)\in D\,:\,\epsilon(x_N+y_N)^2<6 t,\,\,\,t\ge \frac{12\delta^2}{\epsilon}\right\}.
\]
For any $\tau\in(0,t)$, 
we find a unique $\tau_*\in(0,t/2]$ such that
\begin{equation}
\label{eq:3.19}
	t-\tau_*=\frac{\epsilon(x_N+y_N+\tau_*/\delta)^2}{6}.
\end{equation}
Indeed, setting
\[
	F(\tau):=\frac{\epsilon(x_N+y_N+\tau/\delta)^2}{6}-(t-\tau),
\]
we see that $F(0)<0$ by $(x,y,t)\in D_2$ and $F$ is strictly increasing in $[0,\infty)$.
Furthermore, since $x_N+y_N\ge0$, we have
\[
	F\left(\frac{t}{2}\right)
	\ge \frac{\epsilon t^2}{24\delta^2}-\frac{t}{2}=\frac{t}{2}\left(\frac{\epsilon t}{12\delta^2}-1\right)\ge0
\]
for $t\ge 12\delta^2/\epsilon$.
These, together with the intermediate value theorem, imply the unique existence of 
$\tau_*\in(0,t/2]$ satisfying \eqref{eq:3.19}. 
Then
\begin{equation}
\label{eq:3.20}
	\begin{aligned}
	&
	\frac{\epsilon(x_N+y_N+\tau/\delta)^2}{6}
	\le
	\frac{\epsilon(x_N+y_N+\tau_*/\delta)^2}{6}=t-\tau_*\le t-\tau\qquad \mbox{if}\quad\tau\in(0,\tau_*],
	\\
	&
	t-\tau\le t-\tau_*=\frac{(x_N+y_N+\tau_*/\delta)^2}{6}
	\le \frac{\epsilon(x_N+y_N+\tau/\delta)^2}{6}
	\qquad\,\,\mbox{if}\quad\tau\in[\tau_*,t).
	\end{aligned}
\end{equation}
On the other hand, it follows from \eqref{eq:3.1} that  
\[
	\partial_tg(x,y,t,\tau)
	=\left(-6+\frac{\epsilon(x_N+y_N+\tau/\delta)^2}{t-\tau}\right)
	\frac{\epsilon(x_N+y_N+\tau/\delta)}{4(t-\tau)^2}
	\Gamma_1\left(x_N+y_N+\frac{\tau}{\delta},\frac{t-\tau}{\epsilon}\right).
\]
Then we see that
\begin{equation} 
\label{eq:3.21}
	\begin{aligned}
 	& \mbox{ $g(x,y,t,\tau)$ is increasing for 
  	$\displaystyle{t\in\left(\tau,\tau+\frac{\epsilon(x_N+y_N+\tau/\delta)^2}{6}\right)}$};
	\\
 	& \mbox{$g(x,y,t,\tau)$ is decreasing for 
  	$\displaystyle{t\in\left(\tau+\frac{\epsilon(x_N+y_N+\tau/\delta)^2}{6},\infty\right)}$}.
	\end{aligned}
\end{equation}
This, together with \eqref{eq:3.1} and {\eqref{eq:3.20}}, {implies} that
\begin{equation}
\label{eq:3.22}
	\begin{aligned}
 	& g(x,y,t,\tau)\le g(x,y,\tau+t-\tau_*,\tau)
	\\
 	& =-2\partial_{x_N}\Gamma_1\left(x_N+y_N+\frac{\tau}{\delta},\frac{t-\tau_*}{\epsilon}\right)
 	=-2\delta\partial_\tau \Gamma_1\left(x_N+y_N+\frac{\tau}{\delta},\frac{t-\tau_*}{\epsilon}\right).
	\end{aligned}
\end{equation}
By \eqref{eq:3.13} and \eqref{eq:3.22}, we obtain
\begin{equation}
\label{eq:3.23}
	\begin{aligned}
	I(x,y,t)
	&
	\le -2\delta\int_0^t
	\partial_\tau\Gamma_1\left(x_N+y_N+\frac{\tau}{\delta},\frac{t-\tau_*}{\epsilon}\right)\, \dee\tau
	\\
	&
	\le 2\delta\Gamma_1\left(x_N+y_N,\frac{t-\tau_*}{\epsilon}\right)
	\le 2\delta\left(\frac{t}{t-\tau_*}\right)^{\frac{1}{2}}\Gamma_1\left(x_N+y_N,\frac{t}{\epsilon}\right)\\
	 & \le C\Gamma_1\left(x_N+y_N,\frac{t}{\epsilon}\right).
	\end{aligned}
\end{equation}
In contrast, it follows from \eqref{eq:3.1} that 
\[
	\begin{aligned}
	 g(x,y,t,\tau) 
	& 
	= \left(\frac{t}{t-\tau}\right)^{\frac{3}{2}}
	\left(\frac{4\pi t}{\epsilon}\right)^{-\frac{1}{2}}
	\frac{\epsilon(x_N+y_N+\tau/\delta)}{t}\exp\left(-\frac{\epsilon(x_N+y_N+\tau/\delta)^2}{4(t-\tau)}\right)
	\\
	& 
	\ge \left(\frac{4\pi t}{\epsilon}\right)^{-\frac{1}{2}}
	\frac{\epsilon(x_N+y_N+\tau/\delta)}{t}\exp\left(-\frac{\epsilon(x_N+y_N+\tau/\delta)^2}{2t}\right)
	\\
	 & = -C\partial_{x_N}\Gamma_1\left(x_N+y_N+\frac{\tau}{\delta},\frac{t}{2\epsilon}\right)
	 = -C\delta\partial_\tau\Gamma_1\left(x_N+y_N+\frac{\tau}{\delta},\frac{t}{2\epsilon}\right)
	\end{aligned}
\]
for $0<\tau\le t/2$. 
This, together with \eqref{eq:3.13}, implies that
\begin{equation}
\label{eq:3.24}
	\begin{aligned}
	J(x,y,t)
	&
	\ge
	-C\delta\int_0^{t/2}\partial_\tau
	\Gamma_1\left(x_N+y_N+\frac{\tau}{\delta},\frac{t}{2\epsilon}\right)\, \dee\tau
	\\
	&
	=C\delta
	\bigg(\Gamma_1\left(x_N+y_N,\frac{t}{2\epsilon}\right)
	-\Gamma_1\left(x_N+y_N+\frac{t}{2\delta},\frac{t}{2\epsilon}\right)\bigg)
	\\
	&
	=C\delta
	\left(1-\exp\left(-\epsilon\frac{4\delta(x_N+y_N)+t}{8\delta^2}\right)\right)
	\Gamma_1\left(x_N+y_N,\frac{t}{2\epsilon}\right).
	\end{aligned}
\end{equation}
Since $x_N+y_N\ge0$ and $t\ge 12\delta^2/\epsilon$ by $(x,y,t)\in D_2$, 
we observe from \eqref{eq:3.24}  that 
\[
	J(x,y,t)\ge C\Gamma_1\left(x_N+y_N,\frac{t}{2\epsilon}\right).
\]
This, together with \eqref{eq:3.12} and \eqref{eq:3.23}, implies \eqref{eq:1.11} for $(x,y,t)\in D_2$.
\vspace{5pt}
\newline
\underline{Step 3.} 
Let $(x,y,t)\in D_3\cup D_4$, where 
\[
	\begin{aligned}
 	& D_3=\left\{(x,y,t)\in D\,:\,\epsilon(x_N+y_N)^2\ge 6t,\,\,\, x_N+y_N+\frac{t}{\delta}
	<\frac{\delta}{\epsilon}\right\},\\
 	&D_4=\left\{(x,y,t)\in D\,:\,\epsilon(x_N+y_N)^2\ge 6t,\,\,\, x_N+y_N+\frac{t}{\delta}
	\ge \frac{\delta}{\epsilon}\right\}.
	\end{aligned}
\]
For any $\tau\in[0,t)$, since 
\[
	t\le \tau+t\le \tau+\frac{\epsilon(x_N+y_N)^2}{6}\le  \tau+\frac{\epsilon(x_N+y_N+\tau/\delta)^2}{6},
\]
by \eqref{eq:3.1} and \eqref{eq:3.21}, we have
\[
	\begin{aligned}
	g(x,y,t,\tau) & \le g(x,y,\tau+t,\tau)
	\\
 	& =-2\partial_{x_N}\Gamma_1\left(x_N+y_N+\frac{\tau}{\delta},\frac{t}{\epsilon}\right)
	=-2\delta\partial_\tau\Gamma_1\left(x_N+y_N+\frac{\tau}{\delta},\frac{t}{\epsilon}\right).
	\end{aligned}
\]
Then we obtain
\[
	\begin{aligned}
	I(x,y,t)
	&
	\le -2\delta
	\int_0^t \partial_\tau\Gamma_1\left(x_N+y_N+\frac{\tau}{\delta},\frac{t}{\epsilon}\right)\, \dee\tau
	\\
	& 
	\le 2\delta\left(\Gamma_1\left(x_N+y_N,\frac{t}{\epsilon}\right)
	-\Gamma_1\left(x_N+y_N+\frac{t}{\delta},\frac{t}{\epsilon}\right)\right)
	\\
	& 
	=2\delta\left(1-\exp\left(-\epsilon\frac{2\delta(x_N+y_N)+t}{4\delta^2}\right)\right)
	\Gamma_1\left(x_N+y_N,\frac{t}{\epsilon}\right)
	\\
	& 
	\le
	\begin{dcases}
	C\displaystyle{\left(x_N+y_N+\frac{t}{\delta}\right)\Gamma_1\left(x_N+y_N,\frac{t}{\epsilon}\right)}
	& \mbox{if}\quad x_N+y_N+\displaystyle{\frac{t}{\delta}}<\frac{\delta}{\epsilon},
	\vspace{5pt}\\
	2\delta\displaystyle{\Gamma_1\left(x_N+y_N,\frac{t}{\epsilon}\right)} 
	& \mbox{if}\quad x_N+y_N+\displaystyle{\frac{t}{\delta}}\ge \frac{\delta}{\epsilon}.
	\end{dcases}
	\end{aligned}
\]
Furthermore, 
by the same arguments as in \eqref{eq:3.24}, 
we have
\[
	\begin{aligned}
	J(x,y,t)
	 & \ge C\delta
	\left(1-\exp\left(-\epsilon\frac{4\delta(x_N+y_N)+t}{8\delta^2}\right)\right)
	\Gamma_1\left(x_N+y_N,\frac{t}{2\epsilon}\right)
	\\
 	& 
	\ge\left\{
	\begin{array}{ll}
	\displaystyle{C\left(x_N+y_N+\frac{t}{\delta}\right)\Gamma_1\left(x_N+y_N,\frac{t}{2\epsilon}\right) }
	& \mbox{if}\quad x_N+y_N+\displaystyle{\frac{t}{\delta}}<\frac{\delta}{\epsilon},
	\vspace{5pt}\\
	\displaystyle{C\Gamma_1\left(x_N+y_N,\frac{t}{2\epsilon}\right)} 
	& \mbox{if}\quad x_N+y_N+\displaystyle{\frac{t}{\delta}}\ge \frac{\delta}{\epsilon}.
	\end{array}
	\right.
	\end{aligned}
\]
Thus, \eqref{eq:1.11} holds for $(x,y,t)\in D_3\cup D_4$,
and the proof of Theorem~\ref{Theorem:1.2} is complete. 
$\Box$
\vspace{5pt}

As an application of Theorems~\ref{Theorem:1.1} and \ref{Theorem:1.2}, 
we prove Theorem~\ref{Theorem:1.3}.
\vspace{3pt}
\newline
{\bf Proof of Theorem~\ref{Theorem:1.3}.}
Let $(\phi,\psi)\in L^p(\Omega)\times L^p(\partial\Omega)$, where $p\in[1,\infty]$. 
Assertion~(4) immediately follows from Theorem~\ref{Theorem:1.1}\,(4).

We prove assertions~(1) and~(3).
By \cite{IKK25}*{Corollary~1.3}, it suffices to consider the case $k>0$.
By Theorem~\ref{Theorem:1.2} and \eqref{eq:1.3}, 
we have 
\begin{equation}
\label{eq:3.25}
	0<G(x,y,t)\le Ct^{-\frac{N}{2}},
\quad (x,y,t)\in D.
\end{equation}
If $1\le p<\infty$, by H\"older's inequality, Theorem~\ref{Theorem:1.1}\,(3), and \eqref{eq:3.25}, we obtain
\begin{equation}
\label{eq:3.26}
	\begin{aligned}
 	& \sup_{x\in\overline{\Omega}}|[\G(t)(\phi,\psi)](x)| 
	\\
 	& \le \sup_{x\in\overline{\Omega}}
 	\left(\int_\Omega G(x,y,t)|\phi(y)|^p\,\dee y\right)^{\frac{1}{p}}
 	+\frac{\delta}{\epsilon}\sup_{x\in\overline{\Omega}}\left(\int_{\partial\Omega} 
 	G(x,y,t)|\psi(y)|^p\,\dee\sigma(y)\right)^{\frac{1}{p}} 
	\\
  	& \le Ct^{-\frac{N}{2p}}\|(\phi,\psi)\|_{L^p(\Omega)\times L^p(\partial\Omega)}<\infty,\quad t>0.
\end{aligned}
\end{equation}
If $p=\infty$, by Theorem~\ref{Theorem:1.1}\,(3), we have 
\begin{equation}
\label{eq:3.27}
	\sup_{x\in\overline{\Omega}}|[\G(t)(\phi,\psi)](x)|
	\le \|(\phi,\psi)\|_{L^\infty(\Omega)\times L^\infty(\partial\Omega)}<\infty,\quad t>0.
\end{equation}
Furthermore, we see that
\begin{equation}
\label{eq:3.28}
	\|\G(t)\|_{p\to p}\le 1,\quad t>0, 
\end{equation}
for $p\in[1,\infty]$. 
In fact, by \eqref{eq:3.27}, 
we obtain \eqref{eq:3.28} for $p=\infty$. 
Moreover, it follows from Theorem~\ref{Theorem:1.1}~(3) that
\[
	\begin{aligned}
	& 
	\|\G(t)(\phi,\psi)\|_{L^1(\Omega)\times L^1(\partial\Omega)}
	\\
 	& 
	\le \int_\Omega\int_\Omega G(x,y,t)|\phi(y)|\,\dee y\,\dee x
	+\frac{\delta}{\epsilon}\int_\Omega\int_{\partial\Omega}G(x,y,t)|\psi(y)|\,\dee\sigma(y)\,\dee x
	\\
 	& \qquad
 	+ \int_{\partial\Omega}\int_\Omega G(x,y,t)|\phi(y)|\,\dee y\,\dee\sigma(x)
	+\frac{\delta}{\epsilon}\int_{\partial\Omega}\int_{\partial\Omega}G(x,y,t)|\psi(y)|\,\dee\sigma(y)\,\dee\sigma(x)
	\\
 	& 
	\le\int_\Omega\left(\int_\Omega G(x,y,t)\,\dee x
	+\frac{\delta}{\epsilon}\int_{\partial\Omega}G(x,y,t)\,\dee\sigma(x)\right)|\phi(y)|\,\dee y
	\\
 	& \qquad
 	+\int_{\partial\Omega}\left(\int_\Omega G(x,y,t)\,\dee x
	+\frac{\delta}{\epsilon}\int_{\partial\Omega}G(x,y,t)\,\dee\sigma(x)\right)|\psi(y)|\,\dee \sigma(y)\\
 	& 
	=\|\phi\|_{L^1(\Omega)}+\|\psi\|_{L^1(\partial\Omega)}
	=\|(\phi,\psi)\|_{L^1(\Omega)\times L^1(\partial\Omega)},
	\end{aligned}
\]
which implies that \eqref{eq:3.28} holds for $p=1$. 
Then the Riesz--Thorin interpolation theorem implies \eqref{eq:3.28} for $p\in[1,\infty]$. 
Applying the standard theory for convolutions, together with Theorem~\ref{Theorem:1.1} and \eqref{eq:3.28}, 
we see that $u$ is bounded and smooth on $\overline{\Omega}\times[T,\infty)$ 
for any $T>0$ and satisfies \eqref{eq:1.13} and \eqref{eq:1.14}.
Thus, assertion~(1) holds. 

On the other hand, 
it follows from \eqref{eq:3.26} and \eqref{eq:3.27} that
\[
	\|\G(t)(\phi,\psi)\|_{L^\infty(\Omega)}+\|\G(t)(\phi,\psi)\|_{L^\infty(\partial\Omega)}
	\le Ct^{-\frac{N}{2p}}\|(\phi,\psi)\|_{L^p(\Omega)\times L^p(\partial\Omega)}
\]
for $p\in[1,\infty]$ and $t>0$. 
This, together with \eqref{eq:3.28}, implies that
\[
	\begin{aligned}
 	& \|(\G(t)(\phi,\psi),\G(t)(\phi,\psi)|_{\partial\Omega})\|_{L^q(\Omega)\times L^q(\partial\Omega)}
	\\
 	 & \le\|\G(t)(\phi,\psi)\|_{L^q(\Omega)}+\|\G(t)(\phi,\psi)\|_{L^q(\partial\Omega)}
	\\
 	& 
 	\le \|\G(t)(\phi,\psi)\|_{L^\infty(\Omega)}^{1-\frac{p}{q}}\|\G(t)(\phi,\psi)\|_{L^p(\Omega)}^{\frac{p}{q}}
 	+\|\G(t)(\phi,\psi)\|_{L^\infty(\partial\Omega)}^{1-\frac{p}{q}}\|\G(t)(\phi,\psi)\|_{L^p(\partial\Omega)}^{\frac{p}{q}}
	\\
 	& 
 	\le Ct^{-\frac{N}{2}\left(\frac{1}{p}-\frac{1}{q}\right)}
 	\|(\phi,\psi)\|_{L^p(\Omega)\times L^p(\partial\Omega)}
	\end{aligned}
\]
for $p, q\in[1,\infty]$ with $p\le q$ and $t>0$. 
This implies that 
\begin{equation}
\label{eq:3.29}
	\|\G(t)\|_{p\to q}\le Ct^{-\frac{N}{2}\left(\frac{1}{p}-\frac{1}{q}\right)},
\quad t>0,
\end{equation}
for $p, q\in[1,\infty]$ with $p\le q$ and $t>0$. 

Define
\begin{equation}
\label{eq:3.30}
	\varphi(x,t):=-\partial_{x_N}\Gamma_N\left(x,\frac{t}{\epsilon}\right)
	=\frac{\epsilon x_N}{2t}\Gamma_N\left(x,\frac{t}{\epsilon}\right)\ge0
\end{equation}
for $(x,t)\in \overline{\Omega}\times(0,\infty)$. 
Since $\varphi(x,t)=0$ for $(x,t)\in \partial\Omega\times(0,\infty)$, 
for any $q\in[1,\infty]$, 
there exists $C>0$ such that 
\begin{equation}
\label{eq:3.31}
	C^{-1}t^{-\frac{N}{2}\left(1-\frac{1}{q}\right)-\frac{1}{2}}
	\le\|\varphi(t)\|_{L^q(\Omega)}
	=\|(\varphi(t),\varphi(t)|_{\partial\Omega})\|_{L^q(\Omega)\times L^q(\partial\Omega)}
	\le Ct^{-\frac{N}{2}\left(1-\frac{1}{q}\right)-\frac{1}{2}}
\end{equation}
for $t>0$. 
On the other hand, it follows from \eqref{eq:1.3} and property~($\Gamma$1) that 
\begin{equation}
\label{eq:3.32}
	\int_\Omega G_0\left(x,y,\frac{t_1}{\epsilon}\right)\varphi(y,t_2)\,\dee y=\varphi(x,t_1+t_2)
\end{equation}
for $x\in\overline{\Omega}$, $t_1>0$, and $t_2>0$.
Then, by \eqref{eq:3.30}, \eqref{eq:3.31}, and \eqref{eq:3.32}, we obtain
	\begin{align*}
	\|\G(t)(\varphi(t),\varphi(t)|_{\partial\Omega})\|_{L^q(\Omega)\times L^q(\partial\Omega)}
 	& \ge \left(\int_\Omega\left(\int_\Omega G_0\left(x,y,\frac{t}{\epsilon}\right)\varphi(y,t)\,\dee y\right)^q\,\dee x\right)^{1/q}
	\\
 	& =\|\varphi(2t)\|_{L^q(\Omega)}\ge Ct^{-\frac{N}{2}\left(1-\frac{1}{q}\right)-\frac{1}{2}}
	\end{align*}
for $t>0$ and $q\in[1,\infty]$.
This, together with \eqref{eq:3.31}, yields
\begin{equation}
\label{eq:3.33}
\begin{aligned}
	\|\G(t)\|_{p\to q}
	 & \ge \frac{\|(\G(t)(\varphi(t),\varphi(t)|_{\partial\Omega}),\G(t)(\varphi(t),\varphi(t)|_{\partial\Omega})|_{\partial\Omega})\|_{L^q(\Omega)\times L^q(\partial\Omega)}}
	 {\|(\varphi(t),\varphi(t)|_{\partial\Omega})\|_{L^p(\Omega)\times L^p(\partial\Omega)}}\\
	  \ge Ct^{-\frac{N}{2}\left(\frac{1}{p}-\frac{1}{q}\right)},\quad t>0, 
\end{aligned}
\end{equation}
for $p$, $q\in[1,\infty]$ with $p\le q$. 
Combining \eqref{eq:3.29} with \eqref{eq:3.33}, we obtain 
\[
	C^{-1}t^{-\frac{N}{2}\left(\frac{1}{p}-\frac{1}{q}\right)}
	\le \|\G(t)\|_{p\to q}\le Ct^{-\frac{N}{2}\left(\frac{1}{p}-\frac{1}{q}\right)},
	\quad t>0. 
\]
Furthermore, it follows from assertion~(4) that  
\[
	C\le \|\G(nt)\|_{p\to p}\le \|\G((n-1)t)\|_{p\to p}\|\G(t)\|_{p\to p}\le \|\G(t)\|_{p\to p}^n
\]
for $t>0$ and $n=2,3,\dots$, which implies $\|\G(t)\|_{p\to p}\ge 1$ for $t>0$.
This, together with \eqref{eq:3.28}, shows that $\|\G(t)\|_{p\to p}=1$ for $t>0$.
Therefore, assertion~(3) is proved. 

We prove assertion~(2). 
Replacing $\phi$ with $0$ and using Theorem~\ref{Theorem:1.1}\,(5), we obtain
\[
  	\lim_{t\to 0^+}\frac{1}{\epsilon}\int_{\partial\Omega}H(x,y,t)\psi(y)\,\dee\sigma(y)
	=\psi(x),\quad x\in\partial\Omega.
\]
Thus, it suffices to prove that
\begin{equation}
\label{eq:3.34}
	  \lim_{t\to 0^+}\int_\Omega H(x,y,t)\phi(y)\,\dee y=0,\quad x\in\partial\Omega,
\end{equation}
for $\phi\in L^p(\Omega)$ satisfying \eqref{eq:1.16}. 
If $k>0$, then, by Theorem~\ref{Theorem:1.2} and \eqref{eq:2.2}, we have
\begin{equation}
\label{eq:3.35}
	\begin{aligned}
  	& \left|\int_\Omega H(x,y,t)\phi(y)\,\dee y\right|
	\\
    & \le C\int_0^{(6t/\epsilon)^{1/2}}\int_{\mathbb{R}^{N-1}}\Gamma_{N-1}\left(x'-y',Ct\right)
	|\phi(y)|\,\dee y'\,\dee y_N\\
    &\qquad
    +C\int_{(6t/\epsilon)^\frac{1}{2}}^\infty\int_{\mathbb{R}^{N-1}} (y_N+t^\frac{1}{2})\Gamma_{N-1}\left(x'-y',Ct\right)
	\Gamma_1\left(y_N,\frac{t}{\epsilon}\right)|\phi(y)|\,\dee y'\,\dee y_N\\
    & \le Ct^{-\frac{N-1}{2p}}\int_0^{(6t/\epsilon)^{1/2}}\|\phi(\cdot,y_N)\|_{L^p(\mathbb{R}^{N-1})}\,\dee y_N\\
    & \qquad\
    +C\int_{(6t/\epsilon)^\frac{1}{2}}^\infty(y_N+t^\frac{1}{2})
	\Gamma_1\left(y_N,\frac{t}{\epsilon}\right)t^{-\frac{N-1}{2p}}
	\|\phi(\cdot,y_N)\|_{L^p(\mathbb{R}^{N-1})}\,\dee y_N\\
    & \le C\int_0^{(6t/\epsilon)^{1/2}}y_N^{-\frac{N-1}{p}}\|\phi(\cdot,y_N)\|_{L^p(\mathbb{R}^{N-1})}\,\dee y_N\\
    & \qquad
    -C\int_0^\infty \left(\frac{\eta^2}{t}\right)^{\frac{N-1}{2p}}
	\partial_{\eta}\left((\eta+t^\frac{1}{2})\Gamma_1\left(\eta,\frac{t}{\epsilon}\right)\right)\eta^{-\frac{N-1}{p}}
	\left(\int_0^\eta \|\phi(\cdot,y_N)\|_{L^p(\mathbb{R}^{N-1})}\,\dee y_N\right)\,\dee\eta\\
    & \le C\int_0^{(6t/\epsilon)^{1/2}}y_N^{-\frac{N-1}{p}}\|\phi(\cdot,y_N)\|_{L^p(\mathbb{R}^{N-1})}\,\dee y_N\\
    &\qquad
    +C\int_0^\infty \xi^\frac{N-1}{p}\left|\partial_\xi\left((\xi+1)\Gamma_1\left(\xi,\frac{1}{\epsilon}\right)\right)\right|
	  \biggr(\int_0^{t^\frac{1}{2}\xi} y_N^{-\frac{N-1}{p}}
	  \|\phi(\cdot,y_N)\|_{L^p(\mathbb{R}^{N-1})}\,\dee y_N\biggr)\,\dee\xi
	\end{aligned}
\end{equation}
for $(x,t)\in\partial\Omega\times(0,1)$.
This, together with \eqref{eq:1.16} and the Lebesgue dominated convergence theorem, implies \eqref{eq:3.34}. 
On the other hand, if $k=0$, we have
\[
	\begin{aligned}
	\left|\int_\Omega H(x,y,t)\phi(y)\,\dee y\right|
	& \le \left(\int_0^{(6t/\epsilon)^{1/2}}+\int_{(6t/\epsilon)^{1/2}}^\infty\right)
	\int_{\mathbb{R}^{N-1}}H(x,y,t)|\phi(y)|\,\dee y'\,\dee y_N\\
	&=:I_1(x,t)+I_2(x,t)
	\end{aligned}
\]
for $(x,t)\in\partial\Omega\times(0,\infty)$. 
By \cite{IKK25}*{Theorem~1.2} (see also Remark~\ref{Remark:1.3}) and \eqref{eq:1.16}, 
we obtain
\[
	\begin{aligned}
	I_1(x,t)
	&\le C\int_0^{(6t/\epsilon)^{1/2}}\int_{\mathbb{R}^{N-1}}
	P\left(x'-y',y_N+\frac{t}{\delta}\right)|\phi(y)|\,\dee y'\,\dee y_N\\
	&\le C\int_0^{(6t/\epsilon)^{1/2}}
	\left(y_N+\frac{t}{\delta}\right)^{-\frac{N-1}{p}}
	\|\phi(\cdot,y_N)\|_{L^p(\mathbb{R}^{N-1})}\,\dee y_N\\
	&\le C\int_0^{(6t/\epsilon)^{1/2}}
	y_N^{-\frac{N-1}{p}}\|\phi(\cdot,y_N)\|_{L^p(\mathbb{R}^{N-1})}\,\dee y_N
	\to 0
	\quad \text{as}\quad t\to 0^+.
	\end{aligned}
\]
Furthermore, by Theorem~\ref{Theorem:1.2}, we argue as in \eqref{eq:3.35} to obtain
\[
	\begin{aligned}
	 & I_2(x,t)
	=\int_{(6t/\epsilon)^{1/2}}^\infty\int_{\mathbb R^{N-1}}
	H(x,y,t)|\phi(y)|\,\dee y'\,\dee y_N\\
	&\le C\int_0^\infty\int_{\mathbb{R}^{N-1}}
	(y_N+t^\frac{1}{2})\Gamma_{N-1}\left(x'-y',Ct\right)
	\Gamma_1\left(y_N,\frac{t}{\epsilon}\right)|\phi(y)|\,\dee y'\,\dee y_N
	\to 0
	\quad \text{as}\quad t\to 0^+
	\end{aligned}
\]
for $x\in\partial\Omega$. 
These estimates imply that \eqref{eq:3.34} holds in the case $k=0$. 
Therefore, \eqref{eq:1.15} is proved for both cases $k>0$ and $k=0$, 
and assertion~(2) holds. 

It remains to prove assertion~(5). 
Let $\epsilon>0$ and $\delta>0$ be fixed. 
Define 
\begin{equation}
\label{eq:3.36}
	\begin{aligned}
 	& J_0(x,y,t,\tau)
	\\
  	& :=-2\left\{\Gamma_{N-1}\left(x'-y',\frac{t-\tau}{\epsilon}+\frac{k}{\delta}\tau\right)
	-\Gamma_{N-1}\left(x'-y',\frac{t-\tau}{\epsilon}\right)\right\}
	\\
 	& \qquad\times
	\partial_{x_N}\Gamma_1\left(x_N+y_N+\frac{\tau}{\delta},\frac{t-\tau}{\epsilon}\right)
	\\
 	& =-2\left(\int_0^{k\tau/\delta}\partial_\xi \Gamma_{N-1}\left(x'-y',\frac{t-\tau}{\epsilon}+\xi\right)\,\dee\xi\right)
	\partial_{x_N}\Gamma_1\left(x_N+y_N+\frac{\tau}{\delta},\frac{t-\tau}{\epsilon}\right)
	\end{aligned}
\end{equation}
for $(x,y,t,\tau)\in E$ and $k>0$. 
Let $T>0$ and $R>0$. 
Then, by \eqref{eq:2.1} and \eqref{eq:2.2}, we have
\begin{equation}
\label{eq:3.37}
	\begin{aligned}
 	& \frac{1}{\epsilon}\left|\int_0^t\int_{\partial\Omega}J_0(x,y,t,{\tau})\,\dee\sigma(y)\,\dee\tau\right|
	\le \frac{Ckt}{\epsilon \delta}\int_0^t \left(\frac{t-\tau}{\epsilon}\right)^{-2}
	\exp\left(-\frac{\epsilon(x_N+\tau/\delta)^2}{8(t-\tau)}\right)\,\dee\tau
	\\
     	&\le Ckt\left(\int_0^{t/2} \left(\frac{t}{2\epsilon}\right)^{-2}\,\dee\tau
	+\int_{t/2}^t \left(\frac{t-\tau}{\epsilon}\right)^{-2}
	\exp\left(-\frac{\epsilon R^2}{8(t-\tau)}\right)\,\dee\tau\right)\le Ck(1+t^2)\le Ck
	\end{aligned}
\end{equation}
for $(x,t)\in\overline{\Omega}\times(0,T]$ with $x_N+t/\delta\ge R$.
Similarly, by \eqref{eq:2.2} and \eqref{eq:3.36}, we obtain
\begin{equation}
\label{eq:3.38}
	\begin{aligned}
	&\frac{1}{\delta}\left|\int_0^t\int_\Omega J_0(x,y,t,\tau)\,\dee y\,\dee\tau\right|
 	 \le \frac{Ckt}{\delta^2}\int_0^t \left(\frac{t-\tau}{\epsilon}\right)^{-\frac{3}{2}}
	\exp\left(-\frac{\epsilon(x_N+\tau/\delta)^2}{16(t-\tau)}\right)\,\dee\tau
	\\
	&\le \frac{Ckt}{\delta^2}\left(\int_0^{t/2} \left(\frac{t}{2\epsilon}\right)^{-\frac{3}{2}}\,\dee\tau
	+\int_{t/2}^t \left(\frac{t-\tau}{\epsilon}\right)^{-\frac{3}{2}}
	\exp\left(-\frac{\epsilon R^2}{16(t-\tau)}\right)\,\dee\tau\right)\\
    	&\le Ck(t^{1/2}+t^2)\le Ck
	\end{aligned}
\end{equation}
for $(x,t)\in\overline{\Omega}\times(0,T]$ with $x_N+t/\delta\ge R$. 
By \eqref{eq:3.36}, \eqref{eq:3.37}, and \eqref{eq:3.38}, 
we obtain
\[
	\begin{aligned}
 	& |\uHDD{\epsilon}{\delta}{k}(x,t)-\uHD{\epsilon}{\delta}(x,t)|
	\\
 	& \le\frac{\|\phi\|_{L^\infty(\Omega)}}{\delta}\left|\int_0^t\int_\Omega J_0(x,y,t)\,\dee y\,\dee\tau\right|
	+\frac{\|\psi\|_{L^\infty(\partial\Omega)}}{\epsilon}
	\left|\int_0^t\int_{\partial\Omega} J_0(x,y,t)\,\dee\sigma(y)\,\dee\tau\right|
	\le Ck
	\end{aligned}
\]
for $(x,t)\in\overline{\Omega}\times(0,T]$ with $x_N+t/\delta\ge R$. 
Hence, assertion~(4) holds. This completes the proof of Theorem~\ref{Theorem:1.3}.
$\Box$
\begin{remark}
\label{Remark:3.1}
The convergence rate in Theorem~{\rm\ref{Theorem:1.3}}~{\rm (5)} is optimal.
Indeed, consider the case when 
$\phi=0$ in $\Omega$ and $\psi=\chi_{\partial\Omega\setminus B'_1(0)}$ on $\partial\Omega$.
Let $\epsilon>0$ and $\delta>0$ be fixed.
Then it follows from~\eqref{eq:1.12} that
\[
	\begin{aligned}
	 & \uHDD{\epsilon}{\delta}{k}(x,t)-\uHD{\epsilon}{\delta}(x,t)
	 \\
 	& =-\frac{2}{\epsilon}\int_{\partial\Omega}\int_0^t
	\left\{\Gamma_{N-1}\!\left(x'-y',\frac{t-\tau}{\epsilon}+\frac{k}{\delta}\tau\right)
	-\Gamma_{N-1}\!\left(x'-y',\frac{t-\tau}{\epsilon}\right)\right\}
	\\
 	& \qquad\qquad\quad
 	\times\partial_{x_N}\Gamma_1\!\left(x_N+\frac{\tau}{\delta},
 	\frac{t-\tau}{\epsilon}\right)\psi(y)\,\dee\tau\,\dee\sigma(y)
	\\
 	& =-\frac{2}{\epsilon}\int_{\partial\Omega\setminus B'_1(0)}\int_0^t\int_0^{k\tau/\delta}
  	\partial_{\xi}\Gamma_{N-1}\!\left(y',\frac{t-\tau}{\epsilon}+\xi\right)
    	\partial_{x_N}\Gamma_1\!\left(x_N+\frac{\tau}{\delta},\frac{t-\tau}{\epsilon}\right)
    	\,\dee\xi\,\dee\tau\,\dee\sigma(y)
	\end{aligned}
\]
for $(x,t)\in\overline{\Omega}\times(0,\infty)$.
In contrast, there exists $A>0$ such that
\begin{equation}
\label{eq:3.39}
	\begin{aligned}
	\partial_t\Gamma_{N-1}(x',t)
 	& =(4\pi t)^{-\frac{N-1}{2}}\left(-\frac{N-2}{2t}+\frac{|x'|^2}{4t^2}\right)
	\exp\left(-\frac{|x'|^2}{4t}\right)\\
 	& \ge Ct^{-\frac{N+3}{2}}|x'|^2\exp\!\left(-\frac{|x'|^2}{4t}\right)>0
	\end{aligned}
\end{equation}
for $x'\in{\mathbb R}^{N-1}\setminus B'_{At^{1/2}}(0)$ and $t>0$.
Let $T>0$ be such that
\[
	A\left(\frac{1}{\epsilon}+\frac{1}{\delta}\right)^{\frac{1}{2}}T^{\frac{1}{2}}<1.
\]
Then, for any $t\in(0,T)$, $\tau\in(0,t)$,  $\xi\in(0,k\tau/\delta)$, and $k\in(0,1]$,
it follows that
\[
	\begin{aligned}
 	& A\left(\frac{t-\tau}{\epsilon}+\xi\right)^{\frac{1}{2}}
 	\le A\left(\frac{t}{\epsilon}+\frac{t}{\delta}\right)^{\frac{1}{2}}<1,\\
 	& -\frac{2}{\epsilon}
	\partial_{\xi}\Gamma_{N-1}\!\left(y',\frac{t-\tau}{\epsilon}+\xi\right)
 	\partial_{x_N}\Gamma_1\!\left(x_N+\frac{\tau}{\delta},\frac{t-\tau}{\epsilon}\right)
 	\ge 0,
 	\quad y'\in \partial\Omega\setminus B'_1(0).
	\end{aligned}
\]
For any $R>0$, these imply that
\[
	\begin{aligned}
 	& \uHDD{\epsilon}{\delta}{k}(x,t)-\uHD{\epsilon}{\delta}(x,t)\\
 	& \ge C\int_{\partial\Omega\setminus B'_1(0)}\int_{t/4}^{t/2}
 	\frac{k\tau}{\delta}
 	t^{-\frac{N+3}{2}}|y'|^2\\
	 & \qquad\times
	 \left(\frac{t-\tau}{\epsilon}\right)^{-\frac{1}{2}}
	\frac{x_N+\tau/\delta}{(t-\tau)/\epsilon}
 	\exp\!\left(-\epsilon\frac{|y'|^2+(x_N+\tau/\delta)^2}{4(t-\tau)}\right)
 	\,\dee\tau\,\dee\sigma(y)\\
 	& \ge Ckt^{-\frac{N+2}{2}}R
 	\int_{\partial\Omega\setminus B'_1(0)}|y'|^2
 	\exp\!\left(-\epsilon\frac{|y'|^2+(x_N+(t/2)\delta)^2}{8t}\right)\,\dee\sigma(y)
 	\ge Ck
	\end{aligned}
\]
for all $(x,t)\in\overline{\Omega}\times(0,T)$ with $R\le x_N+t/\delta\le 2R$
and all $k\in(0,1]$.
Hence, the convergence rate in Theorem~{\rm\ref{Theorem:1.3}}~{\rm (5)} is optimal.
\end{remark}
%
\section{Diffusion limits with respect to $\epsilon$}
\label{section:4}
In this section, we study the diffusion limits of solutions to problems~\eqref{eq:HDD} 
and \eqref{eq:HD} with respect to $\epsilon$. 
As a byproduct, we derive an explicit representation of the fundamental solution to problem~\eqref{eq:LDD} 
(the Laplace equation in $\Omega$ with a diffusive dynamical boundary condition), 
introduced in Section~1.
A classical solution to problem~\eqref{eq:LDD} is a function~$u$ defined 
on $\overline\Omega\times[0,\infty)$
such that 
\[
	u\in C^{2;1}(\overline{\Omega}\times(0,\infty))\cap C(Q)
\]
and $u$ satisfies the equations in \eqref{eq:LDD} pointwise.
Here
\[
	Q:=(\overline{\Omega}\times[0,\infty))\setminus(\Omega\times\{0\}). 
\]

We first investigate the diffusion limits of solutions to problems~\eqref{eq:HDD} and \eqref{eq:HD} 
in the limit $\epsilon\to 0^+$, 
and obtain the following theorem. 
Define 
\begin{equation}
\label{eq:4.1}
	\GLDD{\delta}{k}(x,y,t)
	:=-2\int_0^\infty \Gamma_{N-1}\left(x'-y',\frac{k}{\delta}t+\tau\right)
	\partial_{x_N}\Gamma_1\left(x_N+y_N+\frac{t}{\delta},\tau\right)\,\dee\tau
\end{equation}
for $(x,y,t)\in\overline{D}$ with $x_N+y_N+t/\delta>0$, where $\delta>0$ and $k\ge 0$. 
\begin{theorem}
\label{Theorem:4.1}
Let $\delta>0$ and $k>0$ be fixed.
Let $(\phi,\psi)\in L^\infty(\Omega)\times L^\infty(\partial\Omega)$. 
Then, for any $L>0$ and any compact set $I\subset(0,\infty)$, 
\begin{equation}
\label{eq:4.2}
	\begin{aligned}
 	& \sup_{(x,t)\in\Omega_L\times I}
	\left|\uHDD{\epsilon}{\delta}{k}(x,t)-\uLDD{\delta}{k}(x,t)\right|
	=O\left(\epsilon^{\frac{1}{2}}\right),
	\\
 	& \sup_{(x,t)\in\Omega_L\times I}\left|\,\uHD{\epsilon}{\delta}(x,t)-\uLD{\delta}(x,t)\right|=O\left(\epsilon^{\frac{1}{2}}\right),
	\end{aligned}
\end{equation}
as $\epsilon\to 0^+$, where $\Omega_L:=\{x=(x',x_N)\in\overline{\Omega}:\,0\le x_N\le L\}$. 
Here, $\uLDD{\delta}{k}$ is a function on $\overline{\Omega}\times(0,\infty)$ given by 
\begin{equation}
\label{eq:4.3}
	\uLDD{\delta}{k}(x,t):=\int_{\partial\Omega}\GLDD{\delta}{k}(x,y,t)\psi(y)\,\dee\sigma(y),\quad (x,t)\in\overline{\Omega}\times(0,\infty),
\end{equation}
and $\uLD{\delta}$ is given in \eqref{eq:1.5}.
In particular, if $\psi\in BC(\partial\Omega)$, then $\uLDD{\delta}{k}\in C(Q)$
and is a bounded classical solution to problem~\eqref{eq:LDD}. 
Furthermore, for any $T>0$ and $R>0$,
\begin{equation}
\label{eq:4.4}
	\left|\uLDD{\delta}{k}(x,t)-\uLD{\delta}(x,t)\right|=O(k)\quad\mbox{as}\quad k\to 0^+
\end{equation}
uniformly for $(x,t)\in \overline{\Omega}\times(0,T]$ with $x_N+t/\delta\ge R$.
\end{theorem}
By Theorem~\ref{Theorem:4.1},
for any $\psi\in L^p(\partial\Omega)$ with $p\in[1,\infty]$,
the function $\uLDD{\delta}{k}$ can be interpreted as a solution to problem~\eqref{eq:LDD}.

For the proof of Theorem~\ref{Theorem:4.1}, 
we first state the following lemma, 
which shows that $\GLDD{\delta}{k}$ is the fundamental solution to problem~\eqref{eq:LDD} (resp.~problem~\eqref{eq:LD}) when $k>0$ (resp.~$k=0$).
\begin{lemma}
\label{Lemma:4.1}
Fix $k\ge 0$ and $\delta>0$. 
\begin{enumerate}[label={\rm(\arabic*)}]
	\item
    	$\GLDD{\delta}{k}(x,y,t)=\GLDD{\delta}{k}(y,x,t)$ for $(x,y,t)\in \overline{D}$ with $x_N+y_N+t/\delta>0$.
  	\item
    	For any fixed $y\in\partial\Omega$, $\GLDD{\delta}{k}=\GLDD{\delta}{k}(x,y,t)$ satisfies 
  	\[
  		\left\{
  		\begin{array}{ll}
  		\Delta_x \GLDD{\delta}{k}=0 & \mbox{in}\quad \Omega\times(0,\infty),
		\vspace{5pt}\\
  		\delta\partial_t \GLDD{\delta}{k}-k\Delta'_x \GLDD{\delta}{k}-\partial_{x_N} \GLDD{\delta}{k}=0 
		& \mbox{on}	\quad \partial\Omega\times(0,\infty),
  		\end{array}
  		\right.
  	\]
	as a function of the variables $(x,t)\in \overline{\Omega}\times(0,\infty)$. 
  	\item
  	$\GLDD{\delta}{k}$ satisfies
  	\begin{equation}
	\label{eq:4.5}
	\int_{\partial\Omega}\GLDD{\delta}{k}(x,y,t)\,\dee\sigma(y)=1,\quad (x,t)\in \overline{\Omega}\times(0,\infty).
   	\end{equation}
   	Furthermore, if $\psi\in BC(\partial\Omega)$, then 
	\[
	\lim_{\substack{(z,t)\in Q,\\ (z,t)\to(x,0)}}\int_{\partial\Omega}\GLDD{\delta}{k}(z,y,t)\psi(y)\,\dee\sigma(y)=\psi(x),\quad x\in\partial\Omega.
	\]
  	\item
    	If $k=0$, then 
  	\[
  		\GLDD{\delta}{k}(x,y,t)=P\left(x'-y',x_N+y_N+\frac{t}{\delta}\right)
  	\]
	for $(x,y,t)\in \overline{D}$ with $x_N+y_N+t/\delta>0$.
\end{enumerate}
\end{lemma}
{\bf Proof.}
Assertion~(1) follows immediately from \eqref{eq:4.1}.
We prove assertion~(2).
It follows from \eqref{eq:1.2} and \eqref{eq:4.1} that
\[
	\begin{aligned}
	\Delta_x \GLDD{\delta}{k}(x,y,t) & =-2\int_0^\infty 
	\Delta_{x'}\Gamma_{N-1}\left(x'-y',\frac{k}{\delta}t+\tau\right)\partial_{x_N}
	\Gamma_1\left(x_N+\frac{t}{\delta},\tau\right)\,\dee\tau
	\\
 	& \qquad\quad
 	-2\int_0^\infty \Gamma_{N-1}\left(x'-y',\frac{k}{\delta}t+\tau\right)
	\partial_{x_N}^3\Gamma_1\left(x_N+\frac{t}{\delta},\tau\right)\,\dee\tau
	\\
 	& =-2\int_0^\infty \partial_\tau\left(\Gamma_{N-1}\left(x'-y',\frac{k}{\delta}t+\tau\right)
	\partial_{x_N}\Gamma_1\left(x_N+\frac{t}{\delta},\tau\right)\right)\,\dee\tau
	\\
 	& =-2\lim_{\tau\to\infty}\left(\Gamma_{N-1}\left(x'-y',\frac{k}{\delta}t+\tau\right)
	\partial_{x_N}\Gamma_1\left(x_N+\frac{t}{\delta},\tau\right)\right)
	\\
 	& \qquad\quad +2\lim_{\tau\to 0^+}\left(\Gamma_{N-1}\left(x'-y',\frac{k}{\delta}t+\tau\right)
	\partial_{x_N}\Gamma_1\left(x_N+\frac{t}{\delta},\tau\right)\right)=0
	\end{aligned}
\]
for $(x,y,t)\in\Omega\times {\partial}\Omega\times(0,\infty)$. 
On the other hand, since
\[
	\begin{aligned}
 	& \delta\partial_t\Gamma_{N-1}\left(x'-y',\frac{k}{\delta}t+\tau\right)
	=k\Delta_{x'}\Gamma_{N-1}\left(x'-y',\frac{k}{\delta}t+\tau\right),\quad x',y'\in{\mathbb R}^{N-1},\,\,\, t,\tau\in(0,\infty),
	\\
 	& \delta\partial_t\partial_{x_N}\Gamma_1\left(x_N+\frac{t}{\delta},\tau\right)\,\biggr|_{x_N=0}
	=\partial_{x_N}^2\Gamma_1\left(x_N+\frac{t}{\delta},\tau\right)\,\biggr|_{x_N=0},\quad t,\tau\in(0,\infty),
	\end{aligned}
\]
we have
\[
	\begin{aligned}
 	& \delta\partial_t \GLDD{\delta}{k}(x,y,t)-k\Delta'_x \GLDD{\delta}{k}(x,y,t)-\partial_{x_N} \GLDD{\delta}{k}(x,y,t)
	\\
 	& =-2\delta\int_0^\infty \partial_t\Gamma_{N-1}\left(x'-y',\frac{k}{\delta}t+\tau\right)
 	\partial_{x_N}\Gamma_1\left(x_N+\frac{t}{\delta},\tau\right)\,\dee\tau\,\biggr|_{x_N=0}
	\\
 	& \qquad
 	-2\delta\int_0^\infty \Gamma_{N-1}\left(x'-y',\frac{k}{\delta}t+\tau\right)
 	\partial_t\partial_{x_N}\Gamma_1\left(x_N+\frac{t}{\delta},\tau\right)\,\dee\tau\,\biggr|_{x_N=0}
	\\
  	& \qquad\quad
 	+2k\int_0^\infty \Delta_{x'}\Gamma_{N-1}\left(x'-y',\frac{k}{\delta}t+\tau\right)
 	\partial_{x_N}\Gamma_1\left(x_N+\frac{t}{\delta},\tau\right)\,\dee\tau\,\biggr|_{x_N=0}
	\\
	& \qquad\qquad
 	+2\int_0^\infty \Gamma_{N-1}\left(x'-y',\frac{k}{\delta}t+\tau\right)
 	\partial_{x_N}^2\Gamma_1\left(x_N+\frac{t}{\delta},\tau\right)\,\dee\tau\,\biggr|_{x_N=0}
 	=0
	\end{aligned}
\]
for $(x,y,t)\in\partial\Omega\times\partial\Omega\times(0,\infty)$. 
Hence, assertion~(2) follows. 

We prove assertion~(3). 
Since
\[
	\begin{aligned}
 	& -2\int_0^\infty\partial_{x_N}\Gamma_1\left(x_N+\frac{t}{\delta}{,\tau}\right)\,\dee\tau
 	=\int_0^\infty\frac{x_N+t/\delta}{\tau}\Gamma_1\left(x_N+\frac{t}{\delta}{,\tau}\right)\,\dee\tau
	\\
 	& \qquad\quad
 	=\pi^{-\frac{1}{2}}\int_0^\infty \left(\frac{(x_N+t/\delta)^2}{4\tau}\right)^{\frac{1}{2}}
	\exp\left(-\frac{(x_N+t/\delta)^2}{4\tau}\right)\,\frac{\dee\tau}{\tau}
	=\pi^{-\frac{1}{2}}\int_0^\infty \xi^{-\frac{1}{2}}e^{-\xi}\,\dee\xi=1
	\end{aligned}
\]
for $(x,t)\in \overline{\Omega}\times(0,\infty)$,
it follows from property~($\Gamma$1) and \eqref{eq:4.1} that 
\[
	\int_{\partial\Omega} \GLDD{\delta}{k}(x,y,t)\,\dee\sigma(y)
	=-2\int_0^\infty\partial_{x_N}\Gamma_1\left(x_N+\frac{t}{\delta},\tau\right)\,\dee\tau=1,
	\quad (x,t)\in \overline{\Omega}\times(0,\infty).
\]
Hence, \eqref{eq:4.5} holds. 
Furthermore, for any $x\in\partial\Omega$, $R>0$, and $z\in B^+_{R/2}(x)$,
it follows from property~($\Gamma$1) 
and the relation $|z'-y'|\ge|x'-y'|/2$ for $y\in\partial\Omega\setminus B'_R(x)$ that
\[
	\begin{aligned}
 	& \int_{\partial\Omega\setminus B'_R(x)}\Gamma_{N-1}(z'-y',t)\,\dee\sigma(y)
	\\
 	& \le\exp\left(-\frac{R^2}{32t}\right)
	\int_{\partial\Omega\setminus B'_R(x)}(4\pi t)^{-\frac{N-1}{2}}\exp\left(-\frac{|x'-y'|^2}{32t}\right)\,\dee\sigma(y)\le C\exp\left(-\frac{R^2}{32t}\right).
	\end{aligned}
\]
This, together with \eqref{eq:4.1}, implies that 
\[
	\begin{aligned}
 	& \int_{\partial\Omega\setminus B'_R(x)} \GLDD{\delta}{k}(z,y,t)\,\dee\sigma(y)
	\\
 	& \le C\int_0^\infty \exp\left(-\frac{R^2}{32((kt/\delta)+\tau)}\right)
	\left(\frac{(z_N+t/\delta)^2}{4\tau}\right)^{\frac{1}{2}}
	\exp\left(-\frac{(z_N+t/\delta)^2}{4\tau}\right)\,\frac{\dee\tau}{\tau}\\
 	& \le C\int_0^\infty\exp\left(-\frac{R^2}{32((kt/\delta)+(z_N+t/\delta)^2/(4\xi))}\right)
	\xi^{-\frac{1}{2}}e^{-\xi}\,\dee\xi
	\end{aligned}
\]
for $(x,t)\in \partial\Omega\times(0,\infty)$, 
$R>0$, $z\in B^+_{R/2}(x)$, and $t\in(0,1)$. 
Applying the Lebesgue dominated convergence theorem, 
we obtain
\begin{equation}
\label{eq:4.6}
\lim_{\substack{(z,t)\in Q,\\ (z,t)\to(x,0)}} 
\int_{\partial\Omega\setminus B'_R(x)} \GLDD{\delta}{k}(z,y,t)\,\dee\sigma(y)=0
\end{equation}
for $x\in \partial\Omega$ and $R>0$.
By \eqref{eq:4.5} and \eqref{eq:4.6}, 
for any $\psi\in BC(\partial\Omega)$, $x\in\partial\Omega$, and $R>0$, 
we obtain
\[
	\begin{aligned}
 	& \left|\int_{\partial\Omega} \GLDD{\delta}{k}(z,y,t)\psi(y)\,\dee \sigma(y)-\psi(x)\right|\\
 	& =\left|\int_{\partial\Omega} \GLDD{\delta}{k}(z,y,t)(\psi(y)-\psi(x))\,\dee\sigma(y)\right|
	\\
 	& \le 2\|\psi\|_{L^\infty(\Omega)}\int_{\partial\Omega\setminus B'_R(x)} \GLDD{\delta}{k}(z,y,t)\,\dee\sigma(y)
 	+\sup_{y\in \partial\Omega\cap B'_R(x)}|\psi(x)-\psi(y)|
	\\
 	& \to \sup_{y\in \partial\Omega\cap B'_R(x)}|\psi(x)-\psi(y)|\quad\mbox{as}\quad (z,t)\to (x,0^+).
	\end{aligned}
\]
Since $R>0$ is arbitrary, 
we conclude, by the continuity of $\psi$, that
    \[
    \lim_{\substack{(z,t)\in Q,\\ (z,t)\to(x,0)}}
    \int_{\partial\Omega}\GLDD{\delta}{k}(z,y,t)\psi(y)\,\dee\sigma(y)=\psi(x),\quad x\in\partial\Omega.
    \]
Hence, assertion~(3) holds.

We prove assertion~(4). 
Let $k=0$. 
Then it follows from \eqref{eq:1.6} and \eqref{eq:4.1} that 
\[
	\begin{aligned}
 	\GLDD{\delta}{k}(x,y,t)
 	& =-2\int_0^\infty \Gamma_{N-1}\left(x'-y',\tau\right)\partial_{x_N}
	\Gamma_1\left(x_N+y_N+\frac{t}{\delta},\tau\right)\,\dee\tau
	\\
 	& =\left(x_N+y_N+\frac{t}{\delta}\right)
	\int_0^\infty \Gamma_N\left(x-y^*+\frac{t}{\delta}e_N,\tau\right)\,\frac{\dee\tau}{\tau}
	\\
 	& =(4\pi)^{-\frac{N}{2}}\left(x_N+y_N+\frac{t}{\delta}\right)
	\int_0^\infty\left(\frac{|x-y^*+(t/\delta)e_N|^2}{4\xi}\right)^{-\frac{N}{2}}e^{-\xi}\frac{\dee\xi}{\xi}
	\\
 	& =\pi^{-\frac{N}{2}}\left(x_N+y_N+\frac{t}{\delta}\right)\left|x-y^*+\frac{t}{\delta}e_N\right|^{-N}
	\int_0^\infty \xi^{-1+\frac{N}{2}}e^{-\xi}\,\dee\xi
	\\
 	& =P\left(x'-y',x_N+y_N+\frac{t}{\delta}\right)
	\end{aligned}
\]
for $(x,y,t)\in \overline{D}$ with $x_N+y_N+t/\delta>0$. 
Hence, assertion~(4) holds. 
This completes the proof of Lemma~\ref{Lemma:4.1}.
$\Box$
\medskip
\newline
{\bf Proof of Theorem~\ref{Theorem:4.1}.}
Let $\delta>0$ be fixed.
Let $(\phi,\psi)\in L^\infty(\Omega)\times L^\infty(\partial\Omega)$, and 
denote by $u$ the solution $\uHDD{\epsilon}{\delta}{k}$ (resp. $\uHD{\epsilon}{\delta}$) when $k>0$ (resp. $k=0$). 
Let $L>0$ and $I$ be a compact set in $(0,\infty)$. 
It follows from \eqref{eq:3.6} that
\[
	\int_\Omega G_0\left(x,y,\frac{t}{\epsilon}\right)\phi(y)\,\dee y
	\le\|\phi\|_{L^\infty(\Omega)}\int_{-x_N}^{x_N}\Gamma_1\left(\xi,\frac{t}{\epsilon}\right)\,\dee\xi
 	\le C\epsilon^{\frac{1}{2}} t^{-\frac{1}{2}}x_N\le C\epsilon^{\frac{1}{2}}
\]
for $(x,t)\in \Omega_L\times I$ and $\epsilon>0$. 
It also follows from \eqref{eq:3.7} that
\[
	\frac{1}{\delta}\int_\Omega H(x,y,t)\phi(y)\,\dee y
	\le \frac{C\epsilon^{\frac{1}{2}}}{\delta}t^{\frac{1}{2}}\|\phi\|_{L^\infty(\Omega)}
	\le C\epsilon^{\frac{1}{2}}
\]
for $(x,t)\in \Omega_L\times I$ and $\epsilon>0$. 
These, together with \eqref{eq:1.12}, imply that
\begin{equation}
\label{eq:4.7}
	\begin{aligned}
 	& \sup_{(x,t)\in\Omega_L\times I}
	\left|\,u(x,t)-\int_{\partial\Omega}\GLDD{\delta}{k}(x,y,t)\psi(y)\,\dee\sigma(y)\,\right|
	\\
 	& \le\sup_{(x,t)\in\Omega_L\times I}
 	\left|\,\frac{1}{\epsilon}\int_{\partial\Omega}H(x,y,t)\psi(y)\,\dee\sigma(y)
	-\int_{\partial\Omega}\GLDD{\delta}{k}(x,y,t)\psi(y)\,\dee\sigma(y)\,\right|
 	+C\epsilon^{\frac{1}{2}}
	\end{aligned}
\end{equation}
for $\epsilon>0$. 

On the other hand, for {$(x,y,t)\in\overline{\Omega}\times\partial\Omega\times(0,\infty)$}, 
by \eqref{eq:1.8} and \eqref{eq:4.1}, we obtain
\[
	\begin{aligned}
	 & H(x,y,t)\\
 	& =-2\int_0^t \Gamma_{N-1}\left(x'-y',\frac{\tau}{\epsilon}+\frac{k}{\delta}(t-\tau)\right)
	\partial_{x_N}\Gamma_1\left(x_N+\frac{t-\tau}{\delta},\frac{\tau}{\epsilon}\right)\,\dee\tau
	\\
 	& =J_1(x,y,t)+J_2(x,y,t)+J_3(x,y,t)
	\\
 	& \qquad\quad
 	-2\int_0^\infty \partial_{x_N}\Gamma_1\left(x_N+\frac{t}{\delta},\frac{\tau}{\epsilon}\right)
	\Gamma_{N-1}\left(x'-y',\frac{\tau}{\epsilon}+\frac{k}{\delta}t\right)\,\dee\tau
	\\
 	& =J_1(x,y,t)+J_2(x,y,t)+J_3(x,y,t)+\epsilon \GLDD{\delta}{k}(x,y,t), 
	\end{aligned}
\]
where
\[
	\begin{aligned}
 	& J_1(x,y,t):= -2\int_0^t \Gamma_{N-1}\left(x'-y',\frac{\tau}{\epsilon}+\frac{k}{\delta}(t-\tau)\right)
	\\
 	& \qquad\quad
 	\times\left\{\partial_{x_N}\Gamma_1\left(x_N+\frac{t-\tau}{\delta},\frac{\tau}{\epsilon}\right)
	-\partial_{x_N}\Gamma_1\left(x_N+\frac{t}{\delta},\frac{\tau}{\epsilon}\right)\right\}\,\dee\tau,
	\\
 	& J_2(x,y,t):=-2\int_0^t \left\{\Gamma_{N-1}\left(x'-y',\frac{\tau}{\epsilon}+\frac{k}{\delta}(t-\tau)\right)
	-\Gamma_{N-1}\left(x'-y',\frac{\tau}{\epsilon}+\frac{k}{\delta}t\right)\right\}
	\\
 	& \qquad\quad
 	\times \partial_{x_N}\Gamma_1\left(x_N+\frac{t}{\delta},\frac{\tau}{\epsilon}\right)\,\dee\tau,
	\\
 	& J_3(x,y,t):=2\int_t^\infty \Gamma_{N-1}\left(x'-y',\frac{\tau}{\epsilon}+\frac{k}{\delta}t\right)
	\partial_{x_N}\Gamma_1\left(x_N+\frac{t}{\delta},\frac{\tau}{\epsilon}\right)\,\dee\tau.
	\end{aligned}
\]
Since
\[
	J_1(x,y,t)
	=2\int_0^t\int_{(t-\tau)/\delta}^{t/\delta}
	\Gamma_{N-1}\left(x'-y',\frac{\tau}{\epsilon}+\frac{k}{\delta}(t-\tau)\right)
	\partial_{x_N}\partial_\xi\Gamma_1\left(x_N+\xi,\frac{\tau}{\epsilon}\right)\,\dee\xi\,\dee\tau,
\]
it follows from properties~($\Gamma$1) and ($\Gamma$2) that
\begin{equation}
\label{eq:4.8}
	\begin{aligned}
 	& \frac{1}{\epsilon}\int_{\partial\Omega}|J_1(x,y,t)|\,\dee\sigma(y)
	\\
 	& \le \frac{C}{\epsilon}\int_0^t\int_{(t-\tau)/\delta}^{t/\delta}
	\left(\frac{\tau}{\epsilon}\right)^{-\frac{3}{2}}\,\dee\xi\,\dee\tau
 	\le\frac{C}{\epsilon}\int_0^t \frac{\tau}{\delta}\left(\frac{\tau}{\epsilon}\right)^{-\frac{3}{2}}\,\dee\tau
 	\le \frac{C\epsilon^{\frac{1}{2}}}{\delta}t^{\frac{1}{2}}
	\le C\epsilon^{\frac{1}{2}}
	\end{aligned}
\end{equation}
for $(x,t)\in \Omega_L\times I$ and $\epsilon>0$. 
Moreover, if $k>0$, since 
\[
 	J_2(x,y,t)=2\int_0^t\int_{-k\tau/\delta}^0
 	\partial_\xi\Gamma_{N-1}\left(x'-y',\frac{\tau}{\epsilon}+\frac{k}{\delta}t+\xi\right)
 	\partial_{x_N}\Gamma_1\left(x_N+\frac{t}{\delta},\frac{\tau}{\epsilon}\right)\,\dee\xi\,\dee\tau,
\]
by properties~($\Gamma$1) and ($\Gamma$2) again, we have
\begin{equation}
\label{eq:4.9}
	\begin{aligned}
 	& \frac{1}{\epsilon}\int_{\partial\Omega} |J_2(x,y,t)|\,\dee\sigma(y)
	\\
 	& \le \frac{C}{\epsilon}\left(\int_0^{t/2}+\int_{t/2}^t\right)\int_{-k\tau/\delta}^0
 	\left(\frac{\tau}{\epsilon}+\frac{k}{\delta}t+\xi\right)^{-1}
 	\left(\frac{\tau}{\epsilon}\right)^{-\frac{3}{2}}\left(x_N+\frac{t}{\delta}\right)\,\dee\xi\,\dee\tau
	\\
	& \le \frac{C}{\epsilon}\left(x_N+\frac{t}{\delta}\right)
	\left\{\left(\frac{k}{\delta}t\right)^{-1}\int_0^{t/2}
	\int_{-k\tau/\delta}^0\left(\frac{\tau}{\epsilon}\right)^{-\frac{3}{2}}\,\dee\xi\,\dee\tau
	+\int_{t/2}^t\int_{-k\tau/\delta}^0\left(\frac{\tau}{\epsilon}\right)^{-\frac{5}{2}}\,\dee\xi\,\dee\tau\right\}
	\\
 	& \le C\epsilon^{\frac{1}{2}}\left(x_N+\frac{t}{\delta}\right)
 	\left\{t^{-1}\int_0^{t/2}\tau^{-\frac{1}{2}}\,\dee\tau
 	+\epsilon\int_{t/2}^t \tau^{-\frac{3}{2}}\,\dee\tau\right\}
	\\
	& \le C\epsilon^{\frac{1}{2}}\left(x_N+\frac{t}{\delta}\right)\left(t^{-\frac{1}{2}}+\epsilon t^{-\frac{1}{2}}\right)
	\le C\epsilon^{\frac{1}{2}}+C\epsilon^{\frac{3}{2}}
	\end{aligned}
\end{equation}
for $(x,t)\in \Omega_L\times I$ and $\epsilon>0$,
If $k=0$, then it holds that $J_2(x,y,t)\equiv0$ for $(x,t)\in\overline{\Omega}\times(0,\infty)$ 
and $\epsilon>0$.
Since 
\[
	J_3(x,y,t)=2\epsilon\int_{t/\epsilon}^\infty \Gamma_{N-1}\left(x'-y',\tau+\frac{k}{\delta}t\right)
	\partial_{x_N}\Gamma_1\left(x_N+\frac{t}{\delta},\tau\right)\,\dee\tau,
\]
by properties~($\Gamma$1) and ($\Gamma$2), we obtain
\begin{equation}
\label{eq:4.10}
	\begin{aligned}
 	& \frac{1}{\epsilon}\int_{\partial\Omega} |J_3(x,y,t)|\,\dee\sigma(y)
 	\le 2\int_{t/\epsilon}^\infty\left|\partial_{x_N}\Gamma_1\left(x_N+\frac{t}{\delta},\tau\right)\right|\,\dee\tau
	\\
 	& \qquad\quad
 	\le C\left(x_N+\frac{t}{\delta}\right)\int_{\epsilon^{-1}t}^\infty\tau^{-\frac{3}{2}}\,\dee\tau
  	\le C\epsilon^{\frac{1}{2}}\left(x_N+\frac{t}{\delta}\right)t^{-\frac{1}{2}}
 	\le C\epsilon^{\frac{1}{2}}
	\end{aligned}
\end{equation}
for $(x,t)\in \Omega_L\times I$ and $\epsilon>0$.
Therefore, by \eqref{eq:4.7}, \eqref{eq:4.8}, \eqref{eq:4.9}, and \eqref{eq:4.10}, 
we conclude that
\[
	\begin{aligned}
 	& \left|u(x,t)-\int_{\partial\Omega}\GLDD{\delta}{k}(x,y,t)\psi(y)\,\dee\sigma(y)\right|
	\\
 	& \le\frac{1}{\epsilon}\|\psi\|_{L^\infty(\partial\Omega)}\sum_{j=1}^3
	\sup_{(x,t)\in\Omega_L\times I}\int_{\partial\Omega} |J_j(x,y,t)|\,\dee\sigma(y)+C\epsilon^{\frac{1}{2}}
 	\le C\epsilon^{\frac{1}{2}}+C\epsilon^{\frac{3}{2}}
	\end{aligned}
\]
for $(x,t)\in \Omega_L\times I$ and $\epsilon>0$.
Hence, \eqref{eq:4.2} holds. 
If $\psi\in BC(\partial\Omega)$, 
by Lemma~\ref{Lemma:4.1}, 
we see that $\uLDD{\delta}{k}$ (resp.~$\uLD{\delta}$) is a bounded classical solution 
to problem~\eqref{eq:LDD} (resp.~problem~\eqref{eq:LD}). 

It remains to prove \eqref{eq:4.4}. 
It follows from \eqref{eq:4.1} that
\[
	\begin{aligned}
 	& \GLDD{\delta}{k}(x,y,t)-\GLDD{\delta}{0}(x,y,t)
	\\
 	& =-2\int_0^\infty \left\{\Gamma_{N-1}\left(x'-y',\frac{k}{\delta}t+\tau\right)
	-\Gamma_{N-1}\left(x'-y',\tau\right)\right\}
	\partial_{x_N}\Gamma_1\left(x_N+\frac{t}{\delta},\tau\right)\,\dee\tau
	\\
 	& =-2\int_0^\infty\left(\int_0^{kt/\delta} \partial_\xi\Gamma_{N-1}\left(x'-y',\xi+\tau\right)\,\dee\xi\right) 
	\partial_{x_N}\Gamma_1\left(x_N+\frac{t}{\delta},\tau\right)\,\dee\tau
	\end{aligned}
\]
for {$(x,y,t)\in\overline{\Omega}\times\partial\Omega\times(0,\infty)$}. 
Then, for any $T>0$ and $R>0$, 
by property~($\Gamma$2) and \eqref{eq:2.2}, we obtain
\begin{equation}
\label{eq:4.11}
	\begin{aligned}
 	& \int_{\partial\Omega}\left|\GLDD{\delta}{k}(x,y,t)-\GLDD{\delta}{0}(x,y,t)\right|\,\dee\sigma(y)
	\\
	& \le C\int_0^\infty\int_0^{kt/\delta}
	(\xi+\tau)^{-1}\frac{x_N+t/\delta}{\tau^{\frac{3}{2}}}\exp\left(-\frac{(x_N+t/\delta)^2}{4\tau}\right)\,\dee\xi\,\dee\tau
	\\
	& \le \frac{Ckt}{\delta}\int_0^\infty\tau^{-2}\exp\left(-\frac{(x_N+t/\delta)^2}{8\tau}\right)\,\dee\tau\\
	& \le \frac{CkT}{\delta}\int_0^\infty\tau^{-2}\exp\left(-\frac{R^2}{8\tau}\right)\,\dee\tau
	\le \frac{CkT}{\delta}
	\end{aligned}
\end{equation}
for $(x,t)\in \overline{\Omega}\times(0,T]$ with $x_N+t/\delta\ge R$. 
Hence, \eqref{eq:4.4} holds. 
Thus, Theorem~\ref{Theorem:4.1} follows.
$\Box$
\begin{remark}
\label{Remark:4.1}
{\rm (1)} 
The convergence rate in~\eqref{eq:4.2} is optimal. 
Indeed, consider the case where 
$\phi=\chi_{{\mathbb R}^{N-1}\times(1,\infty)}$ in $\Omega$ 
and $\psi=0$ on $\partial\Omega$. 
If $k>0$, then it follows from \eqref{eq:1.3}, \eqref{eq:1.12}, and \eqref{eq:4.3} that
\[
	\begin{aligned}
  	& \uHDD{\epsilon}{\delta}{k}(e_N,1)-\uLDD{\delta}{k}(e_N,1)=\uHDD{\epsilon}{\delta}{k}(e_N,1)
	\\
  	& \ge \int_1^\infty\int_{{\mathbb R}^{N-1}}G_0\left(e_N,y,\frac{1}{\epsilon}\right)\,\dee y'\,\dee y_N\\
  	& =\int_0^\infty \Gamma_1\left(y_N,\frac{1}{\epsilon}\right)\,\dee y_N
	-\int_2^\infty \Gamma_1\left(y_N,\frac{1}{\epsilon}\right)\,\dee y_N
	\\
  	&=\int_0^2 \Gamma_1\left(y_N,\frac{1}{\epsilon}\right)\,\dee y_N\ge
  	\left(\frac{4\pi}{\epsilon}\right)^{-\frac{1}{2}}\int_0^2\exp\left(-\frac{y_N^2}{4}\right)\,\dee y_N\\
	& \ge C\epsilon^{1/2}
	\end{aligned}
\]
for $\epsilon\in(0,1)$. 
This shows that the convergence rate in~\eqref{eq:4.2} is optimal when $k>0$. 
The case $k=0$ can be treated similarly.
\newline
{\rm (2)} 
The convergence rate in \eqref{eq:4.4} is optimal. 
To see this, consider the case where $\psi=\chi_{B'_1(0)}$ on~$\partial\Omega$. 
Since it follows from Lemma~{\rm\ref{Lemma:4.1}\,(4)} and \eqref{eq:4.1} that
\[
	\begin{aligned}
	 & \GLDD{\delta}{k}(x,y,t)\\
  	&=-2\int_0^\infty\int_{\R^{N-1}}\Gamma_{N-1}\left(z'-y',\frac{k}{\delta}t\right)
	\Gamma_{N-1}\left(x'-z',\tau\right)\partial_{x_N}
	\Gamma_1\left(x_N+\frac{t}{\delta},\tau\right)\,\dee z'\,\dee\tau
	\\
  	&=\int_{\R^{N-1}}\Gamma_{N-1}\left(z'-y',\frac{k}{\delta}t\right)
	\GLDD{\delta}{0}\left(x,(z',0),t\right)\,\dee z'
	\\
  	&=\int_{\R^{N-1}}\Gamma_{N-1}\left(z'-y',\frac{k}{\delta}t\right)
	P\left(x'-z',x_N+\frac{t}{\delta}\right)\,\dee z'
	\\
  	& =\int_{\R^{N-1}}\Gamma_{N-1}\left(x'-z',\frac{k}{\delta}t\right)
	P\left(z'-y',x_N+\frac{t}{\delta}\right)\,\dee z'
	\end{aligned}
\]
for $(x,y,t)\in\overline{\Omega}\times\partial\Omega\times(0,\infty)$,
we have
\[
	\begin{aligned}
 	& \uLDD{\delta}{k}(x,t)-\uLD{\delta}(x,t)
	\\
 	& =\int_{\R^{N-1}}\int_{\R^{N-1}}\Gamma_{N-1}\left(x'-z',\frac{k}{\delta}t\right)
	P\left(z'-y',x_N+\frac{t}{\delta}\right)\psi(y')\,\dee z'\,\dee y'-\uLD{\delta}(x,t)
	\\
 	& =\int_{\R^{N-1}}\Gamma_{N-1}\left(x'-z',\frac{k}{\delta}t\right)
	\uLD{\delta}((z',x_N),t)\,\dee z'-\uLD{\delta}(x,t)
	\\
 	& =\int_0^{kt/\delta}\int_{\R^{N-1}}\partial_\xi\Gamma_{N-1}(x'-z',\xi)
	\uLD{\delta}((z',x_N),t)\,\dee\xi\,\dee z'
	\\
 	& =\int_0^{kt/\delta}\int_{\R^{N-1}}\Delta_{z'}\Gamma_{N-1}(x'-z',\xi)
	\uLD{\delta}((z',x_N),t)\,\dee\xi\,\dee z'
	\\
 	& =\int_0^{kt/\delta}\int_{\R^{N-1}}\Gamma_{N-1}(x'-z',\xi)\Delta_{z'}
	\uLD{\delta}((z',x_N),t)\,\dee\xi\,\dee z'.
	\end{aligned}
\]
Since $\uLD{\delta}$ is harmonic in $\Omega$, we obtain 
\[
	\begin{aligned}
	\frac{1}{k}(\uLDD{\delta}{k}(x,t)-\uLD{\delta}(x,t))
 	& =-\frac{1}{k}\int_0^{kt/\delta}\int_{\R^{N-1}}\Gamma_{N-1}(x'-z',\xi)
	(\partial_{x_N}^2\uLD{\delta})((z',x_N),t)\,\dee z'\,\dee\xi
	\\
 	& \to -\frac{t}{\delta}(\partial_{x_N}^2\uLD{\delta})(x,t)
	\quad\mbox{as}\quad k\to 0^+
	\end{aligned}
\]
for $x\in\Omega$.
On the other hand, 
since $\psi=\chi_{B'_1(0)}$ on $\partial\Omega$, 
for any fixed $t_*>0$, 
we have $0<\uLD{\delta}(x,t_*)\le 1$ for $x\in\Omega$ 
and $\lim_{x_N\to\infty}\uLD{\delta}(x,t_*)=0$ {\rm({\it see} \eqref{eq:2.13})}.
Then there exists $x_*\in\Omega$ such that $(\partial_{x_N}^2\uLD{\delta})(x_*,t_*)\not=0$. 
Hence, 
\[
	\lim_{k\to 0^+}\frac{1}{k}(\uLDD{\delta}{k}(x_*,t)-\uLD{\delta}(x_*,t))
	=-\frac{t}{\delta}(\partial_{x_N}^2\uLD{\delta})(x_*,t)\not=0.
\]
Thus, the convergence rate in \eqref{eq:4.4} is optimal.
\end{remark}

In the next two theorems, 
we state results concerning the diffusion limits of solutions to problems~\eqref{eq:LDD} and \eqref{eq:LD}.
\begin{theorem}
\label{Theorem:4.2}
Let $\psi\in L^p(\partial\Omega)$ with $p\in[1,\infty)$ and $T>0$. 
Then
\[
	\begin{aligned}
    & \sup_{(x,t)\in \Omega\times(T,\infty)}
	|\uLDD{\delta}{k}(x,t)|=O\left(\delta^\frac{N-1}{p}\right)\quad\textrm{as}\quad\delta\to 0^+\quad 
	\textrm{for any fixed $k>0$},\\
    & \sup_{(x,t)\in \Omega\times(T,\infty)}
	|\uLD{\delta}(x,t)|=O\left(\delta^\frac{N-1}{p}\right)\quad\textrm{as $\delta\to 0^+$},\\
 	& \lim_{k\to\infty}\sup_{(x,t)\in \Omega\times(T,\infty)}|\uLDD{\delta}{k}(x,t)|=O\left(k^{-\frac{N-1}{2p}}\right)
	\quad\textrm{as}\quad k\to\infty\quad \textrm{for any fixed $\delta>0$}.
	\end{aligned}
\]
\end{theorem}
{\bf Proof.}
Let $T>0$. 
By \eqref{eq:2.2} and \eqref{eq:4.1}, we have
\[
	\begin{aligned}
 	& \sup_{(x,t)\in\Omega\times(T,\infty)}
	\left|\,\int_{\partial\Omega} \GLDD{\delta}{k}(x,y,t)\psi(y)\,\dee\sigma(y)\,\right|
	\\
 	& \le C\|\psi\|_{L^p(\partial\Omega)}
	\sup_{(x,t)\in\Omega\times(T,\infty)}
	\int_0^\infty \left(\frac{k}{\delta}t+\tau\right)^{-\frac{N-1}{2p}}(4\pi\tau)^{-\frac{1}{2}}
	\frac{x_N+t/\delta}{\tau}\exp\left(-\frac{|x_N+t/\delta|^2}{4\tau}\right)\,\dee\tau
	\\
 	& \le C\|\psi\|_{L^p(\partial\Omega)}\int_0^\infty 
	\left(\frac{k}{\delta}T+\tau\right)^{-\frac{N-1}{2p}}\tau^{-1}
 	\exp\left(-\frac{(T/\delta)^2}{8\tau}\right)\,\dee\tau.
	\end{aligned}
\]
Combining this with 
\begin{align*}
  &\int_0^\infty 
	\left(\frac{k}{\delta}T+\tau\right)^{-\frac{N-1}{2p}}\tau^{-1}
 	\exp\left(-\frac{(T/\delta)^2}{8\tau}\right)\,\dee\tau\\
    &\le\int_0^\infty 
	\left(\frac{\xi}{\delta^2}\right)^{-\frac{N-1}{2p}}\xi^{-1}
 	\exp\left(-\frac{T^2}{8\xi}\right)\,\dee\xi=O\left(\delta^\frac{N-1}{p}\right)\quad\textrm{as}\quad\delta\to 0^+,\\
    &\int_0^\infty 
	\left(\frac{k}{\delta}T+\tau\right)^{-\frac{N-1}{2p}}\tau^{-1}
 	\exp\left(-\frac{(T/\delta)^2}{8\tau}\right)\,\dee\tau\\
    &\le \int_0^\infty \left(\frac{k}{\delta}T\right)^{-\frac{N-1}{2p}}\tau^{-1}
 	\exp\left(-\frac{(T/\delta)^2}{8\tau}\right)\,\dee\tau=O\left(k^{-\frac{N-1}{2p}}\right)\quad\textrm{as}\quad k\to\infty,
\end{align*}
we complete the proof of Theorem~\ref{Theorem:4.2}.
$\Box$
\vspace{5pt}
\newline
Note that Theorem~\ref{Theorem:4.2} does not necessarily hold in the case $p=\infty$, 
as follows from \eqref{eq:4.5}.
\begin{theorem}
\label{Theorem:4.3}
Let $\psi\in L^{\infty}(\partial\Omega)$. 
Then, for any $L>0$ and $T>0$,  
\begin{equation}
\label{eq:4.12}
	\begin{aligned}
	&\lim_{\delta\to\infty}\sup_{(x,t)\in\Omega_L^c\times(0,T)}|\uLDD{\delta}{k}(x,t)-\uLDi{\psi}(x)|=0
	\quad\mbox{for any fixed $k>0$},
	\\
	&\lim_{\delta\to\infty}\sup_{(x,t)\in\Omega_L^c\times(0,T)}|\uLD{\delta}(x,t)-\uLDi{\psi}(x)|=0.
	\end{aligned}
\end{equation}
Here $\uLDi{\psi}$ is given in \eqref{eq:2.11} and $\Omega_L^c:=\{x\in\overline{\Omega}\,:\,x_N>L\}$.
\end{theorem}
{\bf Proof.}
Similarly to \eqref{eq:4.11}, we have
\[
	\begin{aligned}
 	& \int_{\partial\Omega}\left|\GLDD{\delta}{k}(x,y,t)-\GLDD{\delta}{0}(x,y,t)\right|\,\dee\sigma(y)
	\\
	& \le \frac{Ckt}{\delta}\int_0^\infty\tau^{-2}\exp\left(-\frac{x_N^2}{8\tau}\right)\,\dee\tau
	\le \frac{CkT}{\delta}\int_0^\infty\tau^{-2}\exp\left(-\frac{L^2}{8\tau}\right)\,\dee\tau
	\le \frac{CkT}{\delta}
	\end{aligned}
\]
for $(x,t)\in\Omega_L^c\times(0,T)$.
Furthermore, by ($\Gamma 1$) and \eqref{eq:4.1}, we have
\[
	\begin{aligned}
 	& \int_{\partial\Omega}\left|\GLDD{\delta}{0}(x,y,t)
	-\GLDD{\delta}{0}(x,y,0)\right|\,\dee\sigma(y)
	\\
 	& \le 2\int_0^t \left|\partial_{x_N}\Gamma_1\left(x_N+\frac{t}{\delta},\tau\right)
	-\partial_{x_N}\Gamma_1\left(x_N,\tau\right)\right|\,\dee\tau
	\\
 	& \le \frac{Ct}{\delta}\int_0^t \tau^{-\frac{3}{2}}\exp\left(-\frac{x_N^2}{8\tau}\right)\,\dee\tau
	\le  \frac{CT}{\delta}\int_0^\infty \tau^{-\frac{3}{2}}\exp\left(-\frac{L^2}{8\tau}\right)\,\dee\tau
	\le\frac{CT}{\delta}
	\end{aligned}
\]
for $(x,t)\in\Omega_L^c\times(0,T)$. 
These, together with Lemma~\ref{Lemma:4.1}~(4), imply \eqref{eq:4.12}, and Theorem~\ref{Theorem:4.3} follows.
$\Box$
\vspace{5pt}

At the end of this section 
we obtain a result concerning the limit of solutions to problems~\eqref{eq:HDD} and \eqref{eq:HD} 
in the limit $\epsilon\to\infty$. 
\begin{theorem}
\label{Theorem:4.4}
Let $(\phi,\psi)\in BC(\Omega)\times L^\infty(\partial\Omega)$.
Let $\delta>0$ and $k>0$ be fixed.
Then, for any compact set $K\subset\Omega$ and $T>0$,
\[
	\begin{aligned}
 	& \lim_{\epsilon\to \infty}\sup_{(x,t)\in K\times(0,T)}|\uHDD{\epsilon}{\delta}{k}(x,t)-\phi(x)|=0,
	\\
 	& \lim_{\epsilon\to \infty}\sup_{(x,t)\in K\times(0,T)}|\uHD{\epsilon}{\delta}(x,t)-\phi(x)|=0. 
	\end{aligned}
\]
\end{theorem}
{\bf Proof.}
Let $K$ be a compact set in $\Omega$ and $T>0$.
For $(x_N,y_N)\in(0,\infty)^2$ and $\tau\in(0,t)$, 
since 
\begin{equation}
\label{eq:4.13}
	\left|x_N+y_N+\frac{\tau}{\delta}\right|^2\ge x_N^2+y_N^2+\left(\frac{\tau}{\delta}\right)^2,
\end{equation}
by \eqref{eq:2.1}, we have
\[
	\begin{aligned}
	0 & \le-2\partial_{x_N}\Gamma_1\left(x_N+y_N+\frac{\tau}{\delta},\frac{t-\tau}{\epsilon}\right)
 	\le C\left(\frac{t-\tau}{\epsilon}\right)^{-1}
	\exp\left(-\frac{\epsilon|x_N+y_N+\tau/\delta|^2}{8(t-\tau)}\right)
	\\
 	& \le C\epsilon^{\frac{1}{2}}\exp\left(-\frac{\epsilon x_N^2}{8t}\right)(t-\tau)^{-\frac{1}{2}}
	\exp\left(-\frac{\epsilon \tau^2}{8\delta^2(t-\tau)}\right)
 	\left(\frac{4\pi(t-\tau)}{\epsilon}\right)^{-\frac{1}{2}}\exp\left(-\frac{\epsilon y_N^2}{8(t-\tau)}\right)
	\\
	& \le C\epsilon^{\frac{1}{2}}\exp\left(-\frac{\epsilon x_N^2}{8t}\right)(t-\tau)^{-\frac{1}{2}}
 	\left(\frac{4\pi(t-\tau)}{\epsilon}\right)^{-\frac{1}{2}}\exp\left(-\frac{\epsilon y_N^2}{8(t-\tau)}\right).
	\end{aligned}
\]
This, together with property~($\Gamma$1), \eqref{eq:3.1}, and \eqref{eq:3.2}, implies that
\[
	\begin{aligned}
 	& \frac{1}{\delta}\int_\Omega |H(x,y,t)|\,\dee y
	\\
 	&  \le C\epsilon^{\frac{1}{2}}\exp\left(-\frac{\epsilon x_N^2}{8t}\right)
 	\int_0^t\int_0^\infty (t-\tau)^{-\frac{1}{2}}\left(\frac{4\pi(t-\tau)}{\epsilon}\right)^{-\frac{1}{2}}
	\exp\left(-\frac{\epsilon y_N^2}{8(t-\tau)}\right)\,\dee y_N\,\dee\tau
	\\
 	&  \le C\epsilon^{\frac{1}{2}}\exp\left(-\frac{\epsilon x_N^2}{8t}\right)
  	\int_0^t (t-\tau)^{-\frac{1}{2}}\,\dee\tau
  	\le C\epsilon^{\frac{1}{2}}t^{\frac{1}{2}}\exp\left(-\frac{\epsilon x_N^2}{Ct}\right)
  	\le C\epsilon^{\frac{1}{2}}T^{\frac{1}{2}}\exp\left(-\frac{\epsilon x_N^2}{CT}\right)
	\end{aligned}
\]
for {$(x,t)\in\Omega\times(0,T)$}. 
Similarly, by \eqref{eq:3.1} and \eqref{eq:4.13}, we obtain
\[
	\begin{aligned}
 	& \frac{1}{\epsilon}\int_{\partial\Omega} |H(x,y,t)|\,\dee\sigma(y)
	\\
 	& \le C\epsilon^{\frac{1}{2}}\left(x_N+\frac{t}{\delta}\right)\exp\left(-\frac{\epsilon x_N^2}{4t}\right)
	\left(\int_0^{t/2}+\int_{t/2}^t\right)
	(t-\tau)^{-\frac{3}{2}}\exp\left(-\frac{\epsilon \tau^2}{4\delta^2(t-\tau)}\right)\,\dee\tau
	\\
	& \le C\epsilon^{\frac{1}{2}}\left(x_N+\frac{t}{\delta}\right)\exp\left(-\frac{\epsilon x_N^2}{4t}\right)
	\\
	& \qquad\times\left\{\int_0^{t/2}(t-\tau)^{-\frac{3}{2}}\,\dee\tau
	+\int_{t/2}^t\left(\frac{t-\tau}{\epsilon t^2}\right)^{-\frac{3}{2}}\epsilon^{-\frac{3}{2}}t^{-3}
	\exp\left(-\frac{\epsilon t^2}{16\delta^2(t-\tau)}\right)\,\dee\tau\right\}
	\\
	 & \le C\epsilon^{\frac{1}{2}}\left(x_N+\frac{t}{\delta}\right)\exp\left(-\frac{\epsilon x_N^2}{4t}\right)
	 \left(t^{-\frac{1}{2}}+\epsilon^{-\frac{3}{2}}t^{-2}\right)
	\end{aligned}
\]
for $(x,t)\in \Omega\times(0,T)$ and $\epsilon>0$. 
These imply that
\[
	\frac{1}{\delta}\int_\Omega H(x,y,t)\phi(y)\,\dee y
	+\frac{1}{\epsilon}\int_{\partial\Omega}H(x,y,t)\psi(y)\,\dee\sigma(y)
	\to 0\quad\mbox{as}\quad\epsilon\to\infty
\]
uniformly for $(x,t)\in K\times(0,T)$. 
This, together with \eqref{eq:1.12} and \eqref{eq:2.5}, implies the desired conclusion. 
Thus, Theorem~\ref{Theorem:4.4} follows.
$\Box$
\vspace{5pt}
\newline
A diagram summarizing the diffusion limits established in this section, 
together with Theorem~\ref{Theorem:1.3}~(5), is given below. 
We use the symbol $\Longrightarrow$ to denote convergence with the optimal rate.
\medskip
\[
	\xymatrix{ \phi \quad & 
	\quad \eqref{eq:HDD}_{\epsilon,\delta,k} \ar[l]_{\epsilon\to \infty} \ar@{=>}[r]^{\epsilon\to 0^+}  \ar@{=>}[d]^{k\to 0^+} \quad & 
	\quad \eqref{eq:LDD}_{\delta,k} \ar[r]^{\delta\to\infty} \ar[rd]^{\delta\to 0^+}_{k\to\infty} \ar@{=>}[d]_{k\to 0^+} \quad & \quad\eqref{eq:LiD}\\
 	& \quad \eqref{eq:HD}_{\epsilon,\delta} \ar@{=>}[r]^{\epsilon\to 0^+} \ar[ul]^{\epsilon\to\infty} \quad & \quad \eqref{eq:LD}_\delta \ar[r]_{\delta\to 0^+} \ar[d]_{\delta\to\infty} \quad & \quad 0\\
 	& & \eqref{eq:LiD} &
	}
\]
\section{Diffusion limits with respect to $\delta$ and $k$}
\label{section:5}
In this section, we study the diffusion limits of solutions to problems~\eqref{eq:HDD} 
and \eqref{eq:HD} with respect to $\delta$. 
As a byproduct, we derive an explicit representation of the fundamental solution to problem~\eqref{eq:HDN} 
(the heat equation in $\Omega$ with a diffusive Neumann boundary condition), 
introduced in Section~1.
A classical solution to problem~\eqref{eq:HDN} is a function~$u$ defined 
on $(\overline\Omega\times[0,\infty))\setminus(\partial\Omega\times\{0\})$ such that 
\[
	u\in {C^{2;1}(\Omega\times(0,\infty))\cap C^{2;0}(\overline{\Omega}\times(0,\infty))
	\cap C(\Omega\times[0,\infty))}
\]
and $u$ satisfies the equations in \eqref{eq:HDN} pointwise.

We first investigate the diffusion limits of solutions to problems~\eqref{eq:HDD} 
and \eqref{eq:HD} in the limit $\delta\to 0^+$, 
and obtain the following theorem. 
Define 
\begin{equation}
\label{eq:5.1}
	\begin{aligned}
	\GHDN{\epsilon}{k}(x,y,t) & :=G_0\left(x,y,\frac{t}{\epsilon}\right)+\hat H(x,y,t),
	\\
	\hat H(x,y,t) & :=
	-2\int_0^\infty \Gamma_{N-1}\left(x'-y',\frac{t}{\epsilon}+k\tau\right)
	\partial_{x_N}\Gamma_1\left(x_N+y_N+\tau,\frac{t}{\epsilon}\right)\,\dee\tau,
	\end{aligned}
\end{equation}
for $(x,y,t)\in D$, where $\epsilon>0$ and $k\ge 0$.
Throughout this section, we set
\[
	Q(R):=\{(x,t)\in \overline{\Omega}\times(0,\infty)\,:\,x_N+t>R\},\qquad R>0.
\]
\begin{theorem}
\label{Theorem:5.1}
Let $(\phi,\psi)\in L^\infty(\Omega)\times L^\infty(\partial\Omega)$. 
Then, for any fixed $\epsilon>0$ and $k>0$, and for any $R>0$,
\begin{equation}
\label{eq:5.2}
	\begin{aligned}
 	& \sup_{(x,t)\in {Q(R)}}
	\left|\uHDD{\epsilon}{\delta}{k}(x,t)-\uHDN{\epsilon}{k}(x,t)\right|=O\left(\delta\right),
	\\
 	& \sup_{(x,t)\in {Q(R)}}
  	\left|\uHD{\epsilon}{\delta}(x,t)-\uHhN{\epsilon}(x,t)\right|=O\left(\delta\right),
	\end{aligned}
\end{equation}
as $\delta\to 0^+$.
Here, $\uHDN{\epsilon}{k}$ is a function on $\overline{\Omega}\times(0,\infty)$ defined by 
\begin{equation}
\label{eq:5.3}
	\uHDN{\epsilon}{k}(x,t):=\int_\Omega \GHDN{\epsilon}{k}(x,y,t)\phi(y)\,\dee y,\quad (x,t)\in \overline{\Omega}\times(0,\infty). 
\end{equation}
In particular, if $\phi\in BC(\Omega)$, 
then $\uHDN{\epsilon}{k}$ is a bounded classical solution to problem~\eqref{eq:HDN}. 
Furthermore, the following properties hold.
\begin{enumerate}[label={\rm(\arabic*)}]
\item
	Let $\epsilon>0$ be fixed. 
	Then, for any $R>0$, 
  	\begin{equation}
	\label{eq:5.4}
  		\sup_{(x,t)\in {Q(R)}}\left|\uHDN{\epsilon}{k}(x,t)-\uHhN{\epsilon}(x,t)\right|=O(k)
		\quad\mbox{as}\quad k\to 0^+.
  	\end{equation}
\item
  	Let $k>0$ be fixed. 
	If $\phi\in BC(\Omega)$,
  	then, for any $T>0$ and any compact set $K\subset\Omega$,
  	\[
		\begin{aligned}
		&
	  	\lim_{\epsilon\to\infty}\sup_{(x,t)\in K\times(0,T)}\left|\, \uHDN{\epsilon}{k}(x,t)-\phi(x)\,\right|=0,
		\\
		& 
	 	\lim_{\epsilon\to\infty}\sup_{(x,t)\in K\times(0,T)}\left|\, \uHhN{\epsilon}(x,t)-\phi(x)\,\right|=0.
		\end{aligned}
  	\]
\end{enumerate}
\end{theorem}
By Theorem~\ref{Theorem:5.1},
for any $\psi\in L^p(\partial\Omega)$ with $p\in[1,\infty]$,
the function $\uHDN{\epsilon}{k}$ can be interpreted as a solution to problem~\eqref{eq:HDN}.

For the proof of Theorem~\ref{Theorem:5.1}, we first state the following lemma, 
which shows that $\GHDN{\epsilon}{k}$ is the fundamental solution to problem \eqref{eq:HDN}
(resp.~problem~\eqref{eq:HhN}) when $k>0$ (resp.~$k=0$).
\begin{lemma}
\label{Lemma:5.1}
Let $k\ge 0$ and $\delta>0$ be fixed.
\begin{enumerate}[label={\rm(\arabic*)}]
	\item
    	$\GHDN{\epsilon}{k}(x,y,t)=\GHDN{\epsilon}{k}(y,x,t)$ for $(x,y,t)\in D$.
  	\item
    	For any fixed $y\in\Omega$, $\GHDN{\epsilon}{k}=\GHDN{\epsilon}{k}(x,y,t)$ satisfies 
  	\[
  		\left\{
  		\begin{array}{ll}
  		\epsilon\partial_t \GHDN{\epsilon}{k}-\Delta_x \GHDN{\epsilon}{k}=0 
		& \mbox{in}\quad\Omega\times(0,\infty),
		\vspace{5pt}\\
  		-k\Delta'_x \GHDN{\epsilon}{k}-\partial_{x_N} \GHDN{\epsilon}{k}=0 
		& \mbox{on}\quad \partial\Omega\times(0,\infty),
  		\end{array}
  		\right.
  	\]
  	as a function of the variables $(x,t)\in \overline{\Omega}\times(0,\infty)$. 
  	\item
  	$\GHDN{\epsilon}{k}$ satisfies
  	\begin{equation}
  	\label{eq:5.5}
  		\int_\Omega \GHDN{\epsilon}{k}(x,y,t)\,\dee y=1,\quad (x,t)\in \overline{\Omega}\times(0,\infty).
  	\end{equation}
  	Furthermore, if $\phi\in BC(\Omega)$, then  
	\begin{equation}\label{eq:5.6}
	\lim_{\substack{(z,t)\in Q,\\ (z,t)\to(x,0)}}
	\int_\Omega \GHDN{\epsilon}{k}(z,y,t)\phi(y)\,\dee y=\phi(x),\quad x\in\Omega.
	\end{equation}
	Here $Q=(\overline{\Omega}\times[0,\infty))\setminus(\Omega\times\{0\})$.
  	\item
  	If $k=0$, then 
  	\[
  		\GHDN{\epsilon}{k}(x,y,t)=G_N\left(x,y,\frac{t}{\epsilon}\right),\quad (x,y,t)\in D.
  	\]
\end{enumerate}
\end{lemma}
{\bf Proof.}
Assertion~(1) immediately follows from \eqref{eq:5.1}. 
For any fixed $y\in\Omega$, 
by \eqref{eq:5.1}, we have
\[
	\epsilon \partial_t \hat H(x,y,t)-\Delta_x \hat H(x,y,t)=0,
	\quad (x,t)\in\Omega\times(0,\infty),
\]
and
\[
	\begin{aligned}
 	& k\Delta_x' \hat H(x,y,t)+\partial_{x_N}\hat H(x,y,t)\,\biggr|_{x_N=0}
	\\
 	& =-2\int_0^\infty
	\partial_\tau\left(\Gamma_{N-1}\left(x'-y',\frac{t}{\epsilon}+k\tau\right)
	\partial_{y_N}\Gamma_1\left(y_N+\tau,\frac{t}{\epsilon}\right)\right)\,\dee\tau
	\\
 	& =2\Gamma_{N-1}\left(x'-y',\frac{t}{\epsilon}\right)
	\partial_{y_N}\Gamma_1\left(y_N,\frac{t}{\epsilon}\right)
 	=-\partial_{x_N}G_0\left(x,y,\frac{t}{\epsilon}\right)\,\biggr|_{x_N=0}
	\end{aligned}
\]
for $(x,t)\in \partial\Omega\times(0,\infty)$.
These, together with \eqref{eq:5.1}, imply assertion~(2).

We prove assertion~(3).  
It follows from property~($\Gamma$1), \eqref{eq:1.3}, and \eqref{eq:5.1}, that
\[
	\begin{aligned}
 	& \int_\Omega \GHDN{\epsilon}{k}(x,y,t)\,\dee y\\
 	 & =\int_0^\infty
	\left\{\Gamma_1\left(x_N-y_N,\frac{t}{\epsilon}\right) 
	-\Gamma_1\left(x_N+y_N,\frac{t}{\epsilon}\right)\right\}\,\dee y_N
	\\
 	& \qquad\qquad\qquad\qquad\qquad
 	-2\int_0^\infty\int_0^\infty 
	\partial_{y_N}\Gamma_1\left(x_N+y_N+\tau,\frac{t}{\epsilon}\right)\,\dee y_N\,\dee\tau
	\\
 	& =\left(\int_{-\infty}^{x_N}-\int_{x_N}^\infty\right)\Gamma_1\left(y_N,\frac{t}{\epsilon}\right)\,\dee y_N
 	+2\int_0^\infty \Gamma_1\left(x_N+\tau,\frac{t}{\epsilon}\right)\,\dee\tau\\
 	& =\int_{-\infty}^\infty \Gamma_1\left(\tau,\frac{t}{\epsilon}\right)\,\dee\tau=1
	\end{aligned}
\]
for $(x,t)\in \overline\Omega\times(0,\infty)$, 
which implies \eqref{eq:5.5}.
Similarly, it follows that, for $x\in\Omega$,
\[
	\begin{aligned}
	\int_\Omega \hat H(x,y,t)\,\dee y
 	& =-2\int_0^\infty\int_0^\infty 
	\partial_{y_N}\Gamma_1\left(x_N+y_N+\tau,\frac{t}{\epsilon}\right)\,\dee y_N\,\dee\tau
	\\
 	& =2\int_0^\infty \Gamma_1\left(x_N+\tau,\frac{t}{\epsilon}\right)\,\dee\tau\to 0
	\quad\mbox{as}\quad t\to 0^+,
	\end{aligned}
\]
which, together with the positivity of $\hat{H}$ and \eqref{eq:2.5}, implies that
\[
	\lim_{\substack{(z,t)\in Q,\\ (z,t)\to(x,0)}}\int_\Omega \GHDN{\epsilon}{k}(z,y,t)\phi(y)\,\dee y
	=\lim_{\substack{(z,t)\in Q,\\ (z,t)\to(x,0)}}\int_\Omega 
	G_0\left(z,y,\frac{t}{\epsilon}\right)\phi(y)\,\dee y=\phi(x),\quad x\in\Omega,
\]
if $\phi\in BC(\Omega)$. 
Hence, \eqref{eq:5.6} holds. 
Thus, assertion~(3) holds.

We prove assertion~(4). 
Let $k=0$. 
Since
\[
	\begin{aligned}
 	& \int_0^\infty \Gamma_{N-1}\left(x'-y',\frac{t}{\epsilon}\right)
	\partial_{x_N}\Gamma_1\left(x_N+y_N+\tau,\frac{t}{\epsilon}\right)\,\dee\tau
	\\
 	& =\int_0^\infty \Gamma_{N-1}\left(x'-y',\frac{t}{\epsilon}\right)
	\partial_\tau\Gamma_1\left(x_N+y_N+\tau,\frac{t}{\epsilon}\right)\,\dee\tau
	\\
 	& =-\Gamma_{N-1}\left(x'-y',\frac{t}{\epsilon}\right)\Gamma_1\left(x_N+y_N,\frac{t}{\epsilon}\right)
 	=-\Gamma_N\left(x-y^*,\frac{t}{\epsilon}\right),
	\end{aligned}
\]
by \eqref{eq:1.3}, \eqref{eq:2.10}, and \eqref{eq:5.1}, we obtain 
\[
	\GHDN{\epsilon}{k}(x,y,t)
 	=\Gamma_N\left(x-y,\frac{t}{\epsilon}\right)+\Gamma_N\left(-y^*,\frac{t}{\epsilon}\right)
 	=G_N\left(x,y,\frac{t}{\epsilon}\right)
\]
for $(x,y,t)\in D$. 
Thus, assertion~(4) holds. 
This completes the proof of Lemma~\ref{Lemma:5.1}.
$\Box$
\medskip

\noindent
{\bf Proof of Theorem~\ref{Theorem:5.1}.}
Fix $\epsilon>0$ and $k\ge 0$. 
Given $(\phi,\psi)\in L^\infty(\Omega)\times L^\infty(\partial\Omega)$,  
let $u$ be a solution to \eqref{eq:HDD} (resp. \eqref{eq:HD}) when $k>0$ (resp. $k=0$).
{If $\psi\in BC(\Omega)$, then Lemma~{\ref{Lemma:5.1}} implies that
the function $\uHDN{\epsilon}{k}$ defined by \eqref{eq:5.3} is a bounded classical solution to problem~{\eqref{eq:HDN}}.}

We prove \eqref{eq:5.2}. 
Let $L>0$ and $R>0$. 
Applying properties~($\Gamma$1) and ($\Gamma$2) to \eqref{eq:1.8}, we obtain
\[
	\begin{aligned}
	&\int_{\partial\Omega}H(x,y,t)\,\dee\sigma(y)
	\\
	& 
	\le C\left(\int_0^{t/2}+\int_{t/2}^t\right)\left(\frac{t-\tau}{\epsilon}\right)^{-1}
	\exp\left(-\frac{\epsilon(x_N+\tau/\delta)^2}{8(t-\tau)}\right)\,\dee\tau
	\\
	& 
	\le C\left(\frac{t}{2\epsilon}\right)^{-\frac{1}{2}}\exp\left(-\frac{\epsilon x_N^2}{8t}\right)
	\int_0^{t/2} \delta\left(\frac{\delta^2(t-\tau)}{\epsilon}\right)^{-\frac{1}{2}}
	\exp\left(-\frac{\epsilon \tau^2}{8\delta^2(t-\tau)}\right)\,\dee\tau
	\\
	& \qquad\quad
	+C\int_{t/2}^t \left(x_N+\frac{\tau}{\delta}\right)^{-2}
	\frac{\epsilon(x_N+\tau/\delta)^2}{t-\tau}
	\exp\left(-\frac{\epsilon(x_N+\tau/\delta)^2}{8(t-\tau)}\right)\,\dee\tau
	\\
	& \le C\delta\left(\frac{t}{2\epsilon}\right)^{-\frac{1}{2}}\exp\left(-\frac{\epsilon x_N^2}{8t}\right)
	+Ct\left(x_N+\frac{t}{2\delta}\right)^{-2}
	\end{aligned}
\]
for $(x,t)\in\overline{\Omega}\times(0,\infty)$. 
Since
\begin{equation}
\label{eq:5.7}
	\begin{aligned}
 	& \int_{\partial\Omega}H(x,y,t)\,\dee\sigma(y)
	\le C\delta\left(\frac{t}{2\epsilon}\right)^{-\frac{1}{2}}\exp\left(-\frac{\epsilon R^2}{32t}\right)
	 +C\delta\le C\delta\quad\mbox{if}\quad x_N\ge\frac{R}{2},
	 \\
 	& \int_{\partial\Omega}H(x,y,t)\,\dee\sigma(y)
 	\le C\delta \quad\mbox{if}\quad t\ge\frac{R}{2},
	\end{aligned}
\end{equation}
for $(x,t)\in{Q(R)}$ and $\delta\in(0,1)$, we obtain 
\begin{equation}
\label{eq:5.8}
	\int_{\partial\Omega}H(x,y,t)\,\dee\sigma(y)\le C\delta,\quad (x,t)\in Q(R),\quad \delta\in(0,1).
\end{equation}
It also follows from \eqref{eq:1.8} and \eqref{eq:5.1} that 
\[
	\frac{1}{\delta}H(x,y,t) =J_1(x,y,t)+J_2(x,y,t)+J_3(x,y,t)+\hat H(x,y,t)
\]
for $(x,y,t)\in D$, where
\[
	\begin{aligned}
	J_1(x,y,t) & :=
	-\frac{2}{\delta}\int_0^t 
	\left\{\Gamma_{N-1}\left(x'-y',\frac{t-\tau}{\epsilon}+\frac{k}{\delta}\tau\right)
	-\Gamma_{N-1}\left(x'-y',\frac{t}{\epsilon}+\frac{k}{\delta}\tau\right)\right\}
	\\
 	& \qquad\qquad
 	\times \partial_{x_N}\Gamma_1\left(x_N+y_N+\frac{\tau}{\delta},\frac{t-\tau}{\epsilon}\right)\,\dee\tau,
	\\
	J_2(x,y,t) & := -\frac{2}{\delta}\int_0^t \Gamma_{N-1}\left(x'-y',\frac{t}{\epsilon}+\frac{k}{\delta}\tau\right)
	\\
 	& \qquad\qquad
 	\times\left\{\partial_{x_N}\Gamma_1\left(x_N+y_N+\frac{\tau}{\delta},\frac{t-\tau}{\epsilon}\right)
	-\partial_{x_N}\Gamma_1\left(x_N+y_N+\frac{\tau}{\delta},\frac{t}{\epsilon}\right)\right\}\,\dee\tau,
	\\
 	J_3(x,y,t) & :=-\frac{2}{\delta}\int_0^t \Gamma_{N-1}\left(x'-y',\frac{t}{\epsilon}+\frac{k}{\delta}\tau\right)
 	\partial_{x_N}\Gamma_1\left(x_N+y_N+\frac{\tau}{\delta},\frac{t}{\epsilon}\right)\,\dee\tau
	-\hat H(x,y,t)
	\\
 	& =2\int_{t/\delta}^\infty \Gamma_{N-1}\left(x'-y',\frac{t}{\epsilon}+k\tau\right)
	\partial_{x_N}\Gamma_1\left(x_N+y_N+\tau,\frac{t}{\epsilon}\right)\,\dee\tau.
	\end{aligned}
\]
For $J_1$, since 
$$
	J_1(x,y,t)=\frac{2}{\delta}
	\int_0^t \int^{t/\epsilon}_{(t-\tau)/\epsilon} \partial_\xi\Gamma_{N-1}\left(x'-y',\xi+\frac{k}{\delta}\tau\right)
	\partial_{x_N}\Gamma_1\left(x_N+y_N+\frac{\tau}{\delta},\frac{t-\tau}{\epsilon}\right)\,\dee\xi\,\dee\tau
$$
for $(x,y,t)\in D$, 
by \eqref{eq:2.2}, we have
\[
	\begin{aligned}
 	& \int_\Omega |J_1(x,y,t)|\,dy
	\\
 	& \le \frac{C}{\delta}\int_0^t\int_0^\infty\int^{t/\epsilon}_{(t-\tau)/\epsilon} 
	\left(\xi+\frac{k}{\delta}\tau\right)^{-1}
	\left(-\partial_{y_N}\Gamma_1\left(x_N+y_N+\frac{\tau}{\delta},\frac{t-\tau}{\epsilon}\right)\right)
	\,\dee\xi\,\dee y_N\,\dee\tau
	\\
 	& \le\frac{C}{\delta}\int_0^t\frac{\epsilon}{t-\tau}\frac{\tau}{\epsilon}
	\Gamma_1\left(x_N+\frac{\tau}{\delta},\frac{t-\tau}{\epsilon}\right)\,\dee\tau
	\\
 	& \le \frac{C}{\epsilon\delta}\left(\int_0^{t/2}+ \int_{t/2}^t\right)\tau\left(\frac{t-\tau}{\epsilon}\right)^{-\frac{3}{2}}
	 \exp\left(-\frac{\epsilon(x_N+\tau/\delta)^2}{4(t-\tau)}\right)\,\dee\tau
	 \\
	 & \le\frac{C}{\delta}
	 \int_0^{t/2} \frac{\delta x_N+\tau}{\epsilon}\left(\frac{t}{2\epsilon}\right)^{-\frac{3}{2}}
	 \exp\left(-\frac{\epsilon(x_N+\tau/\delta)^2}{4(t-\tau)}\right)\,\dee\tau\\
	  & \qquad
	  +\frac{Ct}{\epsilon\delta}\int_{t/2}^t
	 \left(x_N+\frac{\tau}{\delta}\right)^{-3}\left(\frac{\epsilon(x_N+\tau/\delta)^2}{t-\tau}\right)^{\frac{3}{2}}
	 \exp\left(-\frac{\epsilon(x_N+\tau/\delta)^2}{4(t-\tau)}\right)\,\dee\tau\\
	 & \le\frac{C\delta}{\epsilon^{\frac{1}{2}}}t^{-\frac{1}{2}}\exp\left(-\frac{\epsilon x_N^2}{8t}\right)
	 \int_0^{t/2} \frac{\epsilon^{\frac{1}{2}}(\delta x_N+\tau)}{\delta t^{\frac{1}{2}}}
	 \left(\frac{\delta^2 t}{\epsilon}\right)^{-\frac{1}{2}}
	 \exp\left(-\frac{\epsilon(\delta x_N+\tau)^2}{8\delta^2t}\right)\,\dee\tau\\
	  & \hspace{8cm}
	  +\frac{Ct}{\epsilon\delta}\int_{t/2}^t\left(x_N+\frac{\tau}{\delta}\right)^{-3}\,\dee\tau\\
	  & \le\frac{C\delta}{\epsilon^{\frac{1}{2}}}t^{-\frac{1}{2}}\exp\left(-\frac{\epsilon x_N^2}{8t}\right)
	 \int_0^{t/2}\left(\frac{\delta^2 t}{\epsilon}\right)^{-\frac{1}{2}}
	 \exp\left(-\frac{\epsilon\tau^2}{16\delta^2t}\right)\,\dee\tau
	  +\frac{Ct}{\epsilon\delta}\left(x_N+\frac{t}{2\delta}\right)^{-2}\\
	  & \le \frac{C\delta}{\epsilon^{\frac{1}{2}}} t^{-\frac{1}{2}}\exp\left(-\frac{\epsilon x_N^2}{8t}\right)
	  +\frac{Ct}{\epsilon}\left(x_N+\frac{t}{2\delta}\right)^{-2}
	\end{aligned}
\]
for $(x,t)\in\overline{\Omega}\times(0,\infty)$. 
Then we argue as in \eqref{eq:5.7} to obtain 
\[
	\int_\Omega |J_1(x,y,t)|\,dy\le C\delta,\quad (x,t)\in Q(R),\quad\delta\in(0,1).
\]
For $J_2$, since
\[
	J_2(x,y,t)
	=\frac{2}{\delta}\int_0^t\int_{(t-\tau)/\epsilon}^{t/\epsilon} 
	\Gamma_{N-1}\left(x'-y',\frac{t}{\epsilon}+\frac{k}{\delta}\tau\right)
 	\partial_{x_N}\partial_\xi \Gamma_1\left(x_N+y_N+\frac{\tau}{\delta},\xi\right)\,\dee\xi\,\dee\tau
\]
for $(x,y,t)\in D$ and
\[
	|\partial_t\partial_\eta\Gamma_1(\eta,t))|
	\le Ct^{-\frac{5}{2}}\eta\left(1+\frac{\eta^2}{t}\right)e^{-\frac{\eta^2}{4t}}
	\le Ct^{-\frac{5}{2}}\eta e^{-\frac{\eta^2}{8t}}
	=-4Ct^{-\frac{3}{2}}\partial_\eta\left(e^{-\frac{\eta^2}{8t}}\right)
\]
for $(\eta,t)\in{\mathbb R}\times(0,\infty)$,
by property~($\Gamma$1), we obtain
\[
	\begin{aligned}
 	& \int_\Omega |J_2(x,y,t)|\,dy
	\\
 	& \le -\frac{C}{\delta}\int_0^t\int_{(t-\tau)/\epsilon}^{t/\epsilon}\int_0^\infty
 	\xi^{-\frac{3}{2}}\partial_{y_N}\exp\left(-\frac{(x_N+y_N+\tau/\delta)^2}{8\xi}\right)\,\dee y_N\,\dee\xi\,\dee\tau
	\\
  	& =\frac{C}{\delta}\left(\int_0^{t/2}+\int_{t/2}^t\right)\int_{(t-\tau)/\epsilon}^{t/\epsilon}
 	\xi^{-\frac{3}{2}}\exp\left(-\frac{(x_N+\tau/\delta)^2}{8\xi}\right)\,\dee\xi\,\dee\tau
	\\
	& \le\frac{C}{\delta}\exp\left(-\frac{\epsilon x_N^2}{8t}\right)
       \int_0^{t/2} \left(\frac{t}{2\epsilon}\right)^{-\frac{3}{2}}\frac{\tau}{\epsilon}
       \exp\left(-\frac{\epsilon\tau^2}{8\delta^2t}\right)\,\dee\tau\\
        & \qquad
        +\frac{C}{\delta}\int_{t/2}^t\int_{(t-\tau)/\epsilon}^{t/\epsilon}
       \left(x_N+\frac{\tau}{\delta}\right)^{-3}
       \left(\frac{(x_N+\tau/\delta)^2}{\xi}\right)^{\frac{3}{2}}\exp\left(-\frac{(x_N+\tau/\delta)^2}{8\xi}\right)\,\dee\xi\,\dee\tau\\
       & \le\frac{C}{\delta}\epsilon^{-\frac{1}{2}}\delta^2t^{-\frac{1}{2}}\exp\left(-\frac{\epsilon x_N^2}{8t}\right)
       \int_0^\infty \eta\exp(-\eta^2)\,\dee\eta
        +\frac{C}{\delta}\int_{t/2}^t\left(x_N+\frac{\tau}{\delta}\right)^{-3}\frac{\tau}{\epsilon}\,\dee\tau\\
        &\le \frac{C\delta}{\epsilon^{\frac{1}{2}}} t^{-\frac{1}{2}}\exp\left(-\frac{\epsilon x_N^2}{8t}\right)
         +\frac{Ct}{\epsilon}\left(x_N+\frac{t}{2\delta}\right)^{-2}
    \end{aligned}
\]
for $(x,t)\in\overline{\Omega}\times(0,\infty)$. 
Then we argue as in \eqref{eq:5.7} to obtain 
\[
	\int_\Omega |J_2(x,y,t)|\,dy\le C\delta,\quad (x,t)\in Q(R),\quad\delta\in(0,1).
\]
Similarly, 
for $J_3$, it follows from property~($\Gamma$1) that
\begin{equation}
\label{eq:5.9}
	\begin{aligned}
	&
	\int_\Omega |J_3(x,y,t)|\,\dee y
	\\
 	& =-2\int_{t/\delta}^\infty\int_0^\infty 
	\partial_{y_N}\Gamma_1\left(x_N+y_N+\tau,\frac{t}{\epsilon}\right)\,\dee y_N\,\dee\tau
	\\
 	& =2\int_{t/\delta}^\infty\Gamma_1\left(x_N+\tau,\frac{t}{\epsilon}\right)\,\dee\tau
 	\le C\left(\frac{t}{\epsilon}\right)^{-\frac{1}{2}}
	\int_{t/\delta}^\infty \exp\left(-\frac{\epsilon {{(x_N+\tau)}}^2}{4t}\right)\,\dee\tau
	\\
 	& \le C\exp\left(-\frac{\epsilon{(\delta x_N+t)^2}}{{8\delta^2t}}\right)
	\le C\exp\left(-\frac{\epsilon(\delta x_N+t)}{{8\delta^2}}\right)
	\end{aligned}
\end{equation}
for $(x,t)\in\overline{\Omega}\times(0,\infty)$. 
Since
\begin{align*}
 & \int_\Omega |J_3(x,y,t)|\,\dee y\le C\exp\left(-\frac{\epsilon R}{16\delta}\right)\quad\,\,\mbox{if}\quad x_N\ge\frac{R}{2},\\
 & \int_\Omega |J_3(x,y,t)|\,\dee y\le C\exp\left(-\frac{\epsilon R}{16\delta^2}\right)\quad\mbox{if}\quad t\ge\frac{R}{2},\\
\end{align*}
we obtain 
\[
	\int_\Omega |J_3(x,y,t)|\,\dee y\le C\exp\left(-\frac{\epsilon R}{16\delta}\right)\le C\delta,
\quad 
(x,t)\in Q(R),\quad \delta\in(0,1).
\]
Combining these estimates with \eqref{eq:1.8} and \eqref{eq:5.8}, we obtain 
\[
	\begin{aligned}
	u(x,t)
 	& =\int_\Omega G_0\left(x,y,\frac{t}{\epsilon}\right)\phi(y)\,\dee y
	+\int_\Omega \hat H(x,y,t)\phi(y)\,\dee y+O(\delta)
	\\
 	& =\int_\Omega \GHDN{\epsilon}{k}(x,y,t)\phi(y)\,\dee y+O(\delta)
	\quad\mbox{as}\quad\delta\to 0^+
	\end{aligned}
\]
uniformly for $(x,t)\in{Q(R)}$. 
Hence, \eqref{eq:5.2} holds. Here we use Lemma~\ref{Lemma:5.1}\,(4) when $k=0$. 

We prove assertion~(1). 
Let $\epsilon>0$ be fixed.
For any $T>0$,
it follows from Lemma~\ref{Lemma:5.1}\,(4) and \eqref{eq:5.1} that 
\begin{equation}
\label{eq:5.10}
	\begin{aligned}
 	& \GHDN{\epsilon}{k}(x,y,t)-\GHDN{\epsilon}{0}(x,y,t)
	\\
 	& =-2\int_0^\infty \left\{\Gamma_{N-1}\left(x'-y',\frac{t}{\epsilon}+k\tau\right)
	-\Gamma_{N-1}\left(x'-y',\frac{t}{\epsilon}\right)\right\}
	\partial_{x_N}\Gamma_1\left(x_N+y_N+\tau,\frac{t}{\epsilon}\right)\,\dee\tau
	\\
 	& =-2\int_0^\infty\int_0^{k\tau} \partial_\xi\Gamma_{N-1}\left(x'-y',\frac{t}{\epsilon}+\xi\right)
 	\partial_{x_N}\Gamma_1\left(x_N+y_N+\tau,\frac{t}{\epsilon}\right)\,\dee\xi\,\dee\tau
	\end{aligned}
\end{equation}
for $(x,y,t)\in D$ and $k>0$, which, together with \eqref{eq:2.2}, implies that
\[
	\begin{aligned}
 	& \int_\Omega\left|\GHDN{\epsilon}{k}(x,y,t)-\GHDN{\epsilon}{0}(x,y,t)
	\right| |\phi(y)|\,\dee y
	\\
 	& \le C\|\phi\|_{L^\infty(\Omega)}
	\int_0^\infty\int_0^\infty\int_0^{k\tau}\left(\frac{t}{\epsilon}+\xi\right)^{-1}
	\left(-\partial_{x_N}\Gamma_1\left(x_N+y_N+\tau,\frac{t}{\epsilon}\right)\right)\,\dee\xi\,\dee y_N\,\dee\tau
	\\
 	& \le \frac{Ck\epsilon}{t}\int_0^\infty\tau\Gamma_1\left(x_N+\tau,\frac{t}{\epsilon}\right)\,\dee\tau
 	\\
    	& \le Ck\left(\frac{t}{\epsilon}\right)^{-\frac{1}{2}}
    	\int_0^\infty \left(\frac{t}{\epsilon}\right)^{-\frac{1}{2}}\left(\frac{\epsilon\tau^2}{t}\right)^{\frac{1}{2}}
	\exp\left(-\frac{\epsilon(x_N+\tau)^2}{4t}\right)\,\dee\tau
    	\le Ck\left(\frac{t}{\epsilon}\right)^{-\frac{1}{2}}\exp\left(-\frac{\epsilon x_N^2}{4t}\right)
    	\le Ck
	\end{aligned}
\]
for $(x,t)\in{Q(R)}$ and $k>0$. 
Here we used a similar argument as in \eqref{eq:5.7}.
Thus assertion~(1) holds.

It remains to prove assertion~(2). 
Let $k>0$ be fixed.
Let $K$ be a compact set in $\Omega$ and $T>0$.
For $(x_N,y_N)\in(0,\infty)^2$ and $\tau\in(0,t)$, 
since $|x_N+y_N+\tau|^2\ge x_N^2+y_N^2+\tau^2$, 
by \eqref{eq:2.1}, we have
\[
	\begin{aligned}
	0 & \le-2\partial_{x_N}\Gamma_1\left(x_N+y_N+\tau,\frac{t}{\epsilon}\right)
 	\le C\left(\frac{t}{\epsilon}\right)^{-1}\exp\left(-\frac{\epsilon|x_N+y_N+\tau|^2}{8t}\right)
	\\
 	& \le C\exp\left(-\frac{\epsilon x_N^2}{8t}\right)
	\left(\frac{t}{\epsilon}\right)^{-\frac{1}{2}}
	\exp\left(-\frac{\epsilon \tau^2}{8t}\right)
 	\left(\frac{4\pi t}{\epsilon}\right)^{-\frac{1}{2}}\exp\left(-\frac{\epsilon y_N^2}{8t}\right).
	\end{aligned}
\]
Combining this with property~($\Gamma$1) and \eqref{eq:5.1}, we see that
\[
	\begin{aligned}
 	& \int_\Omega |\hat H(x,y,t)|\,\dee y
	\\
 	&  \le C\exp\left(-\frac{\epsilon x_N^2}{8t}\right)
 	\int_0^\infty\int_0^\infty\left(\frac{t}{\epsilon}\right)^{-\frac{1}{2}}\exp\left(-\frac{\epsilon \tau^2}{8t}\right)
	\left(\frac{4\pi t}{\epsilon}\right)^{-\frac{1}{2}}
	\exp\left(-\frac{\epsilon y_N^2}{8t}\right)\,\dee y_N\,\dee\tau
	\\
 	&  \le C\exp\left(-\frac{\epsilon x_N^2}{8t}\right)
	\le C\exp\left(-\frac{\epsilon x_N^2}{8T}\right)
	\to 0\quad\mbox{as}\quad\epsilon\to\infty
	\end{aligned}
\]
uniformly for $(x,t)\in K\times(0,T)$.
This, together with \eqref{eq:2.5}, \eqref{eq:5.1}, and \eqref{eq:5.3}, implies the desired conclusion. 
Thus, Theorem~\ref{Theorem:5.1} follows. 
$\Box$
\begin{remark}
\label{Remark:5.1}
{\rm (1)}
The convergence rate in \eqref{eq:5.2} is optimal. 
To see this, let $\epsilon>0$ to be fixed and 
let $J_1$, $J_2$, and $J_3$ be as in the proof of Theorem~{\rm\ref{Theorem:5.1}}.
Then
\[
	\begin{aligned}
  	J_1(x,y,t) 
   	&=\frac{2}{\delta}
   	\int_0^t\int_0^\tau \partial_t\Gamma_{N-1}
	\left(x'-y',\frac{t-\xi}{\epsilon}+\frac{k}{\delta}\tau\right)
	\partial_{x_N}\Gamma_1\left(x_N+y_N+\frac{\tau}{\delta},\frac{t-\tau}{\epsilon}\right)\,\dee\xi\,\dee\tau
	\\
  	&=2\delta\int_0^{t/\delta}\int_0^\tau \partial_t
	\Gamma_{N-1}\left(x'-y',\frac{t-\delta\xi}{\epsilon}+k\tau\right)
	\partial_{x_N}\Gamma_1\left(x_N+y_N+\tau,\frac{t-\delta\tau}{\epsilon}\right)\,\dee\xi\,\dee\tau,
	\\
  	J_2(x,y,t) 
  	& =\frac{2}{\delta}\int_0^t\int_0^\tau\Gamma_{N-1}\left(x'-y',\frac{t}{\epsilon}+\frac{k}{\delta}\tau\right)
  	\partial_{x_N}\partial_t\Gamma_1\left(x_N+y_N+\frac{\tau}{\delta},\frac{t-\xi}{\epsilon}\right)\,\dee\xi\,\dee\tau
	\\
  	&=2\delta\int_0^{t/\delta}\int_0^\tau \Gamma_{N-1}\left(x'-y',\frac{t}{\epsilon}+k\tau\right)
  	\partial_{x_N}\partial_t\Gamma_1\left(x_N+y_N+\tau,\frac{t-\delta\xi}{\epsilon}\right)\,\dee\xi\,\dee\tau,
	\end{aligned}
\]
for $(x,y,t)\in D$.
It follows from the Lebesgue dominated convergence theorem that
\begin{equation}
\label{eq:5.11}
	\begin{aligned}
    	&\lim_{\delta\to 0^+}\frac{1}{\delta}\left(J_1(x,y,t)+J_2(x,y,t)\right)
	\\
    	& =2\int_0^\infty\int_0^\tau \partial_t\Gamma_{N-1}\left(x'-y',\frac{t}{\epsilon}+k\tau\right)
	\partial_{x_N}\Gamma_1\left(x_N+y_N+\tau,\frac{t}{\epsilon}\right)\,\dee\xi\,\dee\tau
	\\
    	& \quad
     	+2\int_0^\infty\int_0^\tau \Gamma_{N-1}\left(x'-y',\frac{t}{\epsilon}+k\tau\right)
	\partial_{x_N}\partial_t\Gamma_1\left(x_N+y_N+\tau,\frac{t}{\epsilon}\right)\,\dee\xi\,\dee\tau
	\\
    	&=2\partial_t K(x,y,t),
	\end{aligned}
\end{equation}
for $(x,y,t)\in D$, where
\[
	K(x,y,t):=2\int_0^\infty\tau \Gamma_{N-1}\left(x'-y',\frac{t}{\epsilon}+k\tau\right)
	\partial_{x_N}\Gamma_1\left(x_N+y_N+\tau,\frac{t}{\epsilon}\right)\,\dee\tau.
\]
Since
\[
	{0<-K(0,0,t)
	\le} -Ct^{-\frac{N-1}{2}}
	\int_0^\infty \zeta\partial_\zeta\Gamma_1\left(\zeta,\frac{t}{\epsilon}\right)\,\dee\zeta
	=Ct^{-\frac{N-1}{2}}\int_0^\infty \Gamma_1\left(\zeta,\frac{t}{\epsilon}\right)\,\dee\zeta
	\le Ct^{-\frac{N-1}{2}},
\]
we find $t_*>0$ and $R>0$ such that
\begin{equation}
\label{eq:5.12}
	(\partial_t K)(x,y,t_*)>0\quad \textrm{if}\quad x, y\in B^+_R(0).
\end{equation}
Consider the case where $\phi:=\chi_{B^+_R(0)}$ in $\Omega$ and $\psi=0$ on $\partial\Omega$. 
Then, by \eqref{eq:5.9}, \eqref{eq:5.11}, and \eqref{eq:5.12}, 
we obtain 
\[
	\lim_{\delta\to 0^+}\frac{1}{\delta}\left(\uHDD{\epsilon}{\delta}{k}(x,t_*)
	-\uHDN{\epsilon}{k}(x,t_*)\right)=2\int_{B^+_R(0)}\partial_tK(x,y,t_*)\,\dee y>0
\]
for $x\in B^+_R(0)$. 
This shows that the convergence rate in \eqref{eq:5.2} is optimal.
\medskip

\noindent
{\rm (2)}
The convergence rate in \eqref{eq:5.4} is optimal. 
Indeed, consider the case where $\phi=\chi_{B^+_1(0)}$ in $\Omega$ and $\psi=0$ on $\partial\Omega$. 
Since 
\[
	-(\partial_t\Gamma_{N-1})(y',t)\ge Ct^{-\frac{N+1}{2}},\quad y'\in B'_1(0),
\]
for all sufficiently large $t>0$ {\rm({\it see} \eqref{eq:3.39})}, 
it follows from~\eqref{eq:5.10} that
\[
	\begin{aligned}
 	& \frac{1}{k}\left(\uHhN{\epsilon}(e_N,t)-\uHDN{\epsilon}{k}(e_N,t)\right)
	\\
 	& =\frac{1}{k}\int_\Omega \left(\GHDN{\epsilon}{0}(e_N,y,t)
	-\GHDN{\epsilon}{k}(e_N,y,t)\right)\phi(y)\,\dee y
	\\
 	& =2\int_{B^+_1(0)} \int_0^\infty
  	\left(\int_0^\tau(\partial_t\Gamma_{N-1})\left(y',\frac{t}{\epsilon}+k\xi\right)\,\dee\xi\right)
	(\partial_{x_N}\Gamma_1)\left(x_N+y_N+\tau,\frac{t}{\epsilon}\right)\,\dee\tau\,\dee y\,\biggr|_{x_N=1}
	\\
 	& \ge Ct^{-\frac{N+1}{2}}\int_{B^+_1(0)}
	\int_0^\infty-\tau(\partial_{\tau}\Gamma_1)\left(1+y_N+\tau,\frac{t}{\epsilon}\right)\,\dee\tau\,\dee y>0
	\end{aligned}
\]
for all sufficiently large $t>0$. 
This shows that the convergence rate in \eqref{eq:5.4} is optimal.
\end{remark}
Next, we study the behavior of solutions to problem~\eqref{eq:HDN}
in the limits $k\to\infty$, $\epsilon\to\infty$, and $\epsilon\to0^+$. 
For any $p\in[1,\infty)$, 
define a function $f_p$ on $(1,\infty)$ by 
\begin{equation}
\label{eq:5.13}
f_p(r):=\begin{cases}
   	r^{-1} & \mbox{if}\quad p<p_N, 
	\vspace{3pt}
	\\
   	r^{-1}\log r & \mbox{if}\quad p=p_N, 
	\vspace{3pt}
	\\
   	r^{-\frac{N-1}{2p}} & \mbox{if}\quad p>p_N,
 	\end{cases}
\end{equation}
for $r\in(1,\infty)$, where $p_N:=(N-1)/2$. 
\begin{theorem}
\label{Theorem:5.2}
Let $\phi\in L^p(\Omega)$, where $p\in[1,\infty)$.
Then, for any $T>0$,
\begin{equation}
\label{eq:5.14}
	\begin{aligned}
	&
	\sup_{(x,t)\in\overline{\Omega}\times(T,\infty)}|\uHDN{\epsilon}{k}(x,t)|
 	=O\!\left(\epsilon^{\frac{N}{2p}}\right)
 	\quad\mbox{as $\epsilon\to0^+$ for any fixed $k>0$},
	\\
	&
	\sup_{(x,t)\in\overline{\Omega}\times(T,\infty)}|\uHhN{\epsilon}(x,t)|
 	=O\!\left(\epsilon^{\frac{N}{2p}}\right)
 	\quad\mbox{as $\epsilon\to0^+$}.
	\end{aligned}
\end{equation}
Furthermore, for any $R>0$,
\begin{equation}
\label{eq:5.15}
	\sup_{(x,t)\in {Q(R)}}
	\left|\uHDN{\epsilon}{k}(x,t)-\uHDi{\epsilon}{0}(x,t)\right|
 	=O(f_p(k))\quad\mbox{as $k\to\infty$ for any fixed $\epsilon>0$}.
\end{equation}
\end{theorem}
{\bf Proof.}
It follows from \eqref{eq:2.2} and \eqref{eq:5.1} that
\begin{equation}
\label{eq:5.16}
	\begin{aligned}
	 & \|\HHDN(x,\cdot,t)\|_{L^{\frac{p}{p-1}}(\Omega)}\\
 	&
	\le 2\int_0^\infty
 	\left\|\Gamma_{N-1}\!\left(\cdot,\frac{t}{\epsilon}+k\tau\right)\right\|_{L^{\frac{p}{p-1}}(\mathbb{R}^{N-1})}
 	\left\|\partial_{x_N}\Gamma_1\!\left(x_N+\cdot+\tau,\frac{t}{\epsilon}\right)
 	\right\|_{{L^{\frac{p}{p-1}}(\mathbb R_+)}}\,\dee\tau
	\\
 	&
	\le C\int_0^\infty\left(\frac{t}{\epsilon}+k\tau\right)^{-\frac{N-1}{2p}}
 	\left(\frac{t}{\epsilon}\right)^{-\frac{1}{2p}-\frac{1}{2}}
 	\exp\!\left(-\frac{\epsilon(x_N+\tau)^2}{16t}\right)\,\dee\tau
	\\
	&
	\le C\left(\frac{t}{\epsilon}\right)^{-\frac{N}{2p}-\frac{1}{2}}
	\int_0^\infty\exp\!\left(-\frac{\epsilon\tau^2}{16t}\right)\,\dee\tau
	\le C\left(\frac{t}{\epsilon}\right)^{-\frac{N}{2p}}
	\end{aligned}
\end{equation}
for $(x,t)\in\overline{\Omega}\times(0,\infty)$, $k\ge0$, and $\epsilon>0$.
Then, by \eqref{eq:5.1}, \eqref{eq:5.3}, and \eqref{eq:5.16}, we have
\begin{equation}
\label{eq:5.17}
	 \left|\uHDN{\epsilon}{k}(x,t)
 	-\uHDi{\epsilon}{0}(x,t)\right|
	\le
 	\|\HHDN(x,\cdot,t)\|_{L^{\frac{p}{p-1}}(\Omega)}\,\|\phi\|_{L^p(\Omega)}
	\le C\,\epsilon^{\frac{N}{2p}}
\end{equation}
for $(x,t)\in\overline{\Omega}\times(T,\infty)$.
This, together with \eqref{eq:2.5}, implies \eqref{eq:5.14}.
Moreover, for any fixed $\epsilon>0$ and any $R>0$,
if $p\le p_N$, then
\[
	\begin{aligned}
 	&
	\int_0^\infty\left(\frac{t}{\epsilon}+k\tau\right)^{-\frac{N-1}{2p}}
 	\left(\frac{t}{\epsilon}\right)^{-\frac{1}{2p}-\frac{1}{2}}
 	\exp\!\left(-\frac{\epsilon(x_N+\tau)^2}{8t}\right)\,\dee\tau
	\\
 	&
	\le C\left(\frac{t}{\epsilon}\right)^{-\frac{1}{2p}-\frac{1}{2}}\exp\left(-\frac{\epsilon x_N^2}{8t}\right)\\
 	 & \qquad\times
	 \left\{\int_0^t\left(\frac{t}{\epsilon}+k\tau\right)^{-\frac{N-1}{2p}}\,\dee\tau
 	+k^{-\frac{N-1}{2p}}\int_t^\infty\tau^{-\frac{N-1}{2p}}
 	\exp\!\left(-\frac{\epsilon\tau^2}{8t}\right)\,\dee\tau
 	\right\}
	\\
 	&
	\le Cf_p(k),
	\end{aligned}
\]
Here we used the facts that $-(N-1)/(2p)< -1$ if $p<p_N$ and $-(N-1)/(2p)=-1$ if $p=p_N$. 
On the other hand, if $p>p_N$, then
\begin{equation}
\label{eq:5.18}
	\begin{aligned}
 	&
	\int_0^\infty\left(\frac{t}{\epsilon}+k\tau\right)^{-\frac{N-1}{2p}}
 	\left(\frac{t}{\epsilon}\right)^{-\frac{1}{2p}-\frac{1}{2}}
 	\exp\!\left(-\frac{\epsilon(x_N+\tau)^2}{8t}\right)\,\dee\tau
	\\
 	&
	\le C\left(\frac{t}{\epsilon}\right)^{-\frac{1}{2p}-\frac{1}{2}}\exp\left(-\frac{\epsilon x_N^2}{8t}\right)
 	\int_0^\infty\left(k\tau\right)^{-\frac{N-1}{2p}}
 	\exp\!\left(-\frac{\epsilon\tau^2}{8t}\right)\,\dee\tau
	\\
 	&
	\le Ck^{-\frac{N-1}{2p}}t^{-\frac{N+1}{4p}}\exp\left(-\frac{\epsilon x_N^2}{8t}\right)
	\int_0^\infty\eta^{-\frac{N-1}{2p}}\exp\!\left(-\eta^2\right)\,\dee\eta
	\le Cf_p(k),
	\end{aligned}
\end{equation}
for $(x,t)\in Q(R)$. 
Here we used the {fact} that $-(N-1)/(2p)>-1$ and argued as in \eqref{eq:5.7}.
These, together with \eqref{eq:5.16} and \eqref{eq:5.17}, imply \eqref{eq:5.15}.
The proof is complete.
$\Box$\medskip

Next, we study the diffusion limit of solutions to problems~\eqref{eq:HDD} and \eqref{eq:HD} 
in the limit $\delta\to\infty$.
\begin{theorem}
\label{Theorem:5.3}
Let $(\phi,\psi)\in L^\infty(\Omega)\times L^\infty(\partial\Omega)$. 
Let $\epsilon>0$ and $k>0$ be fixed.
For any $L>0$ and $T>0$,
\begin{equation}
\label{eq:5.19}
	\begin{aligned}
 	& \sup_{(x,t)\in \Omega_L^c\times(0,T)}|\uHDD{\epsilon}{\delta}{k}(x,t)-\uHDi{\epsilon}{\psi}(x,t)|
	=O\left(\delta^{-1}\right),
	\\
 	& \sup_{(x,t)\in \Omega_L^c\times(0,T)}|\uHD{\epsilon}{\delta}(x,t)-\uHDi{\epsilon}{\psi}(x,t)|
	=O\left(\delta^{-1}\right),
	\end{aligned}
\end{equation}
as $\delta\to\infty$. 
Furthermore, the following properties hold. 
\begin{enumerate}[label={\rm(\arabic*)}]
\item
 	For any $L>0$ and $T>0$, 
  	\begin{equation}
	\label{eq:5.20}
  		\sup_{(x,t)\in\Omega_L\times[T,\infty)}|\uHDi{\epsilon}{\psi}(x,t)-\uLDi{\psi}(x)|
		 =O\left(\epsilon^{\frac{1}{2}}\right)
  		\quad\mbox{as}\quad\epsilon\to 0^+,
  	\end{equation}
	where $\uLDi{\psi}$ is a solution to \eqref{eq:LiD}.
\item
	If $(\phi,\psi)\in BC(\Omega)\times L^\infty(\partial\Omega)$, then 
  	$$
  		\lim_{\epsilon\to\infty}\sup_{(x,t)\in K\times(0,T)}|\uHDi{\epsilon}{\psi}(x,t)-\phi(x)|=0
  	$$ 
  	for any compact set $K\subset\Omega$ and $T>0$. 
\end{enumerate}
\end{theorem}
{\bf Proof.}
Let $\epsilon>0$ and $k\ge 0$ be fixed.
Let $(\phi,\psi)\in L^\infty(\Omega)\times L^\infty(\partial\Omega)$. 
It follows from \eqref{eq:1.8} that
\[
	H(x,y,t)=J_1(x,y,t)+J_2(x,y,t)
	-2\int_0^t\Gamma_{N-1}\left(x'-y',\frac{t-\tau}{\epsilon}\right)
	\partial_{x_N}\Gamma_1\left(x_N,\frac{t-\tau}{\epsilon}\right)\,\dee\tau
\]
for $(x,y,t)\in\overline{\Omega}\times\partial\Omega\times(0,\infty)$, 
where
\[
	\begin{aligned}
	J_1(x,y,t)
 	& :=-2\int_0^t \left\{\Gamma_{N-1}\left(x'-y',\frac{t-\tau}{\epsilon}+\frac{k}{\delta}\tau\right)
	-\Gamma_{N-1}\left(x'-y',\frac{t-\tau}{\epsilon}\right)\right\}
	\\
 	& \qquad\quad
 	\times\partial_{x_N}\Gamma_1\left(x_N+\frac{\tau}{\delta},\frac{t-\tau}{\epsilon}\right)\,\dee\tau
	\\
 	& =-2\int_0^t\int^{k\tau/\delta}_0 \partial_\xi\Gamma_{N-1}\left(x'-y',\frac{t-\tau}{\epsilon}+\xi\right)
	\partial_{x_N}\Gamma_1\left(x_N+\frac{\tau}{\delta},\frac{t-\tau}{\epsilon}\right)\,\dee\xi\,\dee\tau,
	\\
	J_2(x,y,t)
 	& :=-2\int_0^t \Gamma_{N-1}\left(x'-y',\frac{t-\tau}{\epsilon}\right)
	\\
 	& \qquad\quad
 	\times\left\{\partial_{x_N}\Gamma_1\left(x_N+\frac{\tau}{\delta},\frac{t-\tau}{\epsilon}\right)
	-\partial_{x_N}\Gamma_1\left(x_N,\frac{t-\tau}{\epsilon}\right)\right\}\,\dee\tau
	\\
 	& =-2\int_0^t\int^{\tau/\delta}_0 \Gamma_{N-1}\left(x'-y',\frac{t-\tau}{\epsilon}\right)
	\partial_\xi\partial_{x_N}\Gamma_1\left(x_N+\xi,\frac{t-\tau}{\epsilon}\right)\,\dee\xi\,\dee\tau.
	\end{aligned}
\]
Let $L>0$ and $T>0$.
Then, by property~($\Gamma$2), we have
\[
	\begin{aligned}
 	& \int_{\partial\Omega}|J_1(x,y,t)|\,\dee\sigma(y)
	\\
 	& \le C\int_0^t\int^{kt/\delta}_0\left(\frac{t-\tau}{\epsilon}+\xi\right)^{-1}
	\left|\partial_{x_N}\Gamma_1\left(x_N+\frac{\tau}{\delta},\frac{t-\tau}{\epsilon}\right)\right|\,\dee\xi\,\dee\tau
	\\
 	& \le C\frac{kt}{\delta}\int_0^t\left(\frac{t-\tau}{\epsilon}\right)^{-2}
	\exp\left(-\frac{\epsilon x_N^2}{8(t-\tau)}\right)\,\dee\tau
 	\le C\delta^{-1},
	\\
 	& \int_{\partial\Omega}|J_2(x,y,t)|\,\dee\sigma(y)
	\le \frac{Ct}{\delta}\int_0^t\left(\frac{t-\tau}{\epsilon}\right)^{-\frac{3}{2}}
	\exp\left(-\frac{\epsilon x_N^2}{8(t-\tau)}\right)\,\dee\tau\le C\delta^{-1}{,}
	\end{aligned}
\]
for $(x,t)\in \Omega_L^c\times(0,T)$ and $\delta>0$.
These estimates imply that
\[
	\begin{aligned}
 	& \int_{\partial\Omega}H(x,y,t)\psi(y)\,\dee\sigma(y)
	\\
 	& =-2\int_0^t\int_{\partial\Omega}\Gamma_{N-1}\left(x'-y',\frac{t-\tau}{\epsilon}\right)
	\partial_{x_N}\Gamma_1\left(x_N,\frac{t-\tau}{\epsilon}\right)\psi(y)\,\dee\sigma(y)\,\dee\tau
	+O\left(\delta^{-1}\right)
	\end{aligned}
\]
as $\delta\to\infty$ uniformly for $(x,t)\in \Omega_L^c\times (0,T)$. 
On the other hand, by \eqref{eq:3.7}, we have
\[
	\frac{1}{\delta}\int_\Omega H(x,y,t)\phi(y)\,\dee y=O(\delta^{-1})\quad\mbox{as}\quad \delta\to\infty, 
\]
uniformly for $(x,t)\in \Omega_L^c\times (0,T)$.
Combining these with \eqref{eq:1.8}, \eqref{eq:1.9}, and \eqref{eq:2.3},
we obtain \eqref{eq:5.19}.

Next, we prove assertion~(1). 
It follows from Lemma~\ref{Lemma:4.1}~(4), \eqref{eq:2.3}, and \eqref{eq:2.11} that 
\begin{equation}
\label{eq:5.21}
	\begin{aligned}
 	& \uHDi{\epsilon}{\psi}(x,t)-\uHDi{\epsilon}{0}(x,t)
	\\
 	& =-2\int_{\partial\Omega}\int_0^{t/\epsilon}\Gamma_{N-1}\left(x'-y',\tau\right)
	\partial_{x_N}\Gamma_1\left(x_N,\tau\right)\psi(y)\,\dee\tau\,\dee\sigma(y)\\
 	& =\int_{\partial\Omega}P(x'-y',x_N)\psi(y)\,\dee\sigma(y)
	+\int_{\partial\Omega}J_3(x,y,t)\psi(y)\,\dee\sigma(y)\\
	& =\uLDi{\psi}(x)+\int_{\partial\Omega}J_3(x,y,t)\psi(y)\,\dee\sigma(y)
 	\end{aligned}
\end{equation}
for {$(x,t)\in\Omega\times(0,\infty)$}, 
where
\[
	J_3(x,y,t):=2\int_{t/\epsilon}^\infty \Gamma_{N-1}\left(x'-y',\tau\right)
	\partial_{x_N}\Gamma_1\left(x_N,\tau\right)\,\dee\tau.
\]
Let $L>0$ and $T>0$. 
Since $\psi\in L^\infty(\partial\Omega)$, by property~($\Gamma$1) we have
\[
	\int_{\partial\Omega}|J_3(x,t)|\,\dee\sigma(y)
	\le C\int_{t/\epsilon}^\infty \tau^{-\frac{3}{2}}x_N\exp\left(-\frac{x_N^2}{4\tau}\right)\,\dee\tau
	\le C\int_{t/\epsilon}^\infty \tau^{-\frac{3}{2}}x_N\,\dee\tau
	\le C\left(\frac{\epsilon}{T}\right)^{\frac{1}{2}}L
	\le C\epsilon^{\frac{1}{2}}
\]
for $x\in\Omega_L$ and $t\in(T,\infty)$.
This, together with \eqref{eq:2.6} and \eqref{eq:5.21}, implies assertion~(1).

In remains to prove assertion~(2). 
By property~($\Gamma$1), \eqref{eq:2.1}, and \eqref{eq:2.3}, we have
\[
	\begin{aligned}
	\left|\uHDi{\epsilon}{\psi}(x,t)-\uHDi{\epsilon}{0}(x,t)\right|
 	& \le \frac{C\|\psi\|_{L^\infty(\partial\Omega)}}{\epsilon}\int_0^t
	\left(\frac{t-\tau}{\epsilon}\right)^{-1}\exp\left(-\frac{\epsilon x_N^2}{8(t-\tau)}\right)\,\dee\tau
	\\
 	& =C\|\psi\|_{L^\infty(\partial\Omega)}\int_0^{t/\epsilon}\tau^{-1}\exp\left(-\frac{L^2}{8\tau}\right)\,\dee\tau
	\le C\frac{T}{\epsilon}
	\le C\epsilon^{-1}
	\end{aligned}
\]
for $x\in\Omega_L^c$, $t\in(0,T)$, and $\epsilon\in[1,\infty)$.
This, together with \eqref{eq:2.5}, implies assertion~(2).
Thus, Theorem~\ref{Theorem:5.3} follows.
$\Box$
\begin{remark}
\label{Remark:5.2}
{\rm (1)}
The convergence rate in \eqref{eq:5.19} is optimal. 
Indeed, consider the case where $\phi=1$ on $\Omega$ and $\psi=0$ on $\partial\Omega$. 
Then, if $k>0$, by \eqref{eq:1.12} and \eqref{eq:3.7}, we have
\[
	\begin{aligned}
 	& \delta\left(\uHDD{\epsilon}{\delta}{k}(x,t)-\uHDi{\epsilon}{\psi}(x,t)\right)
 	=\delta\left(\uHDD{\epsilon}{\delta}{k}(x,t)-\uHDi{\epsilon}{0}(x,t)\right)
	\\
 	& =\int_\Omega H(x,y,t)\,\dee y=2\int_0^t \left(\frac{4\pi(t-\tau)}{\epsilon}\right)^{-\frac{1}{2}}
	\exp\left(-\frac{\epsilon(x_N+\tau/\delta)^2}{4(t-\tau)}\right)\,\dee\tau\\
	&\to
	2\int_0^t \left(\frac{4\pi(t-\tau)}{\epsilon}\right)^{-\frac{1}{2}}
	\exp\left(-\frac{\epsilon x_N^2}{4(t-\tau)}\right)\,\dee\tau>0
	\quad\mbox{as}\quad \delta\to\infty
	\end{aligned}
\]
for $(x,t)\in\Omega\times(0,\infty)$. 
This shows that the convergence rate in \eqref{eq:5.19} is optimal when $k>0$. 
The case $k=0$ can be treated in a similar way.
\medskip

\noindent
{\rm (2)} 
The convergence rate in \eqref{eq:5.20} is optimal. 
Indeed, consider the case where $\phi=0$ on $\Omega$ and $\psi=1$ on $\partial\Omega$. 
Then $\uHDi{\epsilon}{0}\equiv 0$ in $\Omega\times(0,\infty)$, 
which, together with property~{\rm ($\Gamma$1)} and \eqref{eq:5.21}, implies that
\[
	\begin{aligned}
	|\uHDi{\epsilon}{\psi}(x,t)-\uLDi{\psi}(x,t)|
 	& =2\left|\int_{\partial\Omega}\left(\int_{t/\epsilon}^\infty 
	\Gamma_{N-1}(x'-y',\tau)\partial_{x_N}\Gamma_1(x_N,\tau)\,\dee\tau\right)\,\dee\sigma(y)\right|
	\\
 	& =2\left|\int_{t/\epsilon}^\infty \partial_{x_N}\Gamma_1(x_N,\tau)\,\dee\tau\right|
 	=\int_{t/\epsilon}^\infty (4\pi\tau)^{-\frac{1}{2}}\frac{x_N}{\tau}
	\exp\left(-\frac{x_N^2}{4\tau}\right)\,\dee\tau
	\\
 	& \ge C\frac{x_N}{t^{\frac{1}{2}}}\exp\left(-\frac{x_N^2}{4t}\right)\epsilon^{\frac{1}{2}}
	\end{aligned}
\]
for $(x,t)\in\Omega\times(0,\infty)$ and $\epsilon\in(0,1)$. 
This shows that the convergence rate in \eqref{eq:5.20} is optimal.
\end{remark}

By now, we have studied the diffusion limits of solutions to
problems~\eqref{eq:HDD} and \eqref{eq:HD} with respect to $\epsilon$ and $\delta$.
In the remainder of this paper, we discuss the limit $k\to\infty$, as well as
the simultaneous limit $k\to\infty$ and $\delta\to0^+$ under the condition that
$k/\delta>0$ is fixed.
The latter corresponds to the diffusion limit as $\nu\to0^+$ for the heat
equation with the diffusive dynamical boundary condition
\[
	\delta\partial_t u - k\Delta' u - \nu\partial_{x_N}u = 0
	\quad\mbox{on}\,\,\,\partial\Omega\times(0,\infty).
\]
\begin{theorem}
\label{Theorem:5.4}
Let $(\phi,\psi)\in L^p(\Omega)\times L^p(\partial\Omega)$, where $p\in[1,\infty)$.
Let $\epsilon>0$ and $\delta>0$ be fixed.
Then, for any $R>0$,
\begin{equation}
\label{eq:5.22}
	\sup_{(x,t)\in {Q(R)}}
	\left|\uHDD{\epsilon}{\delta}{k}(x,t)-\uHDi{\epsilon}{0}(x,t)\right|
	=O(f_p(k))
\end{equation}
as $k\to\infty$, where the function $f_p$ is as in \eqref{eq:5.13}.
\end{theorem}
{\bf Proof.}
Let $\epsilon>0$ and $\delta>0$ be fixed. 
Let $p\in[1,\infty)$ and $R>0$.
By \eqref{eq:1.12} and \eqref{eq:2.4}, and applying H\"older's inequality, we see that
\begin{equation}
\label{eq:5.23}
	\begin{aligned}
	& \sup_{(x,t)\in Q(R)}
	\left|\uHDD{\epsilon}{\delta}{k}(x,t)-\uHDi{\epsilon}{0}(x,t)\right|\\
	& \le \sup_{(x,t)\in Q(R)} \bigg\{\|H(x,\cdot,t)\|_{L^{\frac{p}{p-1}}(\Omega)}\|\phi\|_{L^p(\Omega)} 
	+\|H(x,\cdot,t)\|_{L^{\frac{p}{p-1}}(\partial\Omega)}\|\psi\|_{L^p(\partial\Omega)}\bigg\}.
	\end{aligned} 
\end{equation} 

On the other hand, similarly to \eqref{eq:5.16}, by \eqref{eq:1.8} and {\eqref{eq:2.2}, 
we have
\begin{equation}
\label{eq:5.24}
	\begin{aligned}
	& \|H(x,\cdot,t)\|_{L^{\frac{p}{p-1}}(\Omega)}
	\\
 	& \le 2\int_0^t
	\left\|\Gamma_{N-1}\!\left(x'-\cdot,\frac{t-\tau}{\epsilon}+\frac{k}{\delta}\tau\right)\right\|_{L^{\frac{p}{p-1}}(\mathbb R^{N-1})}
 	\left\|\partial_{x_N}\Gamma_1\!\left(x_N+\cdot+\frac{\tau}{\delta},
 	\frac{t-\tau}{\epsilon}\right)\right\|_{{L^{\frac{p}{p-1}}(\mathbb R_+)}}\,\dee\tau
	\\
	& \le \int_0^t I(t,\tau)\,\dee\tau,
	\end{aligned}
\end{equation}
for $(x,t)\in\Omega\times(0,\infty)$, 
where 	
\[
	I(t,\tau):=\left(\frac{t-\tau}{\epsilon}+\frac{k}{\delta}\tau\right)^{-\frac{N-1}{2p}}
 	\left(\frac{t-\tau}{\epsilon}\right)^{-\frac{1}{2p}-\frac{1}{2}}
 	\exp\!\left(-\frac{\epsilon(\delta x_N+\tau)^2}{16\delta^2(t-\tau)}\right).
\]
We claim that 
\begin{equation}
\label{eq:5.25}
	\int_0^{t/2}I(t,\tau)\,\dee\tau\le Cf(k)
\end{equation}
for $(x,t)\in Q(R)$ and $k\ge 2$.
In fact, if $p<p_N$, 
then
\[ 
	\begin{aligned}
	\int_0^{t/2}I(t,\tau)\,\dee\tau
 	& \le Ct^{-\frac{1}{2p}-\frac{1}{2}}\exp\left(-\frac{\epsilon x_N^2}{16t}\right)
	\int_0^{t/2}\left(\frac{t}{2\epsilon}+\frac{k}{\delta}\tau\right)^{-\frac{N-1}{2p}}\,\dee\tau
	\\
 	& \le Ct^{-\frac{1}{2p}-\frac{1}{2}}\exp\left(-\frac{\epsilon x_N^2}{16t}\right)\frac{\delta}{k}
	\left(\frac{t}{2\epsilon}\right)^{-\frac{N-1}{2p}+1}
	\end{aligned}
\]
for $(x,t)\in\Omega\times(0,\infty)$ and $k>0$. 
Then we argue as in \eqref{eq:5.7} to obtain
$$
\int_0^{t/2}I(t,\tau)\,\dee\tau \le Ck^{-1}
$$
for $(x,t)\in Q(R)$ and $k\ge 2$ when $p<p_N$. 
Similarly, if $p=p_N$, then
\[
	\int_0^{t/2}I(t,\tau)\,\dee\tau
	\le Ct^{-\frac{1}{2p}-\frac{1}{2}}\exp\left(-\frac{\epsilon x_N^2}{16t}\right)
	\frac{\delta}{k}\log\left(1+\frac{\epsilon k}{\delta}\right)
 	\le Ck^{-1}\log k
\]
for $(x,t)\in Q(R)$ and $k\ge 2$.
If $p>p_N$, then we argue as in \eqref{eq:5.18} to obtain
\[
	\begin{aligned}
	\int_0^{t/2}I(t,\tau)\,\dee\tau
	& \le Ct^{-\frac{1}{2p}-\frac{1}{2}}\exp\left(-\frac{\epsilon x_N^2}{16t}\right)
	\int_0^{t/2}\left(k\tau\right)^{-\frac{N-1}{2p}}
	\exp\left(-\frac{\epsilon\tau^2}{16\delta^2(t-\tau)}\right)\,\dee\tau
	\\
	& \le Ck^{-\frac{N-1}{2p}}t^{-\frac{1}{2p}-\frac{1}{2}}\exp\left(-\frac{\epsilon x_N^2}{16t}\right)
	\int_0^{t/2}\tau^{-\frac{N-1}{2p}}
	\exp\left(-\frac{\epsilon\tau^2}{16\delta^2t}\right)\,\dee\tau
	\le Ck^{-\frac{N-1}{2p}}
	\end{aligned}
\]
for $(x,t)\in Q(R)$ and $k\ge 2$. Thus \eqref{eq:5.25} holds. 

In contrast, 
we observe that
\[
	\begin{aligned}
 	& \left(\frac{t-\tau}{\epsilon}\right)^{-\frac{1}{2p}-\frac{1}{2}}
	\exp\left(-\frac{\epsilon(\delta x_N+\tau)^2}{32\delta^2(t-\tau)}\right)
	\\
 	& =(\delta x_N+\tau)^{-\frac{1}{p}-1}
	\left(\frac{t-\tau}{\epsilon(\delta x_N+\tau)^2}\right)^{-\frac{1}{2p}-\frac{1}{2}}
	\exp\left(-\frac{\epsilon(\delta x_N+\tau)^2}{32\delta^2(t-\tau)}\right)
 	\le C(\delta x_N+\tau)^{-\frac{1}{p}-1}
	\end{aligned}
\]
for $(x,\tau)\in\Omega\times(t/2,t)$ with $x_N+t\ge R$, 
which yields 
\begin{equation}
\label{eq:5.26}
	\begin{aligned}
	\int_{t/2}^tI(t,\tau)\,\dee\tau
 	& \le C k^{-\frac{N-1}{2p}}{t^{-\frac{N-1}{2p}}
	\exp\left(-\frac{\epsilon x_N^2}{32t}\right)}\int_{t/2}^t(\delta x_N+\tau)^{-\frac{1}{p}-1}\,\dee\tau\\
 	& \le C k^{-\frac{N-1}{2p}}{t^{-\frac{N-1}{2p}}\exp\left(-\frac{\epsilon x_N^2}{32t}\right)}t^{-\frac{1}{p}}
	\le C k^{-\frac{N-1}{2p}}
	\end{aligned}
\end{equation}
for $(x,t)\in Q(R)$ and $k\ge 2$. 
Here we argued as in \eqref{eq:5.7}.
Combining \eqref{eq:5.24}, \eqref{eq:5.25}, and \eqref{eq:5.26}, we obtain 
\begin{equation}
\label{eq:5.27}
	\|H(x,\cdot,t)\|_{L^{\frac{p}{p-1}}(\Omega)}\le Cf_p(k)
\end{equation}
for $(x,t)\in Q(R)$ and $k\ge 2$.
Similarly, 
\[
	\begin{aligned}
 	& \|H(x,\cdot,t)\|_{L^{\frac{p}{p-1}}(\partial\Omega)}
	\\
 	& \le
 	2\int_0^t\left\|\Gamma_{N-1}\!\left(x'-\cdot,\frac{t-\tau}{\epsilon}+\frac{k}{\delta}\tau\right)
 	\right\|_{L^{\frac{p}{p-1}}(\mathbb{R}^{N-1})}
 	\left\|\partial_{x_N}\Gamma_1\!\left(x_N+\cdot+\frac{\tau}{\delta},
 	\frac{t-\tau}{\epsilon}\right)\right\|_{{L^\infty(\mathbb R_+)}}\,\dee\tau
	\\
 	& \le
 	C\left(\int_0^{t/2}+\int_{t/2}^t\right)\left(\frac{t-\tau}{\epsilon}+\frac{k}{\delta}\tau\right)^{-\frac{N-1}{2p}}
 	\left(\frac{t-\tau}{\epsilon}\right)^{-1}
 	\exp\left(-\frac{\epsilon(\delta x_N+\tau)^2}{8\delta^2(t-\tau)}\right)\,\dee\tau
	\\
 	& \le
 	C\exp\left(-\frac{\epsilon x_N^2}{8t}\right)t^{-1}
	\int_0^{t/2}\left(\frac{t-\tau}{\epsilon}+\frac{k}{\delta}\tau\right)^{-\frac{N-1}{2p}}
	\exp\left(-\frac{\epsilon\tau^2}{8\delta^2(t-\tau)}\right)\,\dee\tau
	\\
	&\qquad
	+Ck^{-\frac{N-1}{2p}}{t^{-\frac{N-1}{2p}}\exp\left(-\frac{\epsilon x_N}{16t}\right)}
	\int_{t/2}^t(\delta x_N+\tau)^{-2}\,\dee\tau\le Cf_p(k)
	\end{aligned}
\]
for $(x,t)\in Q(R)$ and $k\ge 2$.
This, together with \eqref{eq:5.23} and \eqref{eq:5.27}, implies \eqref{eq:5.22}.
The proof is complete.
$\Box$
\begin{theorem}
\label{Theorem:5.5}
Let $\theta>0$ be fixed.
Let $\Psi$ be a function on $\partial\Omega\times(0,\infty)$ defined by \eqref{eq:2.7}. 
\begin{enumerate}[label={\rm(\arabic*)}]
\item
 	Let $(\phi,\psi)\in L^\infty(\Omega)\times L^\infty(\partial\Omega)$.
	Let $\epsilon>0$ be fixed.
  	For any $T>0$,  
  	\[
  		\sup_{(x,t)\in\Omega\times(0,T)}
		\left|\uHDD{\epsilon}{\delta}{k}(x,t)-\uHDP{\epsilon}{\theta}(x,t)\right|
		=O\left(k^{-1}\right)\quad\mbox{as}\quad k\to\infty\quad\mbox{with}\quad \frac{\delta}{k}=\theta.
  	\]
	Here $\uHDP{\epsilon}{\theta}$ is defined as in \eqref{eq:2.8}.
\item
	Let $(\phi,\psi)\in L^p(\Omega)\times L^p(\partial\Omega)$ for some $p\in[1,\infty)$. 
	Let $\epsilon>0$ be fixed. 
  	For any $R>0$,
  	\[
  		\sup_{(x,t)\in Q(R)}
		\left|\uHDP{\epsilon}{\theta}(x,t)-\uHDi{\epsilon}{0}(x,t)\right|
  		=O(f_p(\theta^{-1}))
		  \quad\mbox{as}\quad\theta\to 0^+,
  	\]
	where $f_p$ is the function defined as in \eqref{eq:5.13}.
\item
	Let $(\phi,\psi)\in L^\infty(\Omega)\times L^\infty(\partial\Omega)$. 
	Let $\epsilon>0$ be fixed.
  	For any $L>0$ and $T>0$,  
  	\[
  		\sup_{(x,t)\in\Omega_L^c\times(0,T)}\left|\uHDP{\epsilon}{\theta}(x,t)-\uHDi{\epsilon}{\psi}(x,t)\right|
	=O\left(\theta^{-1}\right)\quad\mbox{as}\quad\theta\to\infty.
  	\]
\item
 	Let $(\phi,\psi)\in L^\infty(\Omega)\times L^\infty(\partial\Omega)$ for some $p\in[1,\infty)$.
  	For any $L>0$ and $T>0$,
    	\[
    	\sup_{(x,t)\in\Omega_L\times(T,\infty)}\left|\uHDP{\epsilon}{\theta}(x,t)
    	-\uLDP{\theta}(x,t)\right|
	=O\left(\epsilon^\frac{1}{2}\right)\quad \textrm{as}\quad \epsilon\to 0^+.
    	\]
   	 Here
    	$\uLDP{\theta}$ is a function on $\Omega\times(0,\infty)$ defined by 
    	\begin{equation}
    	\label{eq:5.28}
    		\uLDP{\theta}(x,t)
		=\int_{\partial\Omega}P(x'-y',x_N)\Psi(y',t)\,\dee\sigma(y),\quad x=(x',x_N)\in\Omega.
    	\end{equation}
\item
  	Let $(\phi,\psi)\in BC(\Omega)\times L^\infty(\partial\Omega)$.
  	For any compact set $K\subset\Omega$ and $T>0$,
  	\[
  		\lim_{\epsilon\to\infty}\sup_{(x,t)\in K\times(0,T)}|\uHDP{\epsilon}{\theta}(x,t)-\phi(x)|=0.
  	\]
\end{enumerate}
\end{theorem}
Note that, for any $t>0$, $\uLDP{\theta}(\cdot,t)$ is the solution to problem~\eqref{eq:LiD} 
with $\psi$ replaced by $\Psi(t)$, that is, 
\begin{equation}\label{eq:LiDLH}\tag{$\mbox{L}_{\rm D,\Psi}$}
    -\Delta u=0\quad\textrm{in}\quad\Omega,\quad u=\Psi(t)\quad\textrm{on}\quad\partial\Omega.
\end{equation}
{\bf Proof of Theorem~\ref{Theorem:5.5}.}
Let $\theta>0$ be fixed. Let $k>0$ and set $\delta=k\theta$.
We prove assertion~(1). 
Let $\epsilon>0$ be fixed.
It follows from \eqref{eq:1.8} and \eqref{eq:2.9} that
\[
	\begin{aligned}
 	& H(x,y,t)-\HHDiLH(x,y,t)
	\\
 	& =-2\int_0^t \Gamma_{N-1}\left(x'-y',\frac{t-\tau}{\epsilon}+\frac{\tau}{\theta}\right)
	\left\{\partial_{x_N}\Gamma_1\left(x_N+\frac{\tau}{\delta},\frac{t-\tau}{\epsilon}\right)
	-\partial_{x_N}\Gamma_1\left(x_N,\frac{t-\tau}{\epsilon}\right)\right\}\,\dee\tau
	\\
 	& =-2\int_0^t\int_0^{\tau/\delta}\Gamma_{N-1}\left(x'-y',\frac{t-\tau}{\epsilon}+\frac{\tau}{\theta}\right)
 	\partial_{x_N}\partial_\xi\Gamma_1\left(x_N+\xi,\frac{t-\tau}{\epsilon}\right)\,\dee\xi\,\dee\tau\\
    	&=2\epsilon\int_0^{t/\delta}
	\int_{\delta\xi}^t\Gamma_{N-1}\left(x'-y',\frac{t-\tau}{\epsilon}+\frac{\tau}{\theta}\right)
	\partial_\tau\Gamma_1\left(x_N+\xi,\frac{t-\tau}{\epsilon}\right)\,\dee\tau\,\dee\xi
	\\
     	&=-2\epsilon\int_0^{t/\delta}\Gamma_{N-1}
     	\left(x'-y',\frac{t-\delta\xi}{\epsilon}+\frac{\delta\xi}{\theta}\right)
     	\Gamma_1\left(x_N+\xi,\frac{t-\delta\xi}{\epsilon}\right)\dee\xi
     	\\
    	&\qquad
    	-2\epsilon\int_0^{t/\delta}\int_{\delta\xi}^{t}\partial_\tau\Gamma_{N-1}
    	\left(x'-y',\frac{t-\tau}{\epsilon}+\frac{\tau}{\theta}\right)
    	\Gamma_1\left(x_N+\xi,\frac{t-\tau}{\epsilon}\right)\,\dee\tau\,\dee\xi
	\end{aligned}
\]
for $(x,y,t)\in\overline{\Omega}\times\partial\Omega\times(0,\infty)$. 
Then, by property~($\Gamma$1) and \eqref{eq:2.2}, and the fact that
\begin{equation}
\label{eq:5.29}
	\frac{t-\tau}{\epsilon}+\frac{\tau}{\theta}\ge\min\{\epsilon^{-1},\theta^{-1}\}t,\quad \tau\in(0,t),
\end{equation}
for any $T>0$, 
we have 
\[
	\begin{aligned}
 	& \int_{\partial\Omega}|H(x,y,t)-\HHDiLH(x,y,t)|\,\dee\sigma(y)
	\\
	&\le C\epsilon\int_0^{t/\delta}\left(\frac{t-\delta\xi}{\epsilon}\right)^{-\frac{1}{2}}
	\exp\left(-\frac{\xi^2}{4(t-\delta\xi)}\right)\,\dee\xi\\
    	&\qquad
    	+C\epsilon\int_0^{t/\delta}\int_{\delta\xi}^t
    	\left(\frac{t-\tau}{\epsilon}+\frac{\tau}{\theta}\right)^{-1}
    	\left(\frac{t-\tau}{\epsilon}\right)^{-\frac{1}{2}}\exp\left(-\frac{\xi^2}{4(t-\tau)}\right)\,\dee\tau\,\dee\xi
    	\\
    	& \le C\int_0^{t/\delta}
    	\delta^{-\frac{1}{2}}\left(\frac{t}{\delta}-\xi\right)^{-\frac{1}{2}}\,\dee\xi
    	+Ct^{-1}\int_0^{t/\delta}\int_0^{t}(t-\tau)^{-\frac{1}{2}}\,\dee\tau\dee\xi
    	\le C\delta^{-1}t^\frac{1}{2}
	\le Ck^{-1}
	\end{aligned}
\]
for $x\in\Omega$ and $t\in(0,T)$.
This, together with \eqref{eq:1.10}, \eqref{eq:2.8}, and \eqref{eq:3.7}, yields assertion~(1).

We prove assertion~(2). 
Let $\epsilon>0$ be fixed and let $R>0$.
By \eqref{eq:2.2}, \eqref{eq:2.4}, \eqref{eq:2.8}, and \eqref{eq:2.9}, 
we have 
\[
	\begin{aligned}
 	& \left|u_{H_{D,\Psi}}^{\epsilon,\theta}(x,t)-\uHDi{\epsilon}{0}(x,t)\right|
	\\
 	& \le 2\|\psi\|_{L^p(\partial\Omega)}
	\left(\int_0^{t/2}+\int_{t/2}^t\right)\left\|\Gamma_{N-1}\left(x'-\cdot,\frac{t-\tau}{\epsilon}
	+\frac{\tau}{\theta}\right)\right\|_{L^{\frac{p}{p-1}}({\mathbb R}^{N-1})}
	\left|\partial_{x_N}\Gamma_1\left(x_N,\frac{t-\tau}{\epsilon}\right)\right|\,\dee\tau
	\\
 	& \le C\left(\frac{t}{2\epsilon}\right)^{-1}\exp\left(-\frac{\epsilon x_N^2}{4t}\right)
	\int_0^{t/2}\left(\frac{t-\tau}{\epsilon}+\frac{\tau}{\theta}\right)^{-\frac{N-1}{2p}}\,\dee\tau\\
	& \qquad
	+C\int_{t/2}^t\left(\frac{\tau}{\theta}\right)^{-\frac{N-1}{2p}}
	\left(\frac{t-\tau}{\epsilon}\right)^{-\frac{3}{2}}x_N
	\exp\left(-\frac{\epsilon x_N^2}{4(t-\tau)}\right)\,\dee\tau
	\\
 	& \le Ct^{-1}\exp\left(-\frac{\epsilon x_N^2}{4t}\right)
	\int_0^{t/2}\left(\frac{t-\tau}{\epsilon}+\frac{\tau}{\theta}\right)^{-\frac{N-1}{2p}}\,\dee\tau
 	+C\theta^{\frac{N-1}{2p}}t^{-\frac{N-1}{2p}}
	\int_0^\infty\zeta^{-\frac{3}{2}}\exp\left(-\frac{\epsilon}{4\zeta}\right)\,\dee\zeta
\end{aligned}
\]
for $(x,t)\in\Omega\times(0,\infty)$. 
Then we argue as in \eqref{eq:5.7} to obtain 
\[
	\left|\uHDP{\epsilon}{\theta}(x,t)-\uHDi{\epsilon}{0}(x,t)\right| \le Cf_p(\theta^{-1})
\]
for $(x,t)\in{Q(R)}$ and $\theta\in(0,1/2)$, which implies assertion~(2). 
Furthermore, 
since
\[
	\Gamma_{N-1}\left(x'-y',\frac{t-\tau}{\epsilon}+\frac{\tau}{\theta}\right)
	-\Gamma_{N-1}\left(x'-y',\frac{t-\tau}{\epsilon}\right)
	=\int_0^{\tau/\theta}\partial_\xi\Gamma_{N-1}\left(x'-y',\frac{t-\tau}{\epsilon}+\xi\right)\,\dee\xi,
\]
by property~($\Gamma$1) and \eqref{eq:2.2},
for any $L>0$ and $T>0$, 
we obtain
\[
	\begin{aligned}
 	& \int_{\partial\Omega}\int_0^t
	\left|\Gamma_{N-1}\left(x'-y',\frac{t-\tau}{\epsilon}+\frac{\tau}{\theta}\right)
	-\Gamma_{N-1}\left(x'-y',\frac{t-\tau}{\epsilon}\right)\right|
	\left|\partial_{x_N}\Gamma_1\left(x_N,\frac{t-\tau}{\epsilon}\right)\right|\,\dee\tau\,\dee\sigma(y)
	\\
 	& \le C\int_0^t\int_0^{\tau/\theta}
	\left(\frac{t-\tau}{\epsilon}+\xi\right)^{-1}\left(\frac{t-\tau}{\epsilon}\right)^{-1}
 	\exp\left(-\frac{\epsilon x_N^2}{8(t-\tau)}\right)\,\dee\xi\,\dee\tau\\
 	& \le \frac{Ct}{\theta}\int_0^t \left(\frac{t-\tau}{\epsilon}\right)^{-2}
 	\exp\left(-\frac{\epsilon L^2}{8(t-\tau)}\right)\le \frac{CT}{\theta}
	\end{aligned}
\]
for $x\in\Omega_L^c$ and $t\in(0,T)$. 
This yields assertion~(3). 

We prove assertion~(4). 
It follows from Lemma~\ref{Lemma:4.1}~(4), \eqref{eq:2.7}, \eqref{eq:2.11}, and \eqref{eq:5.28} that 
\begin{equation}
\label{eq:5.30}
	\begin{aligned}
 	\uLDP{\theta}(x,t)
    	&=\int_{\partial\Omega}P_N(x'-z',x_N)
	\left(\int_{\partial\Omega}\Gamma_{N-1}\left(z'-y',\frac{t}{\theta}\right)\psi(y)
	\,\dee\sigma(y)\right)\,\dee\sigma(z)
	\\
    	&=-2\int_{\partial\Omega}\int_0^\infty \Gamma_{N-1}(x'-z',\tau)\partial_{x_N}\Gamma_1(x_N,\tau)
	\\
    	& \qquad\times 
    	\left(\int_{\partial\Omega}\Gamma_{N-1}\left(z'-y',\frac{t}{\theta}\right)\psi(y)\,\dee \sigma(y)\right)
	\,\dee\tau\,\dee\sigma(z)
	\\
    	& =-2\int_{\partial\Omega}\int_0^\infty\Gamma_{N-1}\left(x'-y',\tau+\frac{t}{\theta}\right)
	\partial_{x_N}\Gamma_1(x_N,\tau)\psi(y)\,\dee\tau\,\dee\sigma(y),
    	\quad x\in\Omega.
  	\end{aligned}
\end{equation}
Then it follows from \eqref{eq:2.8} that
\begin{equation}
\label{eq:5.31}
  	\uHDP{\epsilon}{\theta}(x,t)-\uLDP{\theta}(x,t)
	=\uHDi{\epsilon}{0}(x,t)+\int_{\partial\Omega}\left(J_1(x,y,t)+J_2(x,y,t)\right)\psi(y)\,\dee\sigma(y)
\end{equation}
for $x\in\Omega$ and $t>0$,
where
\[
	\begin{aligned}
  	J_1(x,y,t)
	&:=-2\int_0^{t/\epsilon}\left(\Gamma_{N-1}\left(x'-y',\tau+\frac{t-\epsilon\tau}{\theta}\right)
	-\Gamma_{N-1}\left(x'-y',\tau+\frac{t}{\theta}\right)\right)\partial_{x_N}\Gamma_1(x_N,\tau)\,\dee\tau,
	\\
  	J_2(x,y,t)
	&=2\int_{t/\epsilon}^\infty \Gamma_{N-1}\left(x'-y',\tau+\frac{t}{\theta}\right)
	\partial_{x_N}\Gamma_1(x_N,\tau)\,\dee\tau.
	\end{aligned}
\]
Since
\[
	J_1(x,y,t)
	=-2\int_0^{t/\epsilon}\int_0^{\epsilon\tau}\partial_{\xi}\Gamma_{N-1}
	\left(x'-y',\tau+\frac{t-\xi}{\theta}\right)\partial_{x_N}\Gamma_1(x_N,\tau)\,\dee\xi\,\dee\tau,
\]
by \eqref{eq:2.2} and \eqref{eq:5.29}, we estimate
\[
	\begin{aligned}
  	\left|\int_{\partial\Omega}J_1(x,y,t)\psi(y)\,\dee\sigma(y)\right|
	&\le C\|\psi\|_{L^\infty(\partial\Omega)}
  	\int_0^{t/\epsilon}\int_0^{\epsilon\tau}\left(\tau+\frac{t-\xi}{\theta}\right)^{-1}
	x_N\tau^{-\frac{3}{2}}\exp\left(-\frac{x_N^2}{4\tau}\right)\,\dee\xi\,\dee\tau
	\\
  	&\le C\epsilon \int_0^{t/\epsilon}\left(\tau+\frac{t-\epsilon\tau}{\theta}\right)^{-1}
	x_N\tau^{-\frac{1}{2}}\exp\left(-\frac{x_N^2}{4\tau}\right)\,\dee\tau\\
  	&\le C\epsilon x_N^2t^{-1}\int_0^{t/(\epsilon x_N^2)}
	\eta^{-\frac{1}{2}}\exp\left(-\frac{1}{4\eta}\right)\,\dee\eta
	\le C\epsilon^{\frac{1}{2}}t^{-\frac{1}{2}}x_N
\end{aligned}
\]
for $\epsilon\in(0,\theta)$ and $(x,t)\in\Omega\times(0,\infty)$. 
Furthermore, similarly to the proof of Theorem~\ref{Theorem:5.3}~(1), we see that
\[
	\left|\int_{\partial\Omega}J_2(x,y,t)\psi(y)\,\dee\sigma(y)\right|
	\le C\epsilon^{\frac{1}{2}}t^{-\frac{1}{2}}x_N
\]
for $(x,t)\in\Omega\times(0,\infty)$. 
These, together with \eqref{eq:2.6} and \eqref{eq:5.31}, implies assertion~(4).

We finally prove assertion~(5). 
Let $K$ be compact in $\Omega$ and $T>0$. 
By \eqref{eq:2.5} and \eqref{eq:2.8}, we have
\[
	\begin{aligned}
 	& \limsup_{\epsilon\to\infty}\sup_{(x,t)\in K\times(0,T)}|\uHDP{\epsilon}{\theta}(x,t)-\phi(x)|
	\\
	& \le \limsup_{\epsilon\to\infty}\sup_{(x,t)\in K\times(0,T)}|\uHDP{\epsilon}{\theta}(x,t)-\uHDi{\epsilon}{0}(x,t)|
	+\limsup_{\epsilon\to\infty}\sup_{(x,t)\in K\times(0,T)}|\uHDi{\epsilon}{0}(x,t)-\phi(x)|
	\\
 	& \le \|\psi\|_{L^\infty(\partial\Omega)}\limsup_{\epsilon\to\infty}
	\sup_{(x,t)\in K\times(0,T)}
	\frac{1}{\epsilon}\int_{\partial\Omega}\int_0^t \tilde{H}(x,y,t)\,\dee\sigma(y).
	\end{aligned}
\]
It follows from property~($\Gamma$1) and \eqref{eq:2.9} that
\[
	\begin{aligned}
	\frac{1}{\epsilon}\int_{\partial\Omega}\int_0^t \tilde{H}(x,y,t)\,\dee\sigma(y)
 	& =-\frac{2}{\epsilon}\int_0^t\partial_{x_N}\Gamma_1\left(x_N,\frac{t-\tau}{\epsilon}\right)\,\dee\tau
	\\
 	& =-2\int_0^{t/\epsilon}\partial_{x_N}\Gamma_1\left(x_N,\tau\right)\,\dee\tau
	\le C\int_0^{t/\epsilon} \tau^{-1}\exp\left(-\frac{\tilde L^2}{8\tau}\right)\,\dee\tau
	\le C\frac{T}{\epsilon\tilde L^2}
	\end{aligned}
\]
for $(x,t)\in K\times(0,T)$, 
where $\tilde L:=\mbox{dist}(K,\partial\Omega)>0$. 
These imply assertion~(5).
Thus, the proof of Theorem~\ref{Theorem:5.5} is complete.
$\Box$
\begin{Remark}
\label{Remark:5.3}
{\rm (1)}
The convergence rate in {\rm Theorem~\ref{Theorem:5.5}\,(1)} is optimal.
Indeed, consider the case where $\phi=1$ on $\Omega$ and $\psi=0$ on $\partial\Omega$. 
Let $\theta>0$ be fixed. Let $k>0$ and set $\delta=k\theta$.
Since $\delta\to\infty$ as $k\to\infty$, 
applying the same argument as in Remark~{\rm\ref{Remark:5.2}~(1)}, 
we have
\[
	\begin{aligned}
 	& \delta\left(\uHDD{\epsilon}{\delta}{k}(x,t)-\uHDi{\epsilon}{\Psi}(x,t)\right)
 	=\delta\left(\uHDD{\epsilon}{\delta}{k}(x,t)-\uHDi{\epsilon}{0}(x,t)\right)\\
 	& =\int_\Omega H(x,y,t)\,\dee y
 	\to 2\int_0^t \left(\frac{4\pi(t-\tau)}{\epsilon}\right)^{-\frac{1}{2}}
	\exp\left(-\frac{\epsilon x_N^2}{4(t-\tau)}\right)\,\dee\tau>0
	\end{aligned}
\]
as $k\to\infty$ for $(x,t)\in\Omega\times(0,\infty)$. 
This shows that the convergence rate in Theorem~{\rm \ref{Theorem:5.5}~(1)} is optimal.
\medskip

\noindent
{\rm (2)} 
The convergence rate in Theorem~{\rm\ref{Theorem:5.5}\,(3)} is optimal.
To see this, consider the case where $\phi=0$ on $\Omega$ 
and $\psi=\chi_{\partial\Omega\setminus B'_2(0)}$ on $\partial\Omega$.
Then, by \eqref{eq:2.3}, \eqref{eq:2.8}, \eqref{eq:2.9}, and \eqref{eq:3.39}, 
we have
\[
	\begin{aligned}
 	& \theta(\uHDP{\epsilon}{\theta}(x,t)-\uHDi{\epsilon}{\psi}(x,t))
	\\
 	& =-\frac{2\theta}{\epsilon}\int_{\partial\Omega\setminus B'_2(0)}
	\int_0^t \left\{\Gamma_{N-1}\left(x'-y',\frac{t-\tau}{\epsilon}+\frac{\tau}{\theta}\right)
 	-\Gamma_{N-1}\left(x'-y',\frac{t-\tau}{\epsilon}\right)\right\}\\
  	& \qquad\quad
	\times \partial_{x_N}\Gamma_1\left(x_N,\frac{t-\tau}{\epsilon}\right)\,\dee\tau\,\dee\sigma(y)\\
  	& \to -\frac{2}{\epsilon}\int_{\partial\Omega\setminus B'_2(0)}
	\int_0^t \tau (\partial_t\Gamma_{N-1})\left(x'-y',\frac{t-\tau}{\epsilon}\right)
  	\partial_{x_N}\Gamma_1\left(x_N,\frac{t-\tau}{\epsilon}\right)\,\dee\tau\,\dee\sigma(y)>0
\end{aligned}
\]
as $\theta\to\infty$ for all $x\in B_1(0)$ and all sufficiently small $t>0$. 
This shows that the convergence rate in Theorem~{\rm\ref{Theorem:5.5}\,(3)} is optimal.
\medskip

\noindent
{\rm (3)} 
By \eqref{eq:2.6} and \eqref{eq:5.30},
the same argument as in Remark~{\rm\ref{Remark:5.2}~(2)} shows that
the convergence rate in Theorem~{\rm\ref{Theorem:5.5}~(4)} is optimal.
\medskip

\noindent
{\rm (4)} 
Let $\psi\in L^p(\partial\Omega)$, where $1\le p<\infty$. 
Then the diffusion limits of $\uLDP{\theta}(x,t)$ as $\theta\to\infty$ and $\theta\to 0^+$ 
follow from \eqref{eq:5.30}. 
Indeed, since
\[
	\begin{aligned}
    	& \uLDP{\theta}(x,t)-\uLDi{\psi}(x)
	\\
    	& =\int_{\partial\Omega}P_N(x'-z',x_N)\left(\int_{\partial\Omega}
	\Gamma_{N-1}\left(z'-y',\frac{t}{\theta}\right)\psi(y)\,\dee\sigma(y)-\psi(z')\right)\,\dee\sigma(z)
	\end{aligned}
\]
for $(x,t)\in\Omega\times(0,\infty)$ {\rm({\it see} \eqref{eq:5.30})}, 
for any $T>0$, 
by \eqref{eq:2.12} and property~{\rm($\Gamma$3)}, we obtain
\[
	\begin{aligned}
 	& \sup_{(x,t)\in\Omega_L^c\times(0,T)}\left|\uLDP{\theta}(x,t)-\uLDi{\psi}(x)\right|
	\\
 	& \le C L^{-(N-1)\left(1-\frac{1}{p}\right)}
	\left\|\int_{\partial\Omega}\Gamma_{N-1}\left(\cdot-y',\frac{t}{\theta}\right)\psi(y)
	\,\dee\sigma(y)-\psi\right\|_{L^p(\partial\Omega)}
	\to 0\quad \mbox{as} \quad \theta\to\infty
\end{aligned}
\]
uniformly for $t\in(0,T)$. 
In contrast, by \eqref{eq:2.2} and \eqref{eq:2.12}, 
we have
\[
	\begin{aligned}
 	\sup_{t\in(T,\infty)}\|\uLDP{\theta}(t)\|_{L^\infty(\Omega)}
    	& \le C\sup_{t\in(T,\infty)}\left\|\,\int_{\partial\Omega}
	\Gamma_{N-1}\left(\cdot-y',\frac{t}{\theta}\right)\psi(y)\,\dee\sigma(y)\,\right\|_{L^\infty(\partial\Omega)}
	\\
    	& \le C\left(\frac{T}{\theta}\right)^{-\frac{N-1}{2}\left(1-\frac{1}{p}\right)}\|\psi\|_{L^p(\partial\Omega)}
    	\le C\,\theta^{\frac{N-1}{2}\left(1-\frac{1}{p}\right)}\to 0
    	\quad\mbox{as}\quad\theta\to 0^+.
\end{aligned}
\]
 \end{Remark}
 
A diagram summarizing the diffusion limits established in this section, 
together with Theorem~\ref{Theorem:1.3}\,(3) and the results of Section~\ref{section:2}, is given below. 
We use the symbol $\Longrightarrow$ to denote convergence with the optimal rate.
\medskip
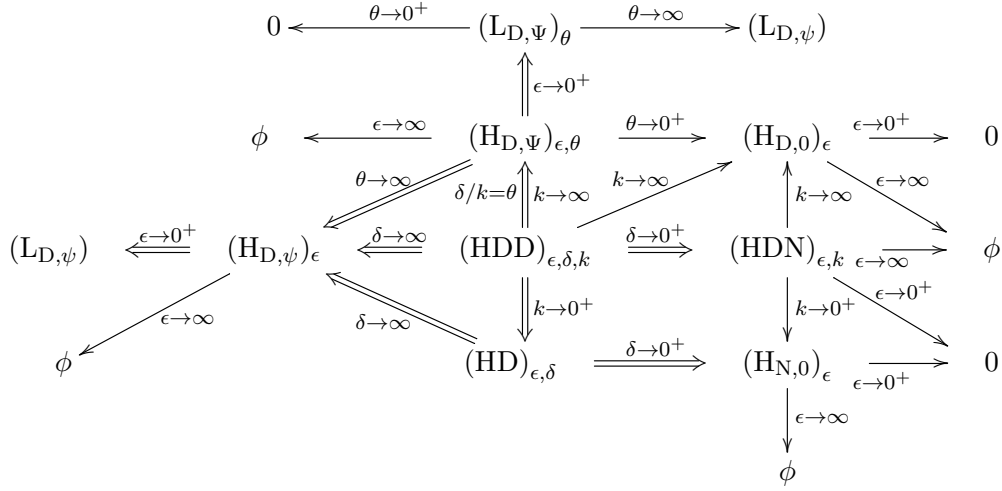
\[
	\xymatrix{ 
	& 0 & \eqref{eq:LiDLH}_\theta \ar[l]_{\theta\to 0^+} \ar[r]^{\theta\to\infty} & \eqref{eq:LiD}& \\
	& \phi \quad & \quad (\mbox{H}_{\rm D,\Psi})_{\epsilon,\theta} \ar@{=>}[u]_{\epsilon\to 0^+}  \ar[l]_{\epsilon\to \infty} 
	\ar@{=>}[ld]_{\theta\to\infty} \ar[r]^{\theta\to 0^+}\quad& \quad (\mbox{H}_{{\rm D,0}})_\epsilon \ar[rd]^{\epsilon\to\infty} \ar[r]^{\epsilon\to 0^+}\quad & \quad 0\\
	\eqref{eq:LiD} \quad & \quad (\mbox{H}_{{\rm D,\psi}})_\epsilon \ar@{=>}[l]_{\epsilon\to 0^+} \ar[ld]^{\epsilon\to \infty} \quad & \quad 
	\eqref{eq:HDD}_{\epsilon,\delta,k} \ar@{=>}[u]_{k\to\infty}^{\delta/k=\theta} \ar@{=>}[l]_{\delta\to \infty} \ar@{=>}[d]^{k\to 0^+} \ar@{=>}[r]^{\delta\to 0^+} \ar[ur]^{k\to\infty} \quad 
	& \quad \eqref{eq:HDN}_{\epsilon,k} \ar[d]^{k\to 0^+} \ar[u]_{k\to \infty} \ar[r]_{\epsilon\to\infty} \ar[rd]^{\epsilon\to 0^+} \quad & \quad \phi \\
	\phi &  & \eqref{eq:HD}_{\epsilon,\delta} \ar@{=>}[ul]^{\delta\to\infty} \ar@{=>}[r]^{\delta\to 0^+} \quad & 
	\quad \eqref{eq:HhN}_{\epsilon} \ar[r]_{\epsilon\to 0^+} \ar[d]^{\epsilon\to\infty} \quad & \quad 0\\
 	& & & \phi & 
	}
\]
\medskip

\noindent
{\bf Acknowledgment.}
K. I. and T. K. were supported in part by JSPS KAKENHI Grant Number 25H00591.
T. K. was supported in part by JSPS KAKENHI Grant Number 22KK0035.
S. K. was supported by JSPS KAKENHI Grant Number 23KJ0645.
\medskip

\noindent
{\bf Data availability.}
No data was generated or analyzed as part of the writing of this paper.
\medskip

\noindent
{\bf  Conflict of interest.}
The authors state no conflict of interest. 

%
%
%

\end{document}